\documentclass[12pt,oneside,reqno]{amsart}
\usepackage{color}
\usepackage{graphicx}
\usepackage{amssymb}
\usepackage{amsthm}
\usepackage{amsmath}
\usepackage[margin=1in]{geometry}
\usepackage{mathabx}
\usepackage{scalerel}
\usepackage{mathrsfs}
\usepackage{tikz}

\newenvironment{rcases}
  {\left.\begin{aligned}}
  {\end{aligned}\right\rbrace}

\newcommand\norm[1]{\left\lVert#1\right\rVert}

\newtheorem{theorem}{Theorem}
\newtheorem{definition}{Definition}
\newtheorem{lemma}{Lemma}
\newtheorem{corollary}{Corollary}
\newtheorem{remark}{Remark}
\newtheorem{proposition}{Proposition}

\newcommand{\bea}{\begin{eqnarray*}}
\newcommand{\eea}{\end{eqnarray*}}                                                                                                                                                                                               
\newcommand{\ben}{\begin{eqnarray}}
\newcommand{\een}{\end{eqnarray}}
\newcommand{\beq}{\begin{equation}}
\newcommand{\eeq}{\end{equation}}

\renewcommand{\hat}[1]{\widehat{#1}}

\allowdisplaybreaks[4]

\begin{document}

\title{Long time behavior of 2d water waves with point vortices}
\email{qingtang@umich.edu}
\address{University of Michigan, Ann Arbor, 
Department of Mathematics
}

\author{Qingtang Su}

\begin{abstract}
In this paper, we study the motion of the two dimensional inviscid incompressible, infinite depth water waves with point vortices in the fluid. We show that Taylor sign condition $-\frac{\partial P}{\partial \boldmath{n}}\geq 0$ can fail if the point vortices are sufficient close to the free boundary, so the water waves could be subject to the Taylor instability. Assuming the Taylor sign condition, we prove that the water wave system is locally wellposed in Sobolev spaces. Moreover, we show that if the water waves is symmetric with a certain symmetric vortex pair traveling downward initially, then the free interface remains smooth for a long time, and for initial data of size $\epsilon\ll 1$, the lifespan is at least $O(\epsilon^{-2})$. 

\end{abstract}

\maketitle

\section{Introduction}
In this paper we investigate the two dimensional inviscid incompressible infinite depth water wave system with point vortices in the fluid. This system arises in the study of submerged bodies in a fluid (see for example \cite{chang2001vortex},\cite{dalrymple2006numerical} and the references therein). It's also believed to give some  insight into the problem of turbulence (\cite{marchioro2012mathematical}, chap 4, \S 4.6).  In idealized situation, such water waves are described by assuming the vorticity is a Dirac delta measure, i.e., the vorticity is given by $\omega(\cdot, t)=\sum_{j=1}^N \lambda_j \delta_{z_j(t)}(\cdot)$, where each $z_j\in \Omega(t)$, $\lambda_j\in\mathbb{R}$ represent the position and the strength of the $j$-th point vortex, respectively. It's well-known that each point vortex $z_j$ generates a velocity field $\frac{\lambda_j }{2\pi}\frac{i}{\overline{z-z_j(t)}}$ , which is purely rotational.  We assume $z_j(0)\neq z_k(0)$, for $j\neq k$. Assume that the fluid is incompressible and inviscid and neglect surface tension. Let $\Omega(t)$ be the fluid region, with free boundary $\Sigma(t)$, such that $\Sigma(t)$ separate the fluid region with density one from the air with density zero. Let $v$ be the fluid velocity, and $P$ be the fluid pressure. Normalize the gravity to be $(0,-1)$. Then the motion of the fluid is described by
\begin{equation}\label{vortex_model}
\begin{cases}
\begin{rcases}
&v_t+v\cdot \nabla v=-\nabla P-(0,1)\\
&div~v=0\\
&curl~v=\omega=\sum_{j=1}^N \lambda_j \delta_{z_j(t)}\\
\end{rcases}
\quad \Omega(t)\\
\frac{d}{dt}z_j(t)=(v-\frac{\lambda_j i}{2\pi}\frac{1}{\overline{z-z_j(t)}})\Big |_{z=z_j}\quad \\
z_j(t)\in \Omega(t),\quad j=1,...,N\\
P\Big|_{\Sigma(t)}\equiv 0\quad \quad \quad \quad \quad \quad \quad \quad \quad \quad \quad \quad \\
(1,v) \text{ is tangent to the free surface } (t, \Sigma(t)).
\end{cases}
\end{equation}
Here, $z=x+iy, z_j(t)=x_j(t)+iy_j(t)$. We've identified $\mathbb{R}^2$ with $\mathbb{C}$, so a point $(x,y)$ is identified as $x+iy$. Given $w\in \mathbb{C}$, $\bar{w}$ represents the complex conjugate of $w$.

Formally, this system is obtained by neglecting the self-interaction of the point vortices:  intuitively, if we pretend that the velocity field $v$ is well-defined at $z=z_j(t)$, then the motion of the fluid particle $z_j(t)$ is given by 
$$\frac{d}{dt}z_j(t)=v(z_j(t),t)=(v-\frac{\lambda_j i}{2\pi}\frac{1}{\overline{z-z_j(t)}})\Big |_{z=z_j}+\frac{\lambda_j i }{2\pi}\frac{1}{\overline{z-z_j(t)}})\Big |_{z=z_j}.$$
We assume that the only singularities of $v$ are at the point vortices, so  $v-\frac{\lambda_j i}{2\pi}\frac{1}{\overline{z-z_j(t)}}$ is smooth, while $\frac{\lambda_j i }{2\pi}\frac{1}{\overline{z-z_j(t)}}$ is not defined at $z=z_j$. However, since the velocity field $\frac{\lambda_j i }{2\pi}\frac{1}{\overline{z-z_j(t)}}\Big |_{z\neq z_j}$ is purely rotational around $z_j(t)$, it won't move the center $z_j(t)$ at all, which means
$$\frac{d}{dt}z_j(t)=(v-\frac{\lambda_j i}{2\pi}\frac{1}{\overline{z-z_j(t)}})\Big |_{z=z_j}.$$
For a rigorous justification of this derivation, see \cite{marchioro2012mathematical} (Theorem 4.1, 4.2, chapter 4) for the fixed boundary case. 

When $N=0$, i.e., without the presence of point vortices, the system (\ref{vortex_model}) has attracted a lot of attention from both mathematics and physics communities, and there is a enormous literature on this problem. Rigorous mathematical analysis for water waves with point vortices is still missing.

\subsection{Historical results}
\subsubsection{Numerical results}
There have been a lot of numerical study of the system (\ref{vortex_model}). See for example \cite{curtis2017vortex}, \cite{hill1975numerical}, \cite{fish1991vortex}, \cite{marcus1989interaction}, \cite{telste1989potential}, \cite{willmarth1989vortex} and the references therein for some numerical investigations.

\subsubsection{Rigorous mathematical analysis: Irrotational case}
For the rigorous mathematical analysis of (\ref{vortex_model}) in Sobolev spaces without the presence of point vortex, i.e., $N=0$,  system (\ref{vortex_model}) reduces to inviscid, incompressible, irrotational 2d water waves, i.e.,
\begin{equation}\label{system}
\begin{cases}
\begin{rcases}
&v_t+v\cdot \nabla v=-\nabla P-(0,1)\\
&div~v=0,~~curl~v=0\\
\end{rcases}
\quad \quad \quad \Omega(t)\\
P\equiv 0\quad \quad  \quad\quad \quad \quad\quad \quad \quad\quad \quad \quad\quad~~\quad\Sigma(t)\\
(1,v) \text{ is tangent to the free surface } (t, \Sigma(t)).
\end{cases}
\end{equation}
The water wave system (\ref{system}) has been intensively studied, for early works, see Newton \cite{newton}, Stokes\cite{Stokes}, Levi-Civita\cite{Levi-Civita}, and G.I. Taylor \cite{Taylor}. Nalimov \cite{Nalimov}, Yosihara\cite{Yosihara} and Craig \cite{Craig} proved local well-posedness for 2d water waves equation (\ref{system}) for small initial data. In S. Wu's breakthrough works \cite{Wu1997}\cite{Wu1999}, she proved that for $n\geq 2$ the important strong Taylor sign condition
\begin{equation}
-\frac{\partial P}{\partial \boldmath{n}}\Big|_{\Sigma(t)}\geq c_0>0
\end{equation}
always holds for the infinite depth water wave system (\ref{system}), as long as the interface is non-self-intersecting and smooth, and she proved that the initial value problem for (\ref{system}) is locally well-posed in $H^s(\mathbb{R}), s\geq 4$ without smallness assumption. Since then, a lot of interesting local well-posedness results were obtaind, see for example \cite{alazard2014cauchy}, \cite{ambrose2005zero}, \cite{christodoulou2000motion}, \cite{coutand2007well}, \cite{iguchi2001well}, \cite{lannes2005well}, \cite{lindblad2005well}, \cite{ogawa2002free}, \cite{shatah2006geometry}, \cite{zhang2008free}.
 Recently, almost global and global well-posedness for water waves (\ref{system}) under irrotational assumption have also been proved, see \cite{Wu2009}, \cite{Wu2011}, \cite{germain2012global}, \cite{Ionescu2015}, \cite{AlazardDelort},  and see also \cite{HunterTataruIfrim1} and \cite{HunterTataruIfrim2}.  More recently, there are strong interests in understanding the singularities of water waves, see for example  \cite{kinsey2018priori}, \cite{Wu2}, \cite{Wu3}, \cite{Wu4}. For the formation of splash singularities, see for example \cite{castro2012finite}\cite{castro2013finite}\cite{coutand2014finite} \cite{coutand2016impossibility} . Note that most of the aforementioned works are done in irrotational setting, both for mathematical convenience and physical considerations. 
 
 \subsubsection{Rigorous mathematical analysis: rotational case}
 Much less rigorous mathematical analysis have been done for rotational water waves. For vorticity $\omega$ that is a smooth function,  Iguchi, Tanaka, and Tani \cite{iguchi1999free} proved the local wellposedness of the free boundary problem for an incompressible ideal fluid in two space dimensions without surface tension. Ogawa and Tani \cite{ogawa2002free} generalized Iguchi, Tanaka, and Tani's work to the case with surface tension. In \cite{ogawa2003incompressible}, Ogawa and Tani generalized the wellposedness result to the finite depth case.  In \cite{christodoulou2000motion}, Chritodoulou and Lindblad obtained a priori energy estimates of $n$ dimensional incompressible fluid, without assuming irrotationality condition in a bounded domain without gravity. In particular, the authors introduce a geometrical point of view, estimating quantities such as the second fundamental form and the velocity of the free surface. For the same problem as in \cite{christodoulou2000motion}, H. Lindblad proved local wellposedness in \cite{lindblad2005well}. The local wellposedness of rotational 3d infinite depth, inviscid incompressible water waves is proved by P. Zhang and Z. Zhang \cite{zhang2008free}. All the aforementioned existence results for rotational water waves are locally in time, and under the assumption that the strong Taylor sign condition holds. 
 
 Regarding the long time behavior, Ifrim and Tataru\cite{ifrim2015two} prove cubic lifespan for 2d inviscid incompressible infinite depth water waves with constant vorticity. Assume the volocity field is $(u,v)$. Assume the vorticity is $v_x-u_y\equiv c$, where $c$ is a constant, replacing the velocity field $(u,v)$ by $(u+cy,v)$, the problem is reduced to irrotational incompressible water waves, and from which the long time existence is proved. In \cite{bieri2017motion}, Bieri, Miao, Shahshahani and Wu prove cubic lifespan for the motion of a self-gravitating incompressible fluid in a bounded domain with free boundary for small initial data for the irrotational case and the case of constant vorticity. The case of constant vorticity is reduced to the irrotational one by working in a rotating framework with constant angular velocity. 
 
 \vspace*{2ex}
 
For water waves with point vortices, there is no obvious transformation to reduce the problem to the irrotational one, so we need some new ideas to study such cases.

\subsection{Governing equation for free boundary}
It's easy to see that the system (\ref{vortex_model}) is completely determined by the free surface $\Sigma(t)$, the velocity and the acceleration along the free surface, and the position of the point vortices. 
\subsubsection{Lagrangian formulation} We parametrize the free surface by Lagrangian coordinates, i.e., let $\alpha$ be such that 
\begin{equation}
z_t(\alpha,t)=v(z(\alpha,t),t).
\end{equation}
We identify $\mathbb{R}^2$ with the complex plane. With this identification, a point $(x,y)$ is the same as $x+iy$. Since $P\equiv 0$ along $\Sigma(t)$, we can write $\nabla P$ as $-iaz_{\alpha}$, where $a=-\frac{\partial P}{\partial\boldmath{n}}\frac{1}{|z_{\alpha}|}$ is real valued. So the momentum equation along the free boundary becomes
\begin{equation}
z_{tt}-iaz_{\alpha}=-i.
\end{equation}
Note that the second and the third equation in (\ref{vortex_model}) imply that  $\overline{v-\sum_{j=1}^N \frac{\lambda_j i}{2\pi(\overline{z-z_j(t)}})}=\bar{v}+\sum_{j=1}^N \frac{\lambda_j i}{2\pi(z-z_j(t))}$ is holomorphic in $\Omega(t)$ with the value at the boundary $\Sigma(t)~$ given by $\bar{z}_t+\sum_{j=1}^N \frac{\lambda_j i}{2\pi (z(\alpha,t)-z_j(t))}$ . Assume that $\bar{z}_t+\sum_{j=1}^N \frac{\lambda_j i}{2\pi (z(\alpha,t)-z_j(t))}\in L^2(\mathbb{R})$. We know that $\bar{z}_t+\sum_{j=1}^N \frac{\lambda_j i}{2\pi (z(\alpha,t)-z_j(t))}$ is the boundary value of a holomorphic function in $\Omega(t)$ if and only if 
\begin{equation}
(I-\mathfrak{H})\Big(\bar{z}_t+\sum_{j=1}^N \frac{\lambda_j i}{2\pi (z(\alpha,t)-z_j(t))}\Big)=0,
\end{equation}
where $\mathfrak{H}$ is the Hilbert transform associated with the curve $z(\alpha,t)$, i.e.,
\begin{equation}
\mathfrak{H}f(\alpha):=\frac{1}{\pi i}p.v.\int_{-\infty}^{\infty}\frac{z_{\beta}}{z(\alpha,t)-z(\beta,t)}f(\beta)d\beta.
\end{equation}
So the system (\ref{vortex_model}) is reduced to a system of equations for the free boundary coupled with the dynamic equation for the motion of the point vortices:
\begin{equation}\label{vortex_boundary}
\begin{cases}
z_{tt}-iaz_{\alpha}=-i\\
\frac{d}{dt}z_j(t)=(v-\frac{\lambda_j i}{2\pi(\overline{z-z_j})})\Big |_{z=z_j}\\
(I-\mathfrak{H})\Big(\bar{z}_t+\sum_{j=1}^N \frac{\lambda_j i}{2\pi(z(\alpha,t)-z_j(t))}\Big)=0.
\end{cases}
\end{equation}
Note that $v$ can be recovered from (\ref{vortex_boundary}). Indeed, we have 
\begin{equation}\label{velocity}
\bar{v}(z)+\sum_{j=1}^N\frac{\lambda_j i}{2\pi(z-z_j(t))}=\frac{1}{2\pi i}\int \frac{z_{\beta}}{z-z(\beta)}\Big(\bar{z}_t(\beta)+\sum_{j=1}^N \frac{\lambda_j i}{2\pi(z(\beta)-z_j(t)})\Big)d\beta.
\end{equation}
So the system (\ref{vortex_model}) and the system (\ref{vortex_boundary}) are equivalent. 

\subsubsection{The Riemann mapping formulation and the Taylor sign condition.} Let $\bold{n}$ be the outward unit normal to the fluid-air interface $\Sigma(t)$.  The quantity $-\frac{\partial P}{\partial \bold{n}}\Big|_{\Sigma(t)}$ plays an important role in the study of water waves.
\begin{definition}
(The Taylor sign condition and the strong Taylor sign condition)
\begin{itemize}
\item [(1)] If $-\frac{\partial P}{\partial \bold{n}}\Big|_{\Sigma(t)}\geq 0$ pointwisely, then we say the Taylor sign condition holds.

\item [(2)] If there is some positive constant $c_0$ such that $-\frac{\partial P}{\partial \bold{n}}\Big|_{\Sigma(t)}\geq c_0>0$ pointwisely, then we say the strong Taylor sign condition holds.
\end{itemize}
\end{definition}
It is well known that when surface tension is neglected and the Taylor sign condition fails, the motion of the water waves can be subject to the Taylor instability \cite{beale1993growth},\cite{birkhoff1962helmholtz},\cite{taylor1950instability},\cite{ebin1987equations},\cite{wu2016wellposedness}. For irrotational incompressible infinite depth water waves without surface tension, S. Wu \cite{Wu1997} \cite{Wu1999} shows that the strong Taylor sign condition always holds provided that the interface is non self-intersecting and smooth. 

For rotational water waves, by constructing explicit examples, we'll show that Taylor sign condition can fail if the point vortice are sufficient close to the interface. We'll also give a criterion for the Taylor sign condition to hold. To calculate the important quantity $-\frac{\partial P}{\partial \bold{n}}$, we use the Riemann mapping formulation of the system (\ref{vortex_boundary}), which we are to describe. 

Let $\Phi(\cdot,t):\Omega(t)\rightarrow \mathbb{P}_-$ be the Riemann mapping such that $\Phi_z\rightarrow 1$ as $z\rightarrow \infty$. Let $h(\alpha,t):=\Phi(z(\alpha,t),t)$. Denote
\begin{equation}
Z(\alpha,t):=z\circ h^{-1}(\alpha,t),\quad \quad b=h_t\circ h^{-1}, \quad \quad D_t:=\partial_t+b\partial_{\alpha},\quad \quad 
\end{equation}
\begin{equation}
A:=(ah_{\alpha})\circ h^{-1}.\quad \quad 
\end{equation}
In Riemann mapping variables, the system (\ref{vortex_boundary}) becomes
\begin{equation}\label{vortex_model_Riemann}
\begin{cases}
(D_t^2-iA\partial_{\alpha})Z=-i\\
\frac{d}{dt}z_j(t)=(v-\frac{\lambda_j i}{2\pi(\overline{z-z_j})})\Big |_{z=z_j}\\
(I-\mathbb{H})(D_t\bar{Z}+\sum_{j=1}^N \frac{\lambda_j i}{2\pi(Z(\alpha,t)-z_j(t))})=0.
\end{cases}
\end{equation}
Here, $\mathbb{H}$ is the standard Hilbert transform which is defined by
\begin{equation}
\mathbb{H}f(\alpha):=\frac{1}{\pi i}p.v.\int_{-\infty}^{\infty}\frac{1}{\alpha-\beta}f(\beta)d\beta.
\end{equation}
Denote 
\begin{equation}
A_1:= A|Z_{\alpha}|^2.
\end{equation}
Since $-\frac{\partial P}{\partial \bold{n}}\Big|_{\Sigma(t)}=\frac{A_1}{|Z_{\alpha}|}$, it's clear that the Taylor sign condition holds if and only if
\begin{equation}\label{conditionA1}
\inf_{\alpha\in \mathbb{R}}\frac{A_1}{|Z_{\alpha}|}\geq 0,
\end{equation}
and the strong Taylor sign condition holds if and only if
\begin{equation}\label{conditionA1}
\inf_{\alpha\in \mathbb{R}}\frac{A_1}{|Z_{\alpha}|}> 0,
\end{equation}

\subsection{The main results} Our first result is a formula for the quantity $A_1$, from which we show that Taylor sign condition could fail if the point vortices are sufficient close to the interface. We also use this formula to find a criterion for strong Taylor sign condition to hold.

Let $F$ be the holomorphic extension of $\bar{z}_t+\sum_{j=1}^N \frac{\lambda_j i}{2\pi}\frac{1}{z(\alpha,t)-z_j(t)}$ in the domain $\Omega(t)$.
\begin{theorem}\label{taylorsignfail}
Denote
\begin{equation}
c_0^j:=(\Phi^{-1})_z(\omega_0^j,t),\quad \quad \omega_0^j:=\Phi(z_j(t),t).
\end{equation}
\begin{equation}
\beta_0(t):= \inf_{\alpha\in \mathbb{R}}|Z_{\alpha}(\alpha,t)|,\quad \quad M_0(t):=\|F(\cdot,t)\|_{\infty}
\end{equation}
\begin{equation}
\tilde{\lambda}=:\frac{\sum_{j=1}^N|\lambda_j|}{2\pi},\quad \quad \tilde{d}_I(t):=\min_{1\leq j\leq N}\inf_{\alpha\in \mathbb{R}}|\alpha-\Phi(z_j(t), t)|,\quad \quad \tilde{d}_P(t)=\min_{j\neq k}|z_j(t)-z_k(t)|. 
\end{equation}
\begin{itemize}
\item [(1)](Formula for the Taylor sign condition ) We have 
\begin{equation}\label{A1formula}
A_1=1+\frac{1}{2\pi}\int \frac{|D_tZ(\alpha,t)-D_tZ(\beta,t)|^2}{(\alpha-\beta)^2}d\beta-\sum_{j=1}^N \frac{\lambda_j}{\pi} Re\Big\{\frac{D_tZ-\dot{z}_j}{c_0^j(\alpha-w_0^j)^2}\Big\},
\end{equation}

\item [(2)](Failure of the Taylor sign condition) Taylor sign condition could Fail if $\tilde{d}_I(t)$ is sufficiently small.

\item [(3)] (A criterion for strong Taylor sign condition)
If 
\begin{equation}\label{taylorassumption}
\frac{\tilde{\lambda}^2}{\tilde{d}_I(t)^3\beta_0}+\frac{\tilde{\lambda}^2}{\tilde{d}_I(t)^2\tilde{d}_P(t)}+\frac{2M_0\tilde{\lambda}}{\tilde{d}_I(t)^2}< \beta_0,
\end{equation}
then the strong Taylor sign condition holds.
\end{itemize}
\end{theorem}
Part (1) is proved in Corollary \ref{goodformula},  part (2) is proved by examples from \S \ref{examplesectionone}, \S \ref{examplesection2}, and part (3) is proved in Proposition \ref{almostsharp}.
\vspace*{2ex}

Our second result justifies that (\ref{vortex_boundary}) is locally wellposed in Sobolev spaces, provided that the strong Taylor sign condition holds initially. Let $H^s$ represents the Sobolev space $H^s(\mathbb{R})$, which is define as  $$H^s(\mathbb{R}):=\Big\{f\in L^2(\mathbb{R}): \quad \norm{f}_{H^s}^2:= \int_{-\infty}^{\infty}(1+|2\pi\xi|^{2s})|\mathcal{F}f(\xi)|^2d\xi<\infty\Big\}.$$
Here, $\mathcal{F}f(\xi)=\int_{-\infty}^{\infty}f(x)e^{-2\pi ix\xi}dx$ is the Fourier transform of $f$. If $s$ is a nonnegative integer, then 
\begin{equation}
\norm{f}_{H^s}^2\leq \sum_{k=0}^s \norm{\partial_{\alpha}^k f}_{L^2}^2.
\end{equation}

\noindent Decompose $z_t$ as 
\begin{equation}\label{decomposition}
\bar{z}_t(\alpha,t)=f+p,\quad \quad \text{where}\quad p=-\sum_{j=1}^N\frac{\lambda_j i}{2\pi}\frac{1}{z(\alpha,t)-z_j(t)}.
\end{equation}
Note that $p$ and $p_t$ are determined by $z(\alpha,t)$ and $f$. 

\noindent $\bullet$ \underline{A discussion on the initial data}
Assume that the initial value for (\ref{vortex_boundary}) is given by 
\begin{equation}
\xi_0(\alpha):=z(\alpha,t=0)-\alpha,\quad \quad  v_0:=z_t(t=0), \quad \quad w_0:=z_{tt}(t=0),
\end{equation}
and denote 
\begin{equation}
 a_0:=a(t=0).
\end{equation}
$\xi_0, v_0, w_0,  a_0$ must satisfy
\begin{equation}\label{initialone}
    w_0-ia_0(\partial_{\alpha}\xi_0+1)=-i,
\end{equation}
where $a_0$ is determined by 
\begin{equation}\label{initialtwo}
    a_0|\partial_{\alpha}\xi_0+1|=|w_0+i|.
\end{equation}
$v_0$ satisfies 
\begin{equation}\label{initialthree}
    (I-\mathfrak{H}_0)\bar{v}_0=-\sum_{j=1}^N\frac{\lambda_j i}{\pi}\frac{1}{\xi_0(\alpha)+\alpha-z_j(0)},
\end{equation}
where $\mathfrak{H}_0$ is the Hilbert transform associated with the curve $z(\alpha,0)=\xi_0+\alpha$. 

\vspace*{1ex}

\noindent Denote 
\begin{equation}\label{interactionmeasure}
d_I(t):=\min_{1\leq j\leq N}\{d(z_j(t), \Sigma(t))\},  \quad \quad d_P(t):=\min_{\substack{1\leq i,j\leq N\\ i\neq j}}\{ d(z_i(t), z_j(t))\}.
\end{equation}
\begin{remark}
$d_I(t)$ represents the distance of the point vortices to the free boundary, the $'I'$ means interface. $d_P(t)$ represents the distance among the point vortices, $'P'$ means point vortices. 
\end{remark}
Let $\delta>0$. Let $|D|^{\delta}$ be defined by
\begin{equation}
    \mathcal{F}|D|^{\delta}f(\xi)=(2\pi |\xi|)^{\delta}\mathcal{F}f(\xi).
\end{equation}
\begin{theorem}(The local wellposedness)\label{theorem1} Assume $s\geq 4$. Assume $(|D|^{1/2}\xi_0, v_0, w_0)\in H^{s}\times H^{s+1/2}\times H^s$, satisfying (\ref{initialone}), (\ref{initialtwo}), (\ref{initialthree}), and 
\begin{itemize}
\item [(H1)]\underline{Strong Taylor sign assumption.} There is some $\alpha_0>0$ such that 
\begin{equation}\label{taylorinitial}
\inf_{\alpha\in \mathbb{R}}a(\alpha,0)|z_{\alpha}(\alpha,t=0)|\geq \alpha_0>0.
\end{equation}
\item [(H2)] \underline{Chord-arc assumption.} There are constants $C_1, C_2>0$ such that 
\begin{equation}
C_1|\alpha-\beta|\leq |z(\alpha,0)-z(\beta,0)|\leq C_2|\alpha-\beta|.
\end{equation}
\end{itemize}
Then exists $T_0>0$ such that (\ref{vortex_boundary}) admits a unique solution 
$$(|D|^{1/2}(z-\alpha), z_t, z_{tt})\in C([0,T_0]; H^{s}\times H^{s+1/2}\times H^s),$$ with $T_0$ depends on $\|(\partial_{\alpha}\xi_0, v_0, w_0)\|_{H^{s-1}\times H^{s}\times H^s}$, $d_I(0)^{-1}, d_P(0)^{-1}, C_1, C_2, \alpha_0$, $s$,  and 
\begin{equation}
\inf_{t\in [0,T_0]}\inf_{\alpha\in\mathbb{R}}a(\alpha,t)|z_{\alpha}(\alpha,t)|\geq \alpha_0/2.
\end{equation}
Moreover, if $T_0^*$ is the maximal lifespan, then either $T_0^*=\infty$, or $T_0^{\ast}<\infty$, but
\begin{equation}
\begin{split}
\lim_{T\rightarrow T_0^*-}\| (z_t, z_{tt})\|_{C([0,T];H^s\times H^s)}+\sup_{t\rightarrow T_0^*}(d_I(t)^{-1}+d_P(t)^{-1})=\infty.
\end{split}
\end{equation}
or 
\begin{equation}
   \lim_{t\rightarrow T_0^{\ast}-} \inf_{\alpha\in \mathbb{R}} a(\alpha,t)|z_{\alpha}(\alpha,t)|\leq 0,
\end{equation}
or 
\begin{equation}
   \sup_{\substack{\alpha\neq \beta\\ 0\leq t<T_0^*}}\Big |\frac{\alpha-\beta}{z(\alpha,t)-z(\beta,t)}\Big|+ \sup_{\substack{\alpha\neq \beta\\ 0\leq t<T_0^*}}\Big |\frac{z(\alpha,t)-z(\beta,t)}{\alpha-\beta}\Big|=\infty.
\end{equation}
\end{theorem}
\begin{remark}
(H1) is the strong Taylor sign condition. By Theorem \ref{taylorsignfail}, the Taylor sign condition (\ref{taylorinitial}) does not always hold. As was explained before, if surface tension is neglected and Taylor sign condition fails, the motion of the water waves could be subject to the Taylor instability. In order for the system (\ref{vortex_boundary}) to be wellposed in Sobolev spaces, we need to assume that the Taylor sign condition (\ref{taylorinitial}) to hold. 
\end{remark}

Our third result is concerned with the long time behavior of the water waves with point vortices. We show that if the water wave is symmetric with a symmetric vortex pair traveling downward initially, then the free interface remains smooth for a long time, and for initial data satisfying
$$\|(|D|^{1/2}(z(\alpha,0)-\alpha), f, f_t)\|_{H^{s}\times H^{s+1/2}\times H^s}\leq \epsilon\ll 1,$$ 
the lifespan is at least $\delta_0\epsilon^{-2}$, for some $\delta_0>0$. Define\footnote{We use the notation $\hat{d}_I$ to distinguish it from $d_I$. Please keep in mind that $\hat{d}_I$ is not the Fourier transform of $d_I$.}
\begin{equation}
    \hat{d}_I(t):=\min_{j=1,2}\inf_{\alpha\in \mathbb{R}}Im\{z(\alpha,t)-z_j(t)\}.
\end{equation}
We make the following assumptions:
\begin{itemize}
\item [(H3)] \underline{Vortex pair assumption.} Assume $N=2$, i.e., there are two point vortices, with positions $z_1(t)=x_1(t)+iy_1(t)), z_2(t)=x_2(t)+iy_2(t)$, strength $\lambda_1, \lambda_2$, respectively. Assume further that $z_1(t)$ and $z_2(t)$ are symmetric about the $y$-axis, i.e., 
$$x_1(t)=-x_2(t)=-x(t)<0,\quad y_1(t)=y_2(t):=y(t)<0, $$
and assume $\lambda_1=-\lambda_2:=\lambda<0$.

\item [(H4)] \underline{Symmetry assumption.}  Assume that velocity field $v=v_1+iv_2$ satisfies: $v_1$ odd in $x$, $v_2$ even in $x$, and the free boundary $\Sigma(t)$ is symmetric about the $y$-axis. 

\item [(H5)]\underline{Smallness assumption.}  Assume that at $t=0$, $$\|(|D|^{1/2}\xi_0, f(t=0), f_t(t=0))\|_{H^{s}\times H^{s+1/2}\times H^s}\leq \epsilon,\quad \quad \lambda^2+|\lambda x(0)|\leq c_0\epsilon,$$
for some constant $c_0=c_0(s)$. We can take $c_0=\frac{1}{((s+12)!)^2}$.

\item [(H6)] \underline{Vortex-vortex interaction.} Assume $\frac{|\lambda|}{x(0)}\geq M \epsilon$ for some constant $M\gg 1$ (say, $M=200\pi$). 

\item [(H7)] \underline{Vortex-interface interaction.} Assume $\hat{d}_I(0)\geq 1$. Assume $|\lambda|+x(0)\leq 1$.
\end{itemize}

\begin{remark}
Assume (H3)-(H4) holds at $t=0$, then by the uniqueness of the solutions to the system (\ref{vortex_model}), (H3)-(H4) holds for all $t\in [0,T]$ when the solution exists.
\end{remark}

\begin{theorem}[Long time behavior]\label{longtime}
Let $s\geq 4$. Assume (H3)-(H7). There exists $\epsilon_0>0$ and $\delta_0>0$ such that for all 
$0<\epsilon\leq \epsilon_0$, the lifespan $T_0^{\ast}$ of the solution to (\ref{vortex_model}) satisfies $T_0^{\ast}\geq \delta_0\epsilon^{-2}$. Moreover, there exists a constant $C_s$ only depends on $s$ such  that
\begin{equation}
\sup_{t\in [0,\delta_0\epsilon^{-2}]}\norm{(|D|^{1/2}(z(\alpha,t)-\alpha), f(\cdot, t), f_t(\cdot, t))}_{H^{s}\times H^{s+1/2}\times H^s}\leq C_s\epsilon.
\end{equation}
Here, $\delta_0$ is an absolute constant independent of $\epsilon$ and $s$.
\end{theorem}


\begin{remark}
If the initial data is sufficiently localized, then we can prove global wellposedness and modified scattering. A brief discussion of the main idea will be given in \S \ref{global_remark}. We will give the full details of the proof in a forthcoming paper.
\end{remark}




\begin{remark}
The assumption $\frac{|\lambda|}{x(0)}\geq M\epsilon$ ensures that the point vortices travel downward at $t=0$. In the proof of Theorem \ref{longtime}, we will show that when $\frac{|\lambda |}{x(0)}=M\epsilon$, the velocity of the point vortices is comparable to $\epsilon$,  which is slow in  some sense. Theorem \ref{longtime} demonstrates that even if the point vortices moves at an initial velocity as slow as $M\epsilon$, the water waves still remain smooth and small for a long time.
\end{remark}

\begin{remark}
Assumption (H7) implies that the strong Taylor sign condition holds initially. The assumption $\hat{d}_I(0)\geq 1$ can be relaxed.  To avoid getting into too many technical issues, we simply assume $\hat{d}_I(0)\geq 1$. The assumption $|\lambda|+x(0)\leq 1$ is not an essential assumption.  We assume this merely for convenience.

\end{remark}

\begin{remark}\label{largeinteraction}
The assumptions (H5), (H6), (H7) do allow $x(0)$ to be as small as we want. So $\frac{|\lambda|}{x(0)}$ can be very large.
\end{remark}

\vspace*{2ex}

\noindent \underline{A discussion on initial data.}
We need to show that initial data for this system satisfy the assumptions of Theorem \ref{longtime} exist. As before, denote 
$$\xi_0(\alpha)=z(\alpha,0)-\alpha,\quad z_0=\alpha+\xi_0,\quad v_0=z_t(\alpha,0),\quad w_0=z_{tt}(\alpha,0).$$
We need $\xi_0, v_0, w_0$ satisfy (\ref{initialone}), (\ref{initialtwo}), and (\ref{initialthree}).
We need also the symmetry condition 
\begin{equation}\label{is4}
Re\{v_0\} \text{ is odd in } \alpha,\quad \quad Im\{v_0\} \text{ is even in } \alpha,\quad \quad Re\{\xi_0\}~is~odd\quad \quad Im\{\xi_0\}~is~even.
\end{equation}
Denote the Hilbert transform associates to $z_0(\alpha)$ by $\mathfrak{H}_0$. We'll use the following lemma.
\begin{lemma}\label{hilbertodd}
Let $Im\{\xi_0\}$ be even, $Re\{\xi_0\}$ be odd. Let $f=f_1+if_2$ be such that $f_1$ is odd and $f_2$ is even. Then $Re\{\mathfrak{H}_0f\}$ is odd and $Im\{\mathfrak{H}_0f\}$ is even.
\end{lemma}
\begin{proof}
This is proved by direct calculation.
\end{proof}
Given $\xi_0$ be such that $Re\{\xi_0\}$ odd and $Im\{\xi_0\}$ even, and given any real valued odd function $f$, if we let $v_0$ be such that
\begin{equation}
\bar{v}_0=\frac{1}{2}(I+\mathfrak{H}_0)f-\sum_{j=1}^2 \frac{\lambda_j i}{2\pi}\frac{1}{z_0(\alpha)-z_j(0)},
\end{equation}
then by lemma \ref{hilbertodd}, we have $Re\{v_0\}$ is odd and $Im\{v_0\}$  is even, and satisfying the compatiability condition
\begin{equation}
(I-\mathfrak{H}_0)(\bar{v}_0+\sum_{j=1}^2 \frac{\lambda_j i}{2\pi}\frac{1}{z_0(\alpha)-z_j(0)})=0.
\end{equation}

\subsection{Strategy of proof} we illustrate the strategy of proving the main theorems in this subsection. The first two theorems are more or less routine, while Theorem \ref{longtime} requires some new idea of controlling the motion of the point vortices. 
\subsubsection{The Taylor sign condition: proof of Theorem \ref{taylorsignfail}} We follow S. Wu's work \cite{Wu1997} to calculate the Taylor sign condition.    Using Riemann variables, the momentum equation is written as $(D_t^2+iA\partial_{\alpha})\bar{Z}=i$. Recall that $A_1:=A|Z_{\alpha}|^2$. Multiply $(D_t^2+iA\partial_{\alpha})\bar{Z}=i$ by $Z_{\alpha}$, apply $I-\mathbb{H}$ on both sides of the resulting equation, then take imaginary part. By using the facts  
\begin{equation}
(I-\mathbb{H})(Z_{\alpha}-1)=0, \quad \quad (I-\mathbb{H})(D_t\bar{Z}+\sum_{j=1}^N \frac{\lambda_j i}{2\pi}\frac{1}{Z(\alpha,t)-z_j(t)})=0,
\end{equation}
we obtain
$$A_1=1+\frac{1}{2\pi}\int \frac{|D_tZ(\alpha,t)-D_tZ(\beta,t)|^2}{(\alpha-\beta)^2}d\beta-Im\Big\{\sum_{j=1}^N \frac{\lambda_j i}{2\pi} \Big((I-\mathbb{H})\frac{Z_{\alpha}}{(Z(\alpha,t)-z_j(t))^2}\Big)(D_tZ-\dot{z}_j(t)) \Big\}.$$
Then use some tools from complex analysis, we obtain (\ref{A1formula}). 

To construct examples for which the Taylor sign condition fails, we consider initially still water waves with its motion purely generated by the point vortices. We are able to derive a formula for $A_1$ in terms of the intensity and location of these point vortices, from which we can see that Taylor sign condition could fail if the point voritces are close to the interface.


\subsubsection{Local wellposedness:proof of Theorem \ref{theorem1}} If there is no point vortices in the water waves,  S. Wu observes that one can obtain quasilinearization of the system (\ref{vortex_boundary}) by taking one time derivative to the momentum equation. It turns out that this is still true for water waves with point vortices:  take $\partial_t$ on both sides of $(\partial_t^2+ia\partial_{\alpha})\bar{z}=i$, we obtain
\begin{equation}\label{onetimederivative}
(\partial_t^2+ia\partial_{\alpha})\bar{z}_t=-ia_t\bar{z}_{\alpha}.
\end{equation}
In (\ref{decomposition}), we decompose $\bar{z}_t$ as $\bar{z}_t=f+p$, where $p=-\sum_{j=1}^N\frac{\lambda_j i}{2\pi}\frac{1}{z(\alpha,t)-z_j(t)}$. A key observation is that $(\partial_t^2+ia\partial_{\alpha})p$ consists of lower order terms. Apply $I-\mathfrak{H}$ on both sides of equation (\ref{onetimederivative}), we obtain
\begin{equation}\label{goodgood}
-i(I-\mathfrak{H})a_t\bar{z}_{\alpha}=g_1+g_2,
\end{equation}
where
\begin{equation}
g_1:=2[z_{tt}, \mathfrak{H}]\frac{\bar{z}_{t\alpha}}{z_{\alpha}}+2[z_t, \mathfrak{H}] \frac{\bar{z}_{tt\alpha}}{z_{\alpha}}-\frac{1}{\pi i}\int \Big(\frac{z_t(\alpha,t)-z_t(\beta,t)}{z(\alpha,t)-z(\beta,t)}\Big)^2 \bar{z}_{t\beta}d\beta.
\end{equation}
\begin{equation}\label{g2g2}
g_2:=\frac{i}{\pi}\sum_{j=1}^N \lambda_j\Big(\frac{2z_{tt}+i-\partial_t^2 z_j}{(z(\alpha,t)-z_j(t))^2}-2\frac{(z_t-\dot{z}_j(t))^2}{(z(\alpha,t)-z_j(t))^3}\Big).
\end{equation}
So $a_t\bar{z}_{\alpha}$  is of lower order. The quasilinear system 
\begin{equation}\label{firstquasilinear}
\begin{cases}
(\partial_t^2+ia\partial_{\alpha})\bar{z}_t=-ia_t\bar{z}_{\alpha}\\
\dot{z}_j(t)=(v-\frac{\lambda_j i}{2\pi}\frac{1}{\overline{z(\alpha,t)-z_j(t)}})\Big|_{z=z_j(t)}\\
(I-\mathfrak{H})(\bar{z}_t+\sum_{j=1}^N\frac{\lambda_j i}{2\pi}\frac{1}{z(\alpha,t)-z_j(t)})=0.
\end{cases}
\end{equation}
is of hyperbolic type as long as the Taylor sign condition $a|z_{\alpha}|\geq \alpha_0>0$ holds.
The local wellposedness is obtained by energy method. 

\subsubsection{Long time behavior: proof of Theorem \ref{longtime}}  To illustrate the idea of studying long time behavior, we begin with the following toy model.

\vspace*{2ex}

\noindent \textbf{\underline{Toy model:}}  Consider 
\begin{equation}
    u_{tt}+|D|u=u_t^p+\frac{C}{(\alpha+it)^m},\quad \quad p\geq 2, m\geq 2.
\end{equation}
for some constant $C$ such that $|C|\lesssim \epsilon$.
Define an energy 
\begin{equation}
    E_s(t)=\sum_{k\leq s}\int |\partial_{\alpha}^k u_t(\alpha,t)|^2+||D|^{1/2}\partial_{\alpha}^ku|^2 d\alpha.
\end{equation}
Then we have 
\begin{align}
    \frac{d}{dt}E_s(t)\lesssim E_s(t)^{(p+1)/2}+\epsilon(1+|t|)^{-(m-1/2)}E_s(t)^{1/2}.
\end{align}
Assume $E_s(0)\lesssim \epsilon^2$. By the bootstrap argument, we can prove 
\begin{equation}
    E_s(t)\lesssim \epsilon^2,\quad \quad \forall~t\lesssim \epsilon^{1-p}.
\end{equation}
If the nonlinearity is at least cubic, i.e., $p\geq 3$, then the lifespan is at least $\epsilon^{-2}$.

\vspace*{2ex}

\noindent \textbf{\underline{The water waves:}}
If we can find $\theta$, $\theta\approx z_t$, such that $(\partial_t^2+|D|)\theta=F(z_t, |D|z,  z_{tt})+O(\frac{1}{(\alpha+it})^m)$, where $F$ is at least cubic and $m\geq 2$, then use an argument similar to that for the toy model, we expect cubic lifespan. Note that the nonlinearity of the quasilinear system (\ref{firstquasilinear}) is quadratic, so it does not directly lead to lifespan of order $O(\epsilon^{-2})$.  In the irrotational cases, S. Wu\cite{Wu2009} found that the fully nonlinear transform $\theta:=(I-\mathfrak{H})(z-\bar{z})$  satisfies
\begin{equation}\label{perfect}
\begin{split}
(\partial_t^2-ia\partial_{\alpha})\theta=-2[z_t,\mathfrak{H}\frac{1}{z_{\alpha}}+\bar{\mathfrak{H}}\frac{1}{\bar{z}_{\alpha}}]z_{t\alpha}+\frac{1}{\pi i}\int_{-\infty}^{\infty}\Big(\frac{z_t(\alpha,t)-z_t(\beta,t)}{z(\alpha,t)-z(\beta,t)}\Big)^2(z-\bar{z})_{\beta}d\beta:=g.
\end{split}
\end{equation}
$g$ is cubic, while $a-1$ contains first order terms, so $(\partial_t^2+|D|)\theta$ contains quadratic terms, which does not imply cubic lifespan. To resolve the problem, S. Wu considered change of variables $\kappa:\mathbb{R}\rightarrow \mathbb{R}$. She let $\zeta=z\circ\kappa^{-1}$, $b=\kappa_t\circ\kappa^{-1}$, $A=(a\kappa_{\alpha})\circ\kappa^{-1}$. 
In new variables, the system (\ref{vortex_boundary}) is written as (with $\lambda_j=0$ for irrotational case)
\begin{equation}\label{vortex_new}
\begin{cases}
(D_t^2-iA\partial_{\alpha})\zeta=-i\\
(I-\mathcal{H})D_t\bar{\zeta}=0,
\end{cases}
\end{equation}
and (\ref{perfect}) becomes 
\begin{equation}\label{perfect_new}
(D_t^2-iA\partial_{\alpha})\theta\circ\kappa^{-1}=-2[D_t\zeta,\mathcal{H}\frac{1}{\zeta_{\alpha}}+\bar{\mathcal{H}}\frac{1}{\bar{\zeta}_{\alpha}}]\partial_{\alpha}D_t\zeta+\frac{1}{\pi i}\int_{-\infty}^{\infty}\Big(\frac{D_t\zeta(\alpha,t)-D_t\zeta(\beta,t)}{\zeta(\alpha,t)-\zeta(\beta,t)}\Big)^2(\zeta-\bar{\zeta})_{\beta}d\beta,
\end{equation}
where 
\begin{equation}
\mathcal{H}f(\alpha):=\frac{1}{\pi i}p.v.\int_{-\infty}^{\infty}\frac{\zeta_{\beta}}{\zeta(\alpha,t)-\zeta(\beta,t)}f(\beta)d\beta.
\end{equation}
She realized that there exists a change of variables $\kappa$ such that 
\begin{equation}\label{formulaforbbb}
(I-\mathcal{H})b=-[D_t\zeta,\mathcal{H}]\frac{\bar{\zeta}_{\alpha}-1}{\zeta_{\alpha}},
\end{equation}
\begin{equation}\label{formulaforaaa}
 (I-\mathcal{H})(A-1)=i[D_t\zeta,\mathcal{H}]\frac{\partial_{\alpha}D_t\bar{\zeta}}{\zeta_{\alpha}}+i[D_t^2\zeta,\mathcal{H}]\frac{\bar{\zeta}_{\alpha}-1}{\zeta_{\alpha}}.
\end{equation}
So $b$, $A-1$ are quadratic. Using this, S. Wu was able to prove the almost global existence for the irrotational water waves with small localized initial data. The method implies lifespan of order $O(\epsilon^{-2})$ for nonlocalized data of size $O(\epsilon)$ for irrotational water waves.

\vspace*{2ex}

Assume there are point vortices in the fluid. We use S. Wu's change of variables, by taking $\kappa:\mathbb{R}\rightarrow \mathbb{R}$, satisfying that for $\zeta=z\circ \kappa^{-1}$, 
\begin{equation}
(I-\mathcal{H})(\bar{\zeta}-\alpha)=0.
\end{equation}
In new variables, by  direct calculation, we have 
\begin{equation}\label{evolutiong1}
\begin{split}
(D_t^2-iA\partial_{\alpha})\tilde{\theta}=G_c+G_d,
\end{split}
\end{equation}
where $\tilde{\theta}=(I-\mathcal{H})(\zeta-\bar{\zeta})$, $A=(a\kappa_{\alpha})\circ \kappa^{-1}$, $b=\kappa_t\circ\kappa^{-1}$, $D_t=\partial_t+b\partial_{\alpha}$. Let $\mathfrak{F}=f\circ\kappa^{-1}$, $q=p\circ\kappa^{-1}$. We have 
\begin{equation}\label{evolutiong2}
G_c:=-2[\bar{\mathfrak{F}}, \mathcal{H}\frac{1}{\zeta_{\alpha}}+\bar{\mathcal{H}}\frac{1}{\bar{\zeta}_{\alpha}}]\bar{\mathfrak{F}}_{\alpha}+\frac{1}{\pi i}\int \Big(\frac{D_t\zeta(\alpha,t)-D_t\zeta(\beta,t)}{\zeta(\alpha,t)-\zeta(\beta,t)}\Big)^2(\zeta-\bar{\zeta})_{\beta}d\beta.
\end{equation}
\begin{equation}\label{evolutiong3}
G_d:=-2[\bar{q}, \mathcal{H}]\frac{\partial_{\alpha}\bar{\mathfrak{F}}}{\zeta_{\alpha}}-2[\bar{\mathfrak{F}},\mathcal{H}]\frac{\partial_{\alpha}\bar{q}}{\zeta_{\alpha}}-2[\bar{q},\mathcal{H}]\frac{\partial_{\alpha}\bar{q}}{\zeta_{\alpha}}-4D_t q.
\end{equation}
\begin{equation}
(I-\mathcal{H})b=-[D_t\zeta, \mathcal{H}]\frac{\bar{\zeta}_{\alpha}-1}{\zeta_{\alpha}}-\frac{i}{\pi} \sum_{j=1}^2 \frac{\lambda_j}{\zeta(\alpha,t)-z_j(t)}.
\end{equation}

\begin{equation}
\begin{split}
(I-\mathcal{H})A=&1+i[D_t\zeta,\mathcal{H}]\frac{\partial_{\alpha}\mathfrak{F}}{\zeta_{\alpha}}+i[D_t^2\zeta,\mathcal{H}]\frac{\bar{\zeta}_{\alpha}-1}{\zeta_{\alpha}}-(I-\mathcal{H})\frac{1}{2\pi}\sum_{j=1}^2 \frac{\lambda_j(D_t\zeta(\alpha,t)-\dot{z}_j(t))}{(\zeta(\alpha,t)-z_j(t))^2}.
\end{split}
\end{equation}
To control the acceleration $D_t^2\zeta$, we consider $\tilde{\sigma}=(I-\mathcal{H})D_t\tilde{\theta}$. We have 
\begin{equation}\label{evoG1}
(D_t^2-iA\partial_{\alpha})\tilde{\sigma}=\tilde{G}, 
\end{equation}
where 
\begin{equation}\label{evG2}
\begin{split}
\tilde{G}=& (I-\mathcal{H})(D_t G+i\frac{a_t}{a}\circ\kappa^{-1}A((I-\mathcal{H})(\zeta-\bar{\zeta}))_{\alpha})-2[D_t\zeta,\mathcal{H}]\frac{\partial_{\alpha}D_t^2(I-\mathcal{H})(\zeta-\bar{\zeta})}{\zeta_{\alpha}}\\
&+\frac{1}{\pi i}\int \Big(\frac{D_t\zeta(\alpha,t)-D_t\zeta(\beta,t)}{\zeta(\alpha,t)-\zeta(\beta,t)}\Big)^2 \partial_{\beta} D_t(I-\mathcal{H})(\zeta-\bar{\zeta})d\beta\\
\end{split}
\end{equation}
\vspace*{2ex}

\noindent \textbf{\underline{The difficulty:}} 
\begin{itemize}
\item [(1)] $G_c$ is cubic,
 while $G_d$ consists of quadratic and first order terms. $G_d$ is the contribution from the point vortices. Similarly, $b$, $A-1$ contains first order terms due to the presence of point vortices. 
 
 \item [(2)] Each term of $G_d$ contains factors of the form $\sum_{j=1}^N\frac{\lambda_j}{(\zeta(\alpha,t)-z_j(t))^k}$ for some $k\geq 1$. It's possible that the point vortices travel upward and get closer and closer to the free interface. In that case, $G_d$ becomes very large. 
 
 \item [(3)] The strong interaction between the point vortices could excite the water waves and make it significantly larger in a short time. Assume two point vortices $z_1(t)=-x+iy, z_2(t)=x+iy$, with strength $\lambda_1$ and $\lambda_2$, respectively.  Then the velocity of $z_1$ is 
 $$\dot{z}_1=-\frac{\lambda_2 i}{2\pi(z_2-z_1)}+\bar{F}(z_1(t),t),\quad \quad \dot{z}_2=\frac{\lambda_1 i}{2\pi(z_2-z_1)}+\bar{F}(z_2(t),t).$$
Roughly speaking, if $\frac{|\lambda|}{2\pi |z_2-z_1|}$ is large, then $|\dot{z}_j(t)|$ is large. Since $p_t=\sum_{j=1}^2\frac{\lambda_j i}{2\pi}\frac{z_t-\dot{z}_j(t)}{(z(\alpha,t)-z_j(t))^2}$, in general, $\|p_t\|_{H^s}$ could be large as well. Therefore, small data theory does not directly apply in such situation. Moreover,  $\tilde{G}$ contains the term $\sum_{j=1}^2\frac{\lambda_j i}{2\pi}\frac{(\dot{z}_j)^2}{(\zeta(\alpha,t)-z_j(t))^3}$, which is even worse than $p_t$ if $\dot{z}_j(t)$ is large. In theorem \ref{longtime}, we do allow $|\frac{\lambda}{z_1-z_2}|$ large. See Remark \ref{largeinteraction}.

\item [(4)] If the point vortices collide, after the collision, we cannot use the same system to describe the motion of the fluid anymore, for the reason that the vorticity after the collision is not the same as that before it, which vialates the conservation of vorticity.
\end{itemize}

\vspace*{2ex}

\noindent \textbf{\underline{The idea:}} Intuitively, if each point vortices $z_j(t)$ moves away from the free boundary rapidly, with the factor $\frac{1}{z(\alpha,t)-z_j(t)}$  decaying in time at least at a linear rate, then we could overcome the difficulties (1) and (2). We will show that this indeed is true if (H3)-(H6) holds initially. To overcome the difficulty (3), we use $\lambda_1=-\lambda_2$ from assumption (H3), and by direct calculation, $\|p_t\|_{H^s}$ does not depend on $\dot{z}_1, \dot{z}_2$, resolving difficulty (3). Also, although the term $\norm{\sum_{j=1}^2\frac{\lambda_j i}{2\pi}\frac{(\dot{z}_j)^2}{(z(\alpha,t)-z_j(t))^3}}_{H^s}$ could be large at time $t=0$, yet its long time effect remains small, i.e., 
$$\int_0^T \norm{\sum_{j=1}^2\frac{\lambda_j i}{2\pi}\frac{(\dot{z}_j)^2}{(z(\alpha,t)-z_j(t))^3}}_{H^s} dt\leq   C' \epsilon,  $$
for some constant $C'$ which is independent of $\frac{|\lambda|}{4\pi x(0)}$.
So we are able to overcome the difficulty of large vortex-vortex interaction, provided that the point vortices keeps traveling away from the free interface.

\vspace*{2ex}
However, the motion of point vortices interferes with the motion of the water waves, it's not obvious at all that why the point vortices should escape to the deep water (toward $y=-\infty$). Indeed, in some cases, they can travel upwards toward the interface.  Recall that the velocity of $z_j$ is given by
\begin{equation}
\dot{z}_j(t)=(v-\frac{\lambda_j i}{2\pi}\frac{1}{\overline{z-z_j(t)}})\Big|_{z=z_j(t)}
\end{equation}
The point vortices could interact with each other and with the water waves. In general, it's not always true that $Im \{\dot{z}_j(t)\}<0$ (i.e., travels downward).  

In the following, we discuss how the number of point vortices, the sign of $\lambda_j$, and the strength of $|\lambda_j|$ affect the motion of point vortices.

\vspace{2ex}
\begin{itemize}
\item[(1)] For $N=1$, the motion of the point vortex is hard to predict except for some special cases.  

\item [(2)] For $N=2$, $\lambda_1=\lambda_2$. The two point vortices is more likely to rotate about each other and excite the fluid. 

\item [(3)] $N=2$, $\lambda_1=-\lambda_2=\lambda>0$. In this case, the point vortices are more likely to move upward and getting closer and closer to the interface and hence cause the Taylor sign condition to fail. 

\item [(4)] $N=2$, $\lambda_1=-\lambda_2=\lambda<0$ and $\frac{|\lambda|}{|z_1-z_2|}$ is relatively large (say, $\frac{|\lambda|}{|z_1-z_2|}\gg \epsilon$, where $\epsilon$ is the size of the initial data), then the point vortices moving straight downward rapidly and hence the term $\frac{1}{z(\alpha,t)-z_j(t)}$ decays like $t^{-1}$.

\item [(5)] $N\geq 3$. This is far beyond understood. Indeed, even for 2d Euler equations (fixed boundary) with point vortices, the problem is still not fully understood. When $N=3$, the problem resembles the three-body problem .

\item [(6)] Vortex pairs. With $2N$ point vortices, denoted by $\{(z_{i1}, z_{i2}\}_{i=1}^N$. Let the corresponding strength be $\lambda_{i1}, \lambda_{i2}$, respectively. Assume $\lambda_{i1}=-\lambda_{i2}<0$. Assume $|z_{i1}-z_{i2}|$ is sufficiently small, while different pairs  are sufficiently far away from each other. Then the point vortices travel downward, at least for a short time. It's likely that the factor $\frac{1}{z(\alpha,t)-z_{ij}}$ decays linearly in time. So long time existence will not be a surprise. For brevity, we consider only one vortex pair.
\end{itemize}

\vspace*{2ex}
Therefore, from the above discussion, if we assume $N=2$ and $\lambda_1=-\lambda_2<0$, we expect that the point vortices keep traveling downward at a speed comparable to its initial speed,
 hence 
\begin{equation}
|\zeta(\alpha,t)-z_j(t)|^{-1}=O(\frac{1}{\alpha+i\frac{|\lambda|}{x(0)}t}),
\end{equation}
and we can expect to have 
$$(D_t^2-iA\partial_{\alpha})\tilde{\theta}=G_c+O(\frac{1}{(\alpha+i\frac{|\lambda|}{x(0)}t)^2}),$$
and 
$$(D_t^2-iA\partial_{\alpha})\tilde{\theta}=(\partial_t^2+|D|)\tilde{\theta}+cubic+O(\frac{1}{(\alpha+i\frac{|\lambda|}{x(0)}t)^2},$$
hence
$$(\partial_t^2+|D|)\tilde{\theta}=cubic+O(\frac{1}{(\alpha+i\frac{|\lambda|}{x(0)}t)^2}).$$
Similarly, at least formally, we have 
$$(\partial_t^2+|D|)\tilde{\sigma}=cubic+O(\frac{1}{(\alpha+i\frac{|\lambda|}{x(0)}t)^2}).$$
From the discussion on the Toy model, we expect to prove lifespan of order $O(\epsilon^{-2})$ for small nonlocalized data of size $O(\epsilon)$.
\vspace*{2ex}

Let's summarize our previous discussion in a more precisely way as the following.
\begin{itemize}
\item [Step 1.] \underline{Change of variables.} Let $\kappa$ be the change of variables such that $(I-\mathcal{H})(\bar{\zeta}-\alpha)=0$, where $\zeta=z\circ\kappa^{-1}$. Then we derive the formula (\ref{formulaforbbb}) for the quantity $b$ and the formula (\ref{formulaforaaa}) for the quantity $A-1$.

\item [Step 2.] \underline{Nonlinear transform.}  Let $\tilde{\theta}=\theta\circ\kappa^{-1}, \tilde{\sigma}=D_t\tilde{\theta}$, where $\tilde{\theta}=(I-\mathcal{H})(\zeta-\bar{\zeta})$. Then we derive water wave equations (\ref{evolutiong1})-(\ref{evolutiong3}) for $\tilde{\theta}$ and water wave equations (\ref{evoG1})-(\ref{evG2}) for $\tilde{\sigma}$.

\item [Step 3.] \underline{Bootstrap assumption.} Assume that on $[0, T]$, we have 
\begin{equation}\label{aprioriaaa}
\|\zeta_{\alpha}-1\|_{H^s}\leq 5\epsilon,\quad \quad  \|\mathfrak{F}\|_{H^{s+1/2}}\leq 5\epsilon,\quad \quad \|D_t \mathfrak{F}\|_{H^s}\leq 5\epsilon\quad \quad, \forall~ t\in [0,T],
\end{equation}
where $\mathfrak{F}=f\circ\kappa^{-1}$.

\item [Step 4.] \underline{Control the motion of $z_j(t)$.} Under the bootstrap assumption (\ref{aprioriaaa}), we show that for any $t\in [0,T]$,
\begin{equation}
\frac{1}{2}\leq \frac{x(t)}{x(0)}\leq 2. 
\end{equation}
In another words, the trajectory of the point vortices are almost parallel to each other.  

Therefore, we obtain decay estimate
\begin{equation}
d_I(t)^{-1}=\Big(\min_{j=1,2}\inf_{\alpha\in \mathbb{R}}|\zeta(\alpha,t)-z_j(t)|\Big)^{-1}\leq (1+\frac{|\lambda|}{20\pi x(0)}t)^{-1},
\end{equation}

\item [Step 5.] \underline{Energy estimates.} Denote $\theta_k=(I-\mathcal{H})\partial_{\alpha}^k \tilde{\theta}$, $\sigma_k=(I-\mathcal{H})\partial_{\alpha}^k \tilde{\sigma}$. Define energy 
\begin{equation}
E(t)=\sum_{0\leq k\leq s}\Big\{ \int \frac{1}{A}|D_t\theta_k|^2+\int \frac{1}{A}|D_t\sigma_k|^2 +i\int \theta_k\overline{\partial_{\alpha}\theta_k} +i\int \sigma_k\overline{\partial_{\alpha}\sigma_k}\Big\}.
\end{equation}
By energy estimates, use the time decay of $\frac{1}{\zeta(\alpha,t)-z_j(t)}$, we obtain control of $D_t\tilde{\theta}$ and $D_t\tilde{\sigma}$.  As a consequence, by bootstrap argument,  
we show that $T^{\ast}\geq \delta\epsilon^{-2}$. 

\item [Step 6.] Change of variables back to lagrangian coordinates and completes the proof.
\end{itemize}

\subsection{A remark on global existence and modified scattering.}\label{global_remark} If the initial data is suffciently localized similar to those considered by Ionescu $\&$ Pusateri in \cite{Ionescu2015} and Alazard $\&$ Delort in \cite{AlazardDelort}, then we expect to prove that the system (\ref{vortex_model}) (or equivalently, the system (\ref{vortex_boundary})) is globally wellposed and there is  modified scattering. In this subsection, we explain  why this should be true. 

Let's consider localized initial data of size $\epsilon$ (in weighted Sobolev spaces), with $\epsilon$ small. Let's assume the assumptions (H3)-(H7) (The smallness assumption in Sobolev spaces is replaced by smallness assumption in weighted Sobolev spaces). By Theorem \ref{longtime}, the system (\ref{vortex_boundary}) admits a unique solution on $[0,\delta\epsilon^{-2}]$. For $t\in [0,\delta\epsilon^{-2}]$, the position of the point vortices $z_j=x_j+iy_j$ satisfies
    \begin{equation}
        \frac{1}{2}\leq \frac{|x_j(t)|}{|x_j(0)|}\leq 2,\quad \quad y_j(t)\leq -\frac{|\lambda|}{20\pi x(0)}t.
    \end{equation}
Therefore, as long as the growth of the water waves (measured by $\zeta_{\alpha}-1, f, f_t$) is slow,  the vortex pair will keep travelling downward at a speed comparable to its initial speed, and therefore the effect of the point vortices will keep small for all time. Indeed, we have 
\begin{equation}
\int_0^T \norm{G_d(\cdot,t)}_{H^s}dt\leq C'\epsilon, 
\end{equation}
for some absolute constant $C'>0$ on any time $[0,T]$ on which the solution exists. From this point of view, the vortex pair is a globally small perturbation of the irrotational flow. If the initial data is sufficiently localized, then the effect from the vortex pair is also localized, and a similar argument as in \cite{Ionescu2015} or \cite{AlazardDelort}  will give global existence and  modified scattering.


\subsection{Outline of the paper}
In \S\ref{notation}, we introduce some basic notation and convention. Further notation and convention will be made throughout the paper if necessary. In \S\ref{prelim} we will provide some analytical tools that will be used in later sections. In \S\ref{sectionsign}, we give a systematic investigation of the Taylor sign condition. We give examples that Taylor sign condition fails. We also give a sufficient condition which implies the strong Taylor sign condition .  In \S\ref{sectionlocal}, we prove Theorem \ref{theorem1}. In \S\ref{sectionlongtime}, we prove Theorem \ref{longtime}. 

\subsection{Notation and convention}\label{notation} We assume that the velocity field $|v(z,t)|\rightarrow 0$ as $|z|\rightarrow \infty$. We assume also that $z(\alpha,t)-\alpha\rightarrow 0$ as $|\alpha|\rightarrow \infty$. We use $C(X_1, X_2,...,X_k)$ to denote a positive constant $C$ depends continuous on the parameters $X_1,...,X_k$. Throughout this paper, such constant $C(X_1,...,X_k)$ could be different even we use the same letter $C$. 
The commutator $[A,B]=AB-BA$. Given a function $g(\cdot,t):\mathbb{R}\rightarrow \mathbb{R}$, the composition $f(\cdot, t)\circ g:=f(g(\cdot,t),t)$. For a function $f(\alpha,t)$ along the free surface $\Sigma(t)$,  we say $f$ is holomorphic in $\Omega(t)$ if there is some holomorphic function $F:\Omega(t)\rightarrow \mathbb{C}$ such that $f=F|_{\Sigma(t)}$. We identify the $\mathbb{R}^2$ with the complex plane. A point $(x,y)$ is identified as $x+iy$. For a point $z=x+iy$, $\bar{z}$ represents the complex conjugate of $z$.

\section{Preliminaries and basic analysis}\label{prelim}

\begin{lemma}[Sobolev embedding]\label{sobolevembedding}
Let $s>1/2$. Let $f\in H^{s}(\mathbb{R})$. Then $f\in L^{\infty}$, and 
$$\|f\|_{\infty}\leq C\|f\|_{H^s},$$
where $C=C(s)$. If $s=1$, we can take $C=1$.
\end{lemma}
\subsection{Hilbert transform, layer potentials}
\begin{definition}[Hilbert transform]
We define the Hilbert transform associates with a curve $z(\alpha,t)$ as 
\begin{equation}
\mathfrak{H}f(\alpha):=\frac{1}{\pi i}p.v.\int_{-\infty}^{\infty}\frac{z_{\beta}(\beta,t)}{z(\alpha,t)-z(\beta,t)}f(\beta,t)d\beta.
\end{equation}
The standard Hilbert transform is the one associated with $z(\alpha)=\alpha$, which is denoted by
\begin{equation}
\mathbb{H}f(\alpha):=\frac{1}{\pi i}p.v.\int_{-\infty}^{\infty}\frac{1}{\alpha-\beta}f(\beta)d\beta.
\end{equation}
\end{definition}
\noindent It is well-known that the following holds:
\begin{lemma}\label{holomorphic}
Let $f\in L^2(\mathbb{R})$. Then $f$ is the boundary value of a holomorphic function in $\Omega(t)$ if and only if $(I-\mathfrak{H})f=0$. $f$ is the boundary value of a holomorphic function in $\Omega(t)^c$ if and only if $(I+\mathfrak{H})f=0$. 
\end{lemma}
Because of the singularity of the velocity at the point vortices, we don't have $(I-\mathfrak{H})\bar{z}_t=0$. However, the following lemma asserts that $\bar{z}_t$ is almost holomorphic, in the sense that $(I-\mathfrak{H})\bar{z}_t$ consists of lower order terms.
\begin{lemma}[Almost holomorphicity]\label{almost}
We have
\begin{equation}
\begin{split}
(I-\mathfrak{H})\bar{z}_t=&-\frac{i}{\pi}\sum_{j=1}^N \frac{\lambda_j}{z(\alpha,t)-z_j(t)}.
\end{split}
\end{equation}
\end{lemma}
\begin{proof}
Since $\bar{z}_t+\sum_{j=1}^N \frac{\lambda_j i}{2\pi(z(\alpha,t)-z_j(t))}$ is the boundary value of a holomorphic function in $\Omega(t)$, by lemma \ref{holomorphic},
$$(I-\mathfrak{H})\Big(\bar{z}_t+\sum_{j=1}^N \frac{\lambda_j i}{2\pi(z(\alpha,t)-z_j(t))}\Big)=0,$$
hence 
\begin{equation}\label{holo1}
(I-\mathfrak{H})\bar{z}_t=-\sum_{j=1}^N (I-\mathfrak{H}) \frac{\lambda_j i}{2\pi(z(\alpha,t)-z_j(t))}.
\end{equation}
Since $\frac{1}{z(\alpha,t)-z_j(t)}$ is the boundary value of the holomorphic function $\frac{1}{z-z_j(t)}$ in $\Omega(t)^c$, by lemma \ref{holomorphic} again,  we have 
\begin{equation}\label{holo2}
(I-\mathfrak{H})\frac{1}{z(\alpha,t)-z_j(t)}=\frac{2}{z(\alpha,t)-z_j(t)}.
\end{equation}
(\ref{holo1}) together with (\ref{holo2}) complete the proof of the lemma.
\end{proof}

\begin{definition}[Double layer potential]
\begin{equation}
\mathfrak{K} f(\alpha):=p.v.\int_{-\infty}^{\infty} Re\{\frac{1}{\pi i} \frac{z_{\beta}}{z(\alpha,t)-z(\beta,t)}\} f(\beta)d\beta.
\end{equation}
\end{definition}

\begin{definition}[Adjoint of double layer potential]
\begin{equation}
\mathfrak{K}^{\ast} f(\alpha):=p.v.\int_{-\infty}^{\infty} Re\{-\frac{1}{\pi i} \frac{z_{\alpha}}{|z_{\alpha}|}\frac{|z_{\beta}|}{z(\alpha,t)-z(\beta,t)}\} f(\beta)d\beta.
\end{equation}
\end{definition}
\begin{lemma}\label{layer}
Let $z(\alpha)$ be a chord-arc curve such that 
\begin{equation}\label{lipchord}
\beta_0|\alpha-\beta|\leq |z(\alpha)-z(\beta)|\leq \beta_1|\alpha-\beta|,\quad \quad \forall~\alpha,\beta\in \mathbb{R}.
\end{equation}
Then we have 
\begin{equation}
\norm{ \mathfrak{H}f}_{L^2}\leq C(\beta_0, \beta_1)\|f\|_{L^2}.
\end{equation}
\begin{equation}
\norm{\mathfrak{K} f}_{L^2}\leq C(\beta_0,\beta_1)\|f\|_{L^2}.
\end{equation}
\begin{equation}
\norm{\mathfrak{K}^{\ast}f}_{L^2}\leq C(\beta_0,\beta_1)\|f\|_{L^2}.
\end{equation}
\begin{equation}
\norm{(I\pm \mathfrak{K})^{-1}f}_{L^2}\leq C(\beta_0,\beta_1)\|f\|_{L^2}.
\end{equation}
\begin{equation}\label{adjoint}
\|(I\pm \mathfrak{K}^{\ast})^{-1}f\|_{L^2}\leq C(\beta_0,\beta_1)\|f\|_{L^2}.
\end{equation}
\end{lemma}
\noindent For proof, see for example \cite{coifman1982integrale}, \cite{taylor2007tools}.

\subsection{Commutator estimates}
Denote
\begin{equation}
S_1(A,f)=p.v.\int \prod_{j=1}^m \frac{A_j(\alpha)-A_j(\beta)}{\gamma_j(\alpha)-\gamma_j(\beta)} \frac{f(\beta)}{\gamma_0(\alpha)-\gamma_0(\beta)}d\beta.
\end{equation}

\begin{equation}
S_2(A,f)=\int \prod_{j=1}^m \frac{A_j(\alpha)-A_j(\beta)}{\gamma_j(\alpha)-\gamma_j(\beta)} f_{\beta}(\beta)d\beta.
\end{equation}
\noindent We have the following comutator estimates, which can be found in \cite{Totz2012}, \cite{Wu2009}.
\begin{lemma}\label{singular}
(1) Assume each $\gamma_j$ satisfies the chord-arc condition
\begin{equation}\label{chordarc1}
C_{0,j}|\alpha-\beta|\leq |\gamma_j(\alpha)-\gamma_j(\beta)|\leq C_{1,j}|\alpha-\beta|.
\end{equation}
Then both $\norm{S_1(A,f)}_{L^2}$ and $\norm{S_2(A,f)}_{L^2}$ are bounded by
$$C\prod_{j=1}^m \norm{A_j'}_{X_j}\norm{f}_{X_0},$$
where one of the $X_0, X_1, ...X_m$ is equal to $L^2$ and the rest are $L^{\infty}$. The constant $C$ depends on $\norm{\gamma_j'}_{L^{\infty}}^{-1}, j=1,..,m$.

(2) Let $s\geq 3$ be given, and suppose chord-arc condition (\ref{chordarc1}) holds for each $\gamma_j$, then 
$$\norm{S_2(A,f)}_{H^s}\leq C\prod_{j=1}^m \norm{A_j'}_{Y_j}\norm{f}_Z,$$
where for all $j=1,...,m$, $Y_j=H^{s-1}$ or $W^{s-2,\infty}$ and $Z=H^s$ or $W^{s-1,\infty}$. The constant $C$ depends on $\norm{\gamma_j'}_{H^{s-1}}$, $\norm{\gamma_j}_{\infty}^{-1}, j=1,...,m$.
\end{lemma}
As a consequence of lemma \ref{singular}, we have the following commutator estimates.

\begin{lemma}\label{lemmacommutator1}
Let $k\geq 1$. Assume $z(\alpha,t)$ satisfies chord-arc condition
\begin{equation}\label{chordarc}
C_1|\alpha-\beta|\leq |z(\alpha,t)-z(\beta,t)|\leq C_2|\alpha-\beta|,
\end{equation}
and $z_{\alpha}-1\in H^{k-1}$. 
Then 
\begin{equation}
\norm{[\partial_{\alpha}^k, \mathfrak{H}]f}_{L^2}\leq C\|\partial_{\alpha}f\|_{H^{k-1}},
\end{equation}
where the constant $C=C(\|z_{\alpha}-1\|_{H^{k-1}}, C_1, C_2)$.
\end{lemma}
\begin{proof}
Use $$[\partial_{\alpha}^k, \mathfrak{H}]=\sum_{l=0}^k\partial_{\alpha}^l [\partial_{\alpha}, \mathfrak{H}]\partial_{\alpha}^{k-l}$$
Then use induction to complete the proof.
\end{proof}

\subsection{Some estimates involving point vortices} In this subsection we estimate some integrals involving the point vortices. 
\begin{lemma}\label{integral}
Assume $z(\alpha,t)$ satisfies the same condition as in lemma \ref{lemmacommutator1}.
Let $k>1$. Then 
\begin{equation}
\int_{-\infty}^{\infty}\frac{1}{|z_j(t)-z(\beta,t)|^k}d\beta\leq Cd_I(t)^{-k+1},
\end{equation}
where $C=4C_0^{-1}+\frac{4C_1^{k-1}}{(k-1)C_0^{k-1}}$.
\end{lemma}
\begin{proof}
We may assume that $d_I(t)=d(z_j(t), z(0,t))$. 
\begin{align*}
&\int_{-\infty}^{\infty}\frac{1}{|z_j(t)-z(\beta,t)|^k}d\beta\\
=& \int_{|z(0,t)-z(\beta,t)|\leq 2d_I(t)}\frac{1}{|z_j(t)-z(\beta,t)|^k}d\beta+\int_{|z(0,t)-z(\beta,t)|\geq 2d_I(t)}\frac{1}{|z_j(t)-z(\beta,t)|^k}d\beta\\
:=&  I+\it{II}.
\end{align*}
Denote $$E:=\{\beta: |z(0,t)-z(\beta,t)|\leq 2d_I(t)\}.$$
Since 
$$C_0|\beta-0|\leq |z(\beta,t)-z(0,t)|,$$
we have for $\beta\in E$,
$$|\beta-0|\leq \frac{2}{C_0}d_I(t).$$
Therefore
\begin{align*}
I\leq & 4C_0^{-1} d_I(t)^{-k+1}.
\end{align*}
For $\beta\in E^c$, use the chord-arc condition (\ref{chordarc}), we have 
\begin{equation}
|z(\beta,t)-z(0,t)-d_I(t)|\geq |z(\beta,t)-z(0,t)|-d_I(t)\geq \frac{1}{2}|z(\beta,t)-z(0,t)|\geq \frac{C_0}{2}|\beta-0|.
\end{equation}
Also, we have 
\begin{equation}
C_1|\beta-0|\geq |z(\beta,t)-z(0,t)|\geq 2d_I(t).
\end{equation}
So 
\begin{equation}
|\beta|\geq \frac{2}{C_1}d_I(t)
\end{equation}
Therefore, for $\it{II}$, we have 
\begin{align*}
\it{II}\leq & \frac{2^k}{C_0^k}\int_{|\beta|\geq \frac{2}{C_1}d_I(t) }|\beta|^{-k} d\beta=2\frac{2^k}{(k-1)C_0^k}\frac{C_1^{k-1}}{2^{k-1}}d_I(t)^{-k+1}=\frac{4C_1^{k-1}}{(k-1)C_0^{k-1}}d_I(t)^{-k+1}.
\end{align*}
\end{proof}

\begin{corollary}\label{integral2}
Assume $z(\alpha,t)$ satisfies the same condition as in lemma \ref{lemmacommutator1}. Given $m\geq 2$, there exist $C=(k+m)!C(C_0, C_1, \norm{z_{\alpha}-1}_{H^{m-1}})$ such that 
\begin{equation}
\norm{ \frac{1}{(z(\alpha,t)-z_j(t))^k}}_{H^m}\leq C(d_I(t)^{-k+1/2}+d_I(t)^{-k-m+1/2})
\end{equation}
In particular, if $d_I(t)\geq 1$, then we have 
\begin{equation}
\norm{ \frac{1}{(z(\alpha,t)-z_j(t))^k}}_{H^m}\leq Cd_I(t)^{-k+1/2}.
\end{equation}
\end{corollary}

\subsection{Basic identities}
\begin{lemma}\label{basic1}
Assume $z_t, z_{\alpha}-1\in C^1([0,T];H^1(\mathbb{R}))$, $f\in C(\mathbb{R}\times [0,T])$ and $f_{\alpha}(\alpha,t)\rightarrow 0$ as $|\alpha|\rightarrow \infty$. We have 
\begin{align}
[\partial_t,\mathfrak{H}]f=&[z_t,\mathfrak{H}]\frac{f_{t\alpha}}{z_{\alpha}}\\
[\partial_t^2,\mathfrak{H}]f=&2[z_{tt},\mathfrak{H}]\frac{f_{\alpha}}{z_{\alpha}}+2[z_t,\mathfrak{H}]\frac{f_{t\alpha}}{z_{\alpha}}-\frac{1}{\pi i}\int (\frac{z_t(\alpha,t)-z_t(\beta,t)}{z(\alpha,t)-z(\beta,t)})^2f_{\beta}(\beta,t)d\beta\\
[a\partial_{\alpha},\mathfrak{H}]f=&[az_{\alpha},\mathfrak{H}]\frac{f_{\alpha}}{z_{\alpha}},\quad \quad  \partial_{\alpha}\mathfrak{H}f=z_{\alpha}\mathfrak{H}\frac{f_{\alpha}}{z_{\alpha}}\\
[\partial_t^2-ia\partial_{\alpha},\mathfrak{H}]f=&2[z_{t},\mathfrak{H}]\frac{f_{t\alpha}}{z_{\alpha}}-\frac{1}{\pi i}\int (\frac{z_t(\alpha,t)-z_t(\beta,t)}{z(\alpha,t)-z(\beta,t)})^2f_{\beta}(\beta,t)d\beta \label{abcde}
\end{align}
\end{lemma}
\noindent For proof, see \cite{Wu2009}.
\begin{lemma}\label{commutehigh}
Let $D_t=\partial_t+b\partial_{\alpha}$, then 
\begin{equation}
[D_t^2, \partial_{\alpha}]=-D_t(b_{\alpha})\partial_{\alpha}-b_{\alpha}D_t\partial_{\alpha}-b_{\alpha}\partial_{\alpha}D_t
\end{equation}
\begin{equation}
\begin{split}
[D_t^2,\partial_{\alpha}^k]=&\sum_{m=0}^{k-1}\Big[\partial_{\alpha}^m(D_tb_{\alpha})\partial_{\alpha}^{k-m}+\partial_{\alpha}^m (b_{\alpha}\partial_{\alpha}^{k-m}D_t)+\partial_{\alpha}^m (b_{\alpha}[b\partial_{\alpha},\partial_{\alpha}^{k-m}])+\partial_{\alpha}^m b_{\alpha}\partial_{\alpha}^{k-m}D_t\\
&+\partial_{\alpha}^mb_{\alpha}[b\partial_{\alpha},b]\partial_{\alpha}^{k-m-1}\Big]
\end{split}
\end{equation}
\end{lemma}
\begin{proof}
It's easy to see that 
\begin{align*}
[D_t^2,\partial_{\alpha}^k]=&-\sum_{m=0}^{k-1}\Big[\partial_{\alpha}^m(D_tb_{\alpha})\partial_{\alpha}^{k-m}+\partial_{\alpha}^m (b_{\alpha}D_t\partial_{\alpha}^{k-m})+\partial_{\alpha}^mb_{\alpha}\partial_{\alpha}D_t\partial_{\alpha}^{k-m-1}\Big]\\
=&-\sum_{m=0}^{k-1}\Big[\partial_{\alpha}^m(D_tb_{\alpha})\partial_{\alpha}^{k-m}+\partial_{\alpha}^m (b_{\alpha}\partial_{\alpha}^{k-m}D_t)+\partial_{\alpha}^m (b_{\alpha}[b\partial_{\alpha},\partial_{\alpha}^{k-m}])+\partial_{\alpha}^m b_{\alpha}\partial_{\alpha}^{k-m}D_t\\
&+\partial_{\alpha}^mb_{\alpha}[b\partial_{\alpha},b]\partial_{\alpha}^{k-m-1}\Big]\\
\end{align*}
\end{proof}

\subsection{Preservation of symmetries.} The water waves with point vortices preserve the symmetry (H4). Such symmetry is well-known if there is no point vortex. 

\begin{lemma}[Preservation of symmetries]\label{symmetry}
Let $\Omega(0)\in \mathbb{R}^2$ be symmetric about $x=0$. Let $Re\{F\}$ be odd in $x$, and $Im\{F\}$ be even in $x$ at time $t=0$. 	Suppose the solution to the system (\ref{vortex_model}) exists on $[0,T_0]$. Then $Re\{F\}$ remains odd in $x$, and $Im\{F\}$ remains even in $x$ for all $t\in [0,T_0]$.
\end{lemma}
\noindent This is the consequence of the uniqueness of the solutions to equation (\ref{vortex_model}).

\section{Taylor sign condition}\label{sectionsign}
In this section we give a systematic study of the Taylor sign condition. We derive the formula (\ref{A1formula}), and then use this formula to show that the Taylor sign condition could fail if the point vortices are sufficiently close to the free interface. We also obtain a criterion for the Taylor sign condition to hold.

\subsection{The Taylor sign condition in Riemann variables}
Recall that 
$\bar{z}_t+\sum_{j=1}^N \frac{\lambda_j i}{2\pi(z(\alpha,t)-z_j(t))}$ is holomorphic, i.e., there is a holomorphic function $F(z,t)$ in $\Omega(t)$ such that $$F(z(\alpha,t),t)=\bar{z}_t(\alpha,t)+\sum_{j=1}^N \frac{\lambda_j i}{2\pi(z(\alpha,t)-z_j(t))}.$$
So we have 
\begin{equation}\label{calculateholo}
\begin{split}
\bar{z}_{tt}=& \partial_t \bar{z}_t=\partial_t \Big(F(z(\alpha,t),t)-\sum_{j=1}^N \frac{\lambda_j i}{2\pi(z(\alpha,t)-z_j(t))}\Big)\\
=& F_z(z(\alpha,t),t) z_t+F_t+\sum_{j=1}^N \frac{\lambda_j i(z_t(\alpha,t)-\dot{z}_j(t))}{2\pi (z(\alpha,t)-z_j(t))^2}
\end{split}
\end{equation}
Note that 
\begin{equation}
\dot{z}_j(t)=(v-\frac{\lambda_j i}{\overline{2\pi(z-z_j(t)})})\Big|_{z=z_j(t)}=\bar{F}(z_j(t),t)-\sum_{k: k\neq j}\frac{\lambda_k i}{2\pi\overline{z_k(t)-z_j(t)}}
\end{equation}
Let $\Phi(\cdot,t):\Omega(t)\rightarrow \mathbb{P}_-$ be the Riemann mapping such that $\Phi_z\rightarrow 1$ as $z\rightarrow \infty$. Let $h(\alpha,t):=\Phi(z(\alpha,t),t)$. Denote \footnote{In \S \ref{sectionlongtime}, we also use the notation $A, b, D_t$. We'd like the readers to keep in mind that they are not the same. In \S \ref{sectionlongtime}, $A=(a\kappa_{\alpha})\circ\kappa^{-1}, b=\kappa_t\circ\kappa^{-1}, D_t=\partial_t+\kappa_t\circ\kappa^{-1}\partial_{\alpha}$.}
\begin{equation}
Z(\alpha,t):=z\circ h^{-1}(\alpha,t),\quad \quad b=h_t\circ h^{-1}, \quad \quad D_t:=\partial_t+b\partial_{\alpha},\quad \quad 
\end{equation}
\begin{equation}
A:=(ah_{\alpha})\circ h^{-1}.\quad \quad 
\end{equation}
Use (\ref{calculateholo}), apply $h^{-1}$ on both sides of $\bar{z}_{tt}+ia\bar{z}_{\alpha}=i$, we obtain
\begin{equation}\label{transform}
\begin{split}
F_z\circ Z(\alpha,t)D_tZ+&F_t\circ Z(\alpha,t)+\sum_{j=1}^N \frac{\lambda_j i(D_tZ(\alpha,t)- \dot{z}_j(t))}{2\pi (Z(\alpha,t)-z_j(t))^2}+iA\bar{Z}_{\alpha}=i.
\end{split}
\end{equation}
Multiply by $Z_{\alpha}$ on both sides of (\ref{transform}), and denote
$$A_1:= A|Z_{\alpha}|^2,$$
we obtain
\begin{equation}
F_z\circ Z Z_{\alpha}D_tZ+F_t\circ Z Z_{\alpha}+\sum_{j=1}^N \frac{\lambda_j i D_t Z Z_{\alpha}-\lambda_j i\dot{z}_j(t)Z_{\alpha}}{2\pi(Z(\alpha,t)-z_j(t))^2}+iA_1=iZ_{\alpha}.
\end{equation}
Apply $I-\mathbb{H}$ on both sides of the above equation, then take imaginary parts, we obtain
\begin{equation}
A_1=1-Im\Big\{ (I-\mathbb{H})(F_z\circ Z Z_{\alpha} D_t Z)+(I-\mathbb{H})\sum_{j=1}^N \frac{\lambda_j i (D_t Z Z_{\alpha}-\dot{z}_j(t)Z_{\alpha})}{2\pi(Z(\alpha,t)-z_j(t))^2}\Big\}
\end{equation}
Note that 
$$F_z\circ Z Z_{\alpha}=\partial_{\alpha}(D_t\bar{Z}+\sum_{j=1}^N \frac{\lambda_j i}{2\pi(Z(\alpha,t)-z_j(t))})=\partial_{\alpha}D_t\bar{Z}-\sum_{j=1}^N\frac{\lambda_j i Z_{\alpha}}{2\pi(Z(\alpha,t)-z_j(t))^2}$$
is holomorphic. So we have 
\begin{equation}
\begin{split}
(I-\mathbb{H})F_z\circ Z Z_{\alpha} D_tZ=&[D_tZ,\mathbb{H}](\partial_{\alpha}D_t\bar{Z}-\sum_{j=1}^N\frac{\lambda_j i Z_{\alpha}}{2\pi(Z(\alpha,t)-z_j(t))^2})\\
=& [D_tZ,\mathbb{H}]\partial_{\alpha}D_t\bar{Z}-\sum_{j=1}^N[D_tZ,\mathbb{H}]\frac{\lambda_j i Z_{\alpha}}{2\pi(Z(\alpha,t)-z_j(t))^2})
\end{split}
\end{equation}
We know that 
\begin{equation}
-Im[D_tZ,\mathbb{H}]\partial_{\alpha}D_t\bar{Z}=\frac{1}{2\pi}\int \frac{|D_tZ(\alpha,t)-D_tZ(\beta,t)|^2}{(\alpha-\beta)^2}d\beta\geq 0.
\end{equation}
Also,
\begin{equation}
\begin{split}
&-\sum_{j=1}^N[D_tZ,\mathbb{H}]\frac{\lambda_j i Z_{\alpha}}{2\pi(Z(\alpha,t)-z_j(t))^2})+(I-\mathbb{H})\sum_{j=1}^N \frac{\lambda_j i (D_t Z Z_{\alpha}-\dot{z}_j(t)Z_{\alpha})}{2\pi(Z(\alpha,t)-z_j(t))^2}\\
=& \sum_{j=1}^N[D_tZ,I-\mathbb{H}]\frac{\lambda_j i Z_{\alpha}}{2\pi(Z(\alpha,t)-z_j(t))^2})+(I-\mathbb{H})\sum_{j=1}^N \frac{\lambda_j i (D_t Z Z_{\alpha}-\dot{z}_j(t)Z_{\alpha})}{2\pi(Z(\alpha,t)-z_j(t))^2}\\
=& \sum_{j=1}^N \frac{\lambda_j i}{2\pi} D_tZ(I-\mathbb{H})\frac{Z_{\alpha}}{(Z(\alpha,t)-z_j(t))^2}-\sum_{j=1}^N \frac{\lambda_j i}{2\pi} (I-\mathbb{H})\frac{\dot{z}_j(t) Z_{\alpha}}{(Z(\alpha,t)-z_j(t))^2}\\
=& \sum_{j=1}^N \frac{\lambda_j i}{2\pi} \Big((I-\mathbb{H})\frac{Z_{\alpha}}{(Z(\alpha,t)-z_j(t))^2}\Big)(D_tZ-\dot{z}_j(t))
\end{split}
\end{equation}
So we have 
\begin{equation}\label{AONE}
\begin{split}
A_1=&1+\frac{1}{2\pi}\int \frac{|D_tZ(\alpha,t)-D_tZ(\beta,t)|^2}{(\alpha-\beta)^2}d\beta-Im\Big\{\sum_{j=1}^N \frac{\lambda_j i}{2\pi} \Big((I-\mathbb{H})\frac{Z_{\alpha}}{(Z(\alpha,t)-z_j(t))^2}\Big)(D_tZ-\dot{z}_j(t)) \Big\}\\
=& 1+\frac{1}{2\pi}\int \frac{|D_tZ(\alpha,t)-D_tZ(\beta,t)|^2}{(\alpha-\beta)^2}d\beta-\sum_{j=1}^N \frac{\lambda_j}{2\pi} Re\Big\{\Big((I-\mathbb{H})\frac{Z_{\alpha}}{(Z(\alpha,t)-z_j(t))^2}\Big)(D_tZ-\dot{z}_j(t))\Big\}
\end{split}
\end{equation}
To get an estimate as sharp as possible for the Taylor sign condition, we'd like to get rid of the Hilbert transform $\mathbb{H}$ in the formula above.

It's easy to see that, if $Z=\alpha$, then $\frac{Z_{\beta}}{(Z(\beta,t)-z_j(t))^2}(D_tZ-\dot{z}_j(t))$ is boundary value of a holomorphic function in $\mathbb{P}_+$, so $$\Big((I-\mathbb{H})\frac{Z_{\alpha}}{(Z(\alpha,t)-z_j(t))^2}\Big)(D_tZ-\dot{z}_j(t))=2\frac{Z_{\alpha}}{(Z(\alpha,t)-z_j(t))^2}(D_tZ-\dot{z}_j(t))$$
For the general case, we use the following lemma.
\begin{lemma}
Let $z_0\in \Omega(t)$.  Then 
\begin{equation}
(I-\mathbb{H})\frac{1}{Z(\alpha,t)-z_0}=\frac{2}{c_1(\alpha-w_0)},\quad \quad c_1=(\Phi^{-1})_z(w_0),\quad \quad w_0=\Phi(z_0,t).
\end{equation}
\end{lemma}
\begin{proof}
Note that $Z(\alpha,t)=\Phi^{-1}(\alpha,t)$. So $Z(\alpha,t)-z_0$ is the boundary value of $\Phi^{-1}(z,t)-z_0$ in the lower half plane. Since $\Phi^{-1}$ is 1-1 and onto, $\Phi^{-1}(z,t)-z_0$ has a unique  zero $w_0:=\Phi(z_0)$, so $\frac{1}{Z(\alpha,t)-z_0}$ has a exactly one pole of multiplicity one. For $z$ near $w_0$, we have 
\begin{equation}
\Phi^{-1}(z,t)-z_0=c_1(z-w_0)+\sum_{n=2}^{\infty}c_n(z-w_0)^n,\quad \quad where\quad c_1=(\Phi^{-1})_z(w_0)\neq 0.
\end{equation}
Therefore, we have $\frac{1}{Z(\alpha,t)-z_0}-\frac{1}{c_1(\alpha-w_0)}$ is holomorphic in $\mathbb{P}_-$, and hence 
\begin{equation}
(I-\mathbb{H})(\frac{1}{Z(\alpha,t)-z_0}-\frac{1}{c_1(\alpha-w_0)})=0.
\end{equation}
Since $\frac{1}{c_1(\alpha-w_0)}$ is holomorphic in $\mathbb{P}_+$, we obtain
\begin{equation}
(I-\mathbb{H})\frac{1}{Z(\alpha,t)-z_0}=(I-\mathbb{H})\frac{1}{c_1(\alpha-w_0)}=\frac{2}{c_1(\alpha-w_0)}.
\end{equation}
\end{proof}

\begin{corollary}
We have 
\begin{equation}
(I-\mathbb{H})\frac{Z_{\alpha}}{(Z(\alpha,t)-z_j(t))^2}=\frac{2}{(\Phi^{-1})_z(\Phi(z_j(t)))(\alpha-\Phi(z_j(t)))^2}
\end{equation}
\end{corollary}
\begin{proof}
We have 
\begin{align*}
(I-\mathbb{H})\frac{Z_{\alpha}}{(Z(\alpha,t)-z_j(t))^2}=-&\partial_{\alpha}(I-\mathbb{H})\frac{1}{Z(\alpha,t)-z_j(t)}\\
=&-\partial_{\alpha}\frac{2}{(\Phi^{-1})_z(\Phi(z_j(t)))(\alpha-\Phi(z_j(t)))}\\
=&\frac{2}{(\Phi^{-1})_z(\Phi(z_j(t)))(\alpha-\Phi(z_j(t)))^2}.
\end{align*}
\end{proof}
\begin{corollary}\label{corollaryformulaA1}
We have 
\begin{equation}\label{formulaA1}
A_1=1+\frac{1}{2\pi}\int \frac{|D_tZ(\alpha,t)-D_tZ(\beta,t)|^2}{(\alpha-\beta)^2}d\beta-\sum_{j=1}^N \frac{\lambda_j}{\pi} Re\Big\{\frac{D_tZ-\dot{z}_j}{c_0^j(\alpha-w_0^j)^2}\Big\},
\end{equation}
where 
\begin{equation}
c_0^j=(\Phi^{-1})_z(\omega_0^j),\quad \quad \omega_0^j=\Phi(z_j).
\end{equation}
\end{corollary}

\subsection{\textbf{A formula for $A_1$ when $Z(\alpha,t)=\alpha$, $D_tZ=\sum_{j=1}^N\frac{\lambda_j i}{2\pi}\frac{1}{\overline{\alpha-z_j(t)}}$}}
If the point vortices are very close to the interface, then Taylor sign condition can fail. To see this, we study the special case when $Z(\alpha,t)=\alpha$ and $D_tZ(\alpha,t)=\sum_{j=1}^N\frac{\lambda_j i}{2\pi}\frac{1}{\overline{\alpha-z_j(t)}}$. Since the integral term of the formula (\ref{formulaA1}) is nonlocal, in order to obtain a more convenient form of (\ref{formulaA1}), we use residue theorem to calculate this integral. We'll use the following formula.
\begin{lemma}\label{rational}
Let $w_1, w_2\in \mathbb{P}_-$. Then 
\begin{equation}
\int_{-\infty}^{\infty}\frac{1}{(\beta-w_1)(\beta-\overline{w_2})}d\beta=\frac{2\pi i}{\overline{w_2}-w_1}
\end{equation}
\end{lemma}
\begin{proof}
$\overline{w_2}$ is the only residue of $\frac{1}{(\beta-w_1)(\beta-\overline{w_2})}$ in $\mathbb{P}_+$. By residue Theorem, 
$$\int_{-\infty}^{\infty}\frac{1}{(\beta-w_1)(\beta-\overline{w_2})}d\beta=\frac{2\pi i}{\overline{w_2}-w_1}.$$
\end{proof}
As a consequence, we have 
\begin{corollary}
Assume $Z(\alpha,t)=\alpha$, $D_tZ=\sum_{j=1}^N\frac{\lambda_j i}{2\pi}\frac{1}{\overline{\alpha-z_j(t)}}$. We have 
\begin{equation}
\frac{1}{2\pi}\int \frac{|D_tZ(\alpha,t)-D_tZ(\beta,t)|^2}{(\alpha-\beta)^2}d\beta=\sum_{1\leq j,k\leq N}\frac{\lambda_j\lambda_k}{(2\pi)^2}\frac{1}{(\alpha-z_j)\overline{(\alpha-z_k)}} \frac{i}{\overline{z_k}-z_j}.
\end{equation}
\end{corollary}
\begin{proof}
We have
\begin{equation}
\begin{split}
D_tZ(\alpha,t)-D_t(\beta,t)=\sum_{j=1}^N \frac{\lambda_j i}{2\pi}\frac{\beta-\alpha}{\overline{(\alpha-z_j)(\beta-z_j)}}.
\end{split}
\end{equation}
So we have 
\begin{equation}
\begin{split}
&\Big| \frac{D_tZ(\alpha,t)-D_t(\beta,t)}{\alpha-\beta}\Big|^2=\Big|\sum_{j=1}^N\frac{\lambda_j}{2\pi}\frac{1}{(\alpha-z_j)(\beta-z_j)}\Big|^2\\
=&\sum_{j=1}^N\sum_{k=1}^N\frac{\lambda_j\lambda_k}{(2\pi)^2}\frac{1}{(\alpha-z_j)(\beta-z_j)\overline{(\alpha-z_k)(\beta-z_k)}}\\
\end{split}
\end{equation}
Apply lemma \ref{rational}, we have 
\begin{align*}
\int_{-\infty}^{\infty}\frac{1}{(\alpha-z_j)(\beta-z_j)\overline{(\alpha-z_k)(\beta-z_k)}}d\beta=\frac{1}{(\alpha-z_j)\overline{(\alpha-z_k)}}\frac{2\pi i}{\overline{z_k}-z_j}.
\end{align*}
So we have 
\begin{align*}
\frac{1}{2\pi}\int \frac{|D_tZ(\alpha,t)-D_tZ(\beta,t)|^2}{(\alpha-\beta)^2}d\beta= & \sum_{1\leq j,k\leq N}\frac{\lambda_j\lambda_k}{(2\pi)^3}\frac{1}{(\alpha-z_j)\overline{(\alpha-z_k)}} \frac{2\pi i}{\overline{z_k}-z_j}\\
=& \sum_{1\leq j,k\leq N}\frac{\lambda_j\lambda_k}{(2\pi)^2}\frac{1}{(\alpha-z_j)\overline{(\alpha-z_k)}} \frac{i}{\overline{z_k}-z_j}.
\end{align*}
\end{proof}

\begin{corollary}\label{goodformula}
Assume $Z(\alpha,t)=\alpha$, $D_tZ=\sum_{j=1}^N\frac{\lambda_j i}{2\pi}\frac{1}{\overline{\alpha-z_j(t)}}$. Then 
\begin{equation}
A_1=1+\sum_{1\leq j,k\leq N}\frac{\lambda_j\lambda_k}{(2\pi)^2}\frac{1}{(\alpha-z_j)\overline{(\alpha-z_k)}} \frac{i}{\overline{z_k}-z_j}-\sum_{j=1}^N \frac{\lambda_j}{\pi} Re\Big\{\frac{D_tZ-\dot{z}_j}{(\alpha-z_j)^2}\Big\}.
\end{equation}
\end{corollary}

In the following two subsections, we use Corollary \ref{goodformula} to construct examples for which the Taylor sign condition fails.
\subsection{One point vortex: An example that the Taylor condition fails}\label{examplesectionone}
We have the following characterization. Recall that $\frac{A_1}{|Z_{\alpha}|}=-\frac{\partial P}{\partial\bold{n}}$.
\begin{proposition}\label{exampleone}
Assume that at time $t$, the interface is $\Sigma(t)=\mathbb{R}$, the fluid velocity is $v(z,t)=\frac{\lambda i}{2\pi}\frac{1}{\overline{z-z_1(t)}}$, i.e., it is generated by a single point vortex $z_1(t):=x(t)+iy(t)$. 
Then
\begin{itemize}
\item [(1)] If $\frac{\lambda^2}{|y|^3}<\frac{8\pi^2}{3}$, strong Taylor sign condition holds. We have 
\begin{equation}
\inf_{\alpha\in \mathbb{R}}A_1(\alpha,t)\geq 1-\frac{3}{8\pi^2}\frac{\lambda^2}{|y|^3}>0.
\end{equation}
\item [(2)] If $\frac{\lambda^2}{|y|^3}>\frac{8\pi^2}{3}$,  Taylor sign condition fails, i.e., there exists $\alpha\in \mathbb{R}$ such that $A_1(\alpha,t)<0$.
\item [(3)] If $\frac{\lambda^2}{|y|^3}=\frac{8\pi^2}{3}$, 'Degenerate Taylor sign condition' holds, i.e., 
\begin{equation}
A_1(\alpha,t)>0, \quad \forall~\alpha\neq x(t);\quad \quad and~~ A_1(x(t),t)=0.
\end{equation}
\end{itemize}
\end{proposition}
\begin{remark}
The quantity $\frac{\lambda^2}{|y|^3}$ is a measurement of the interface-vortex interaction. $\lambda$ is the intensity of the point vortex, and $|y|$ is the distance from the point vortex to the interface.
\end{remark}

\begin{proof}
Note that 
$$\dot{z}_1(t)=0.$$
So we have 
\begin{equation}\label{badpart}
\begin{split}
-\sum_{j=1}^N \frac{\lambda_j}{\pi} Re\Big\{\frac{D_tZ-\dot{z}_j}{(\alpha-z_j)^2}\Big\}=&-\frac{\lambda}{\pi}Re\Big\{ \frac{D_tZ(\alpha,t)}{(\alpha-z_1(t))^2}\Big\}\\
=&-\frac{\lambda}{\pi}Re\Big\{ \frac{1}{(\alpha-z_1(t))^2}\frac{\lambda i}{2\pi\overline{\alpha-z_1(t)}}\Big\}\\
=& \frac{\lambda^2}{2\pi^2} \frac{y}{|\alpha-z_1(t)|^4}.
\end{split}
\end{equation}

Therefore, by Corollary \ref{goodformula}, we have 
\begin{align*}
A_1(\alpha,t)=&1+\frac{\lambda^2}{4\pi^2}\frac{1}{|\alpha-z_1|^2}\frac{i}{-2iy}+\frac{\lambda^2}{2\pi^2} \frac{y}{|\alpha-z_1(t)|^4}.
\end{align*}
Without loss of generality, we can assume $x=0$. Setting $\partial_{\alpha}A_1(\alpha,t)=0$, it's easy to see that $A_1(\alpha,t)$ admits a unique local minimum at $\alpha=x=0$. Moreover, it's easy to see that 
\begin{equation}
    \lim_{\alpha\rightarrow \pm\infty}A_1(\alpha,t)=1.
\end{equation}
Therefore, $\inf_{\alpha\in \mathbb{R}}A(\alpha,t)\leq 0$ if and only if $A_1(0,t)\leq 0$.

We have 
$$A_1(0,t)=1-\frac{\lambda^2}{8\pi^2}\frac{1}{y^3}+\frac{\lambda^2}{2\pi^2}\frac{1}{y^3}=1-\frac{3\lambda^2}{8\pi^2}\frac{1}{|y|^3}.$$
If $\frac{\lambda^2}{|y|^3}<\frac{8\pi^2}{3}$, then $A(\alpha,t)\geq A(0,t)>0$. If $\frac{\lambda^2}{|y|^3}>\frac{8\pi^2}{3}$, then $A_1(0,t)<0$. If $\frac{\lambda^2}{|y|^3}=\frac{8\pi^2}{3}$, then $A_1(0,t)=0$, and for $\alpha\neq 0$, $A_1(\alpha,t)>A_1(0,t)=0$.
\end{proof}

\subsection{Two point vortices: Another example that Taylor sign condition fails}\label{examplesection2} We show that if the point vortices are too close to the interface, then the Taylor sign condition fails. 
  
Assume that $Z(\alpha,t)=\alpha$. Assume $D_tZ=\sum_{j=1}^2 \frac{\lambda_j i}{2\pi}\frac{1}{\overline{\alpha-z_j(t)}}$. Assume $z_1(t)=-x(t)+iy(t)$, $z_2(t)=x(t)+iy(t)$, with $x(t)>0$, $y(t)<0$, and $\lambda_1=-\lambda_2:=\lambda$. Let's  calculate $A_1(0,t)$.

We have 

\begin{align*}
D_tZ(\alpha,t)=\sum_{j=1}^2\frac{\lambda_j i}{2\pi}\frac{1}{\overline{\alpha-z_j(t)}}=\frac{\lambda i}{2\pi}\frac{\overline{z_1(t)-z_2(t)}}{\overline{(\alpha-z_1(t))(\alpha-z_2(t))}}
\end{align*}
Since $z_1-z_2=-2x$,
$$D_tZ(0,t)=\frac{\lambda i}{\pi}\frac{x}{x^2+y^2}.$$
At $\alpha=0$, we have 
\begin{align*}
\sum_{j=1}^2 \frac{\lambda_j }{\pi}\frac{1}{(0-z_j(t))^2}=\frac{\lambda}{\pi}\frac{4xyi}{(x^2+y^2)^2}.
\end{align*}
So 
\begin{equation}
-\sum_{j=1}^2 \frac{\lambda_j }{\pi}\frac{1}{(0-z_j(t))^2}D_tZ(0)=-\frac{\lambda i}{\pi}\frac{x}{x^2+y^2}\frac{\lambda}{\pi}\frac{4xyi}{(x^2+y^2)^2}=\frac{4\lambda^2 x^2y}{\pi^2 (x^2+y^2)^3}
\end{equation}
We have 
\begin{align*}
\dot{z}_1=\frac{-\lambda i}{2\pi}\frac{1}{\overline{z_1-z_2}}=\frac{\lambda i}{4\pi}\frac{1}{x}.
\end{align*}
\begin{align*}
\dot{z}_2=\frac{\lambda i}{2\pi}\frac{1}{\overline{z_2-z_1}}=\frac{\lambda i}{4\pi}\frac{1}{x}.
\end{align*}
So we have 
\begin{align*}
\sum_{j=1}^2 \frac{\lambda_j}{\pi}\frac{1}{(0-z_j)^2}\dot{z}_j(t)=\frac{\lambda i}{4\pi}\frac{1}{x}\sum_{j=1}^2 \frac{\lambda_j}{\pi}\frac{1}{(0-z_j)^2}=\frac{\lambda i}{4\pi x}\frac{\lambda}{\pi}\frac{4xyi}{(x^2+y^2)^2}=-\frac{\lambda^2 y}{\pi^2(x^2+y^2)^2}.
\end{align*}
So we have 
$$-\sum_{j=1}^N \frac{\lambda_j}{\pi} Re\Big\{ \frac{1}{(0-z_j(t))^2}(D_tZ(0,t)-\dot{z}_j(t))\Big\}=\frac{\lambda^2y(3x^2-y^2)}{\pi^2(x^2+y^2)^3}.$$
On the other hand, we have 
\begin{align*}
&\sum_{1\leq j,k\leq 2}\frac{\lambda_j\lambda_k}{(2\pi)^2}\frac{1}{(0-z_j)\overline{(0-z_k)}} \frac{i}{\overline{z_k}-z_j}\\
=&\frac{\lambda^2}{4\pi^2}\frac{1}{|z_1|^2}\frac{i}{\bar{z}_1-z_1}+\frac{\lambda^2}{4\pi^2}\frac{1}{|z_2|^2}\frac{i}{\bar{z}_2-z_2}-\frac{\lambda^2}{4\pi^2}\frac{1}{z_1\bar{z}_2}\frac{i}{\bar{z}_2-z_1}-\frac{\lambda^2}{4\pi^2}\frac{1}{z_2\bar{z}_1}\frac{i}{\bar{z}_1-z_2}\\
=& \frac{\lambda^2}{4\pi^2}\frac{1}{|z_1|^2}\frac{1}{|y|}+2Re\Big\{ \frac{\lambda^2}{8\pi^2}\frac{i}{(x-iy)^3}\Big\}\\
=& \frac{\lambda^2}{4\pi^2}\frac{1}{x^2+y^2}\frac{1}{|y|}+\frac{\lambda^2}{4\pi^2}\frac{y^3-3x^2y}{(x^2+y^2)^3}\\
=& \frac{\lambda^2}{4\pi^2}\frac{x^4+5x^2y^2}{|y|(x^2+y^2)^3}.
\end{align*}
Therefore, by Corollary \ref{goodformula}, we have 
\begin{equation}
\begin{split}
A_1(0,t)=&1+\frac{\lambda^2}{4\pi^2}\frac{x^4+5x^2y^2}{|y|(x^2+y^2)^3}+\frac{\lambda^2y(3x^2-y^2)}{\pi^2(x^2+y^2)^3}\\
=&1+\frac{\lambda^2}{4\pi^2}\frac{x^4+5x^2y^2-12x^2y^2+4y^4}{|y|(x^2+y^2)^3}\\
=&1+\frac{\lambda^2}{4\pi^2}\frac{x^4+4y^4-7x^2y^2}{|y|(x^2+y^2)^3}
\end{split}
\end{equation}
So we have 
\begin{proposition}\label{twopoint}
Assume that at time $t$, the interface is $\Sigma(t)=\mathbb{R}$, the fluid velocity is $v(z,t)=\sum_{j=1}^2\frac{\lambda_j i}{2\pi}\frac{1}{\overline{z-z_j(t)}}$, where $\lambda_1=-\lambda_2:=\lambda$. Assume 
$$z_1(t)=-x+iy,\quad \quad z_2(t)=x+iy,\quad where\quad x>0, y<0.$$
If 
\begin{equation}
1+\frac{\lambda^2}{4\pi^2}\frac{x^4+4y^4-7x^2y^2}{|y|(x^2+y^2)^3}<0,
\end{equation}
then Taylor sign condition fails.
\end{proposition}

\begin{corollary}\label{x=y}
Under the assumption of Proposition \ref{twopoint}, if $|x|=|y|$ and $\frac{\lambda^2}{|y|^3}>16\pi^2$, then the strong Taylor sign condition fails.
\end{corollary}

\subsection{A criterion that implies the strong Taylor sign condition} If the vortex-vortex, vortex-interface interaction is weak, then the Taylor sign condition holds. Let's recall that we denote $F$ by
\begin{equation}
\bar{F}(z,t):=v(z,t)-\sum_{j=1}^N \frac{\lambda_j i}{2\pi}\frac{1}{\overline{z-z_j(t)}}.
\end{equation}
We have the following.
\begin{proposition}\label{almostsharp}
Assume $\inf_{\alpha\in \mathbb{R}}|Z_{\alpha}|= \beta_0$, $\|F\|_{\infty}= M_0$. Denote $$\tilde{\lambda}=:\frac{\sum_{j=1}^N|\lambda_j|}{\pi},\quad \quad \tilde{d}_I(t):=\min_{1\leq j\leq N}\inf_{\alpha\in \mathbb{R}}|\alpha-\Phi(z_j(t))|,\quad \quad \tilde{d}_P(t)=\min_{j\neq k}|z_j(t)-z_k(t)|. $$
If 
\begin{equation}\label{taylorassumption}
\frac{\tilde{\lambda}^2}{2\tilde{d}_I(t)^3\beta_0}+\frac{\tilde{\lambda}^2}{2\tilde{d}_I(t)^2\tilde{d}_P(t)}+\frac{2M_0\tilde{\lambda}}{\tilde{d}_I(t)^2}< \beta_0,
\end{equation}
then the strong Taylor sign condition holds.
\end{proposition}

\begin{proof}
Use formula
\begin{equation}
A_1=1+\frac{1}{2\pi}\int \frac{|D_tZ(\alpha,t)-D_tZ(\beta,t)|^2}{(\alpha-\beta)^2}d\beta-\sum_{j=1}^N \frac{\lambda_j}{\pi} Re\Big\{ \frac{D_tZ-\dot{z}_j}{c_0^j(\alpha-\omega_0^j)^2}\Big\}
\end{equation}
For $\sum_{j=1}^N \frac{\lambda_j}{\pi} Re\Big\{ \frac{D_tZ-\dot{z}_j}{c_0^j(\alpha-\omega_0^j)^2}\Big\}$, we have 
\begin{align*}
\Big|\sum_{j=1}^N \frac{\lambda_j}{\pi} Re\Big\{ \frac{D_tZ-\dot{z}_j}{c_0^j(\alpha-\omega_0^j)^2}\Big\}\Big|\leq & \frac{\sum_{j=1}^N |\lambda_j| }{\pi}\max_{1\leq j\leq N}|c_0^j|^{-1}(\|D_tZ\|_{\infty}+|\dot{z}_j|)(\inf_{\alpha\in \mathbb{R}}|\alpha-\Phi(z_j)|)^{-2}
\end{align*}
Since $Z(\alpha)=\Phi^{-1}(\alpha)$,  we have $\partial_{\alpha}\Phi^{-1}(\alpha)=Z_{\alpha}$ and $\partial_{\alpha}\Phi^{-1}(\alpha)$ is the boundary value of $(\Phi^{-1})_z$. Note that $(\Phi^{-1})_z$ never vanishes. By maximum modules principle of holomorphic functions (apply to $\frac{1}{(\Phi^{-1})_z}$), 
\begin{equation}
|c_0^j|=|(\Phi^{-1})_z(\Phi(z_j))|\geq \inf_{\alpha\in \mathbb{R}}|Z_{\alpha}|\geq \beta_0.
\end{equation}
Since 
\begin{equation}
    Z(\alpha,t)-z_j(t)=\Phi^{-1}(\alpha,t)-\Phi^{-1}(\omega_0^j)=\Phi^{-1}_z(z')(\alpha-\omega_0^j)
\end{equation} for some $z'\in \mathbb{P}_-$, so we have 
\begin{equation}
    |Z(\alpha,t)-z_j(t)|\geq \beta_0|\alpha-\omega_0^j|.
\end{equation}
Therefore,
\begin{align*}
|D_tZ|\leq |F|+\sum_{j=1}^N \frac{|\lambda_j|}{2\pi \beta_0}\frac{1}{|\alpha-\omega_0^j|}= M_0+\frac{\tilde{\lambda}}{ 2\tilde{d}_I(t)\beta_0}.
\end{align*}
similarly, 
\begin{align*}
|\dot{z}_j(t)|=|\bar{F}+\sum_{k\neq j}\frac{\lambda_k i}{2\pi}\frac{1}{\overline{z_j(t)-z_k(t)}}|\leq M_0+\frac{\sum_{j=1}^N|\lambda_j|}{2\pi} \tilde{d}_P(t)^{-1}=M_0+\frac{\tilde{\lambda}}{ 2\tilde{d}_P(t)}.
\end{align*}
So we obtain
\begin{align*}
|\sum_{j=1}^N \frac{\lambda_j}{\pi} Re\Big\{ \frac{D_tZ-\dot{z}_j}{c_0^j(\alpha-\omega_0^j)^2}\Big\}|\leq & \beta_0^{-1}\frac{\tilde{\lambda}}{\tilde{d}_I(t)^2}(M_0+\frac{\tilde{\lambda}}{ 2\tilde{d}_I(t)\beta_0}+M_0+\frac{\tilde{\lambda}}{ 2\tilde{d}_P(t)})\\
\leq & \beta_0^{-1}(\frac{\tilde{\lambda}^2}{2\tilde{d}_I(t)^3\beta_0}+\frac{\tilde{\lambda}^2}{2\tilde{d}_I(t)^2\tilde{d}_P(t)}+\frac{2M_0\tilde{\lambda}}{\tilde{d}_I(t)^2}).
\end{align*}
If (\ref{taylorassumption}) holds, then 
$$|\sum_{j=1}^N \frac{\lambda_j}{\pi} Re\Big\{ \frac{D_tZ-\dot{z}_j}{c_1^j(\alpha-\omega_0^j)^2}\Big\}|<1.$$
Then $A_1>0$, so strong Taylor sign condition holds. 
\end{proof}
In particular, if $Z_{\alpha}\sim 1$, $d_I(t)\gtrsim 1$, $M_0\ll 1$, and $|\lambda|\ll 1$, then the strong Taylor sign condition holds.

\section{Local wellposedness: proof of Theorem \ref{theorem1}}\label{sectionlocal}

In this section we prove Theorem \ref{theorem1}, i.e., prove local wellposedness of water waves with general $N$ point vortices. As was explained in the introduction, our strategy is to quasilinearize the system (\ref{vortex_boundary}) by taking one time derivative of the momentum equation $(\partial_t^2+ia\partial_{\alpha})\bar{z}=i$, and then obtain a closed energy estimate. 

Recall that we assume
\begin{equation}
\inf_{\alpha\in \mathbb{R}}a(\alpha,0)|z_{\alpha}(\alpha,0)|\geq \alpha_0>0.
\end{equation}
\begin{equation}
C_1|\alpha-\beta|\leq |z(\alpha,0)-z(\beta,0)|\leq C_2|\alpha-\beta|.
\end{equation}
Let $T_0\geq 0$, we make the following a priori assumptions:
\begin{equation}\label{apriori}
\inf_{t\in [0,T_0]}\inf_{\alpha\in \mathbb{R}} a(\alpha,t)|z_{\alpha}(\alpha,t)|\geq \frac{\alpha_0}{2}, \quad \quad  \frac{1}{2}C_1|\alpha-\beta|\leq |z(\alpha,0)-z(\beta,0)|\leq 2C_2|\alpha-\beta|,
\end{equation}
and
\begin{equation}\label{apriori2}
    \sup_{t\in [0,T_0]}\|z_{tt}(\cdot, t)\|_{H^1}\leq 2\|w_0\|_{H^1}.
\end{equation}
Without loss of generality, we assume $T_0\leq 1$.
\begin{remark}
The a priori assumptions (\ref{apriori}) and (\ref{apriori2}) hold at $t=0$. For $0<t\leq T_0$, (\ref{apriori}) and (\ref{apriori2}) will be justified by a bootstrap argument. 
\end{remark}
\subsection{Velocity and acceleration of the point vortices.} For water waves with point vortices, the motion of the point vortices affects the dynamics of the water waves in a fundamental way, so we need to have a good understanding of the velocity and acceleration of the point vortices.
We decompose the velocity field $\bar{v}$ as 
\begin{equation}
\bar{v}(z,t)=F(z,t)-\sum_{j=1}^N\frac{\lambda_j i}{2\pi}\frac{1}{z-z_j(t)}.
\end{equation}
So $F$ is holomorphic in $\Omega(t)$. We have the following estimate for the velocity and acceleration of the point vortices.

\begin{lemma}\label{velocity}
Assume that the assumptions of Theorem \ref{theorem1} hold, and assume the a priori assumptions (\ref{apriori}) and (\ref{apriori2}). Then 
\begin{equation}
|\dot{z}_j(t)|+|\ddot{z}_j(t)|\leq C(N\lambda_{max}, \|z_t\|_{H^2}, \|z_{tt}\|_{H^1},d_I(t)^{-1} , d_P(t)^{-1}, C_1).\end{equation}
 where 
\begin{equation}
\lambda_{max}=\max_{1\leq j\leq N} |\lambda_j|,
\end{equation}
and $C: (\mathbb{R}_+\cup \{0\})^6\rightarrow \mathbb{R}_+\cup\{0\}$ is a polynomial with positive coefficients. 
\end{lemma}

\begin{proof} The main tool is the maximum principle of holomorphic functions.

\vspace*{2ex}

\noindent \underline{\textbf{Estimate $\dot{z}_j$:  }}
By (\ref{velocity}), we have 
\begin{equation}\label{movement1}
\begin{split}
\dot{z}_j(t)=&(v-\frac{\lambda_j i}{\overline{z-z_j(t)}})\Big|_{z=z_j(t)}
=\sum_{1\leq k\leq N, k\neq j}\frac{\lambda_k i}{2\pi(\overline{z_j(t)-z_k(t)})}+\bar{F}(z_j(t)).
\end{split}
\end{equation}
Note that $\bar{F}$ is an anti-holomorphic function with boundary value $z_t-\sum_{j=1}^N\frac{\lambda_j i}{2\pi}\frac{1}{\overline{z(\alpha,t)-z_j(t))}}$, by maximum principle, we have 
\begin{equation}
|\bar{F}(z_j(t), t)|\leq \norm{F(\cdot, t)}_{L^{\infty}(\Sigma(t))}=\norm{z_t-\sum_{j=1}^N\frac{\lambda_j i}{2\pi}\frac{1}{\overline{z(\alpha,t)-z_j(t))}}}_{\infty}.
\end{equation}
By Triangle inequality, we obtain
\begin{equation}
\begin{split}
|\dot{z}_j(t)|\leq & \|z_t\|_{\infty}+\norm{\sum_{k\neq j}\frac{\lambda_k i}{2\pi(\overline{z_j(t)-z_k(t)})}}_{\infty}+\norm{\sum_{k=1}^N \frac{\lambda_k i}{2\pi\overline{z(\alpha,t)-z_k(t)}}}_{\infty}\\
\leq &\|z_t\|_{H^1}+  N \lambda_{max} (d_P(t)^{-1}+d_I(t)^{-1}).
\end{split}
\end{equation}

\vspace*{2ex}

\noindent \underline{\textbf{Estimate $\ddot{z}_j(t)$:  }}
Take time derivative of both sides of $\dot{z}_j(t)=\sum_{k:k\neq j}\frac{\lambda_k i}{2\pi(\overline{z_j(t)-z_k(t)})}+\bar{F}(z_j(t),t)$, we obtain
\begin{equation}
\ddot{z}_j(t)=-\sum_{k:k\neq j}\frac{\lambda_k i\overline{\dot{z}_j(t)-\dot{z}_k(t)}}{2\pi(\overline{z_j(t)-z_k(t)})^2}+\bar{F}_z(z_j(t),t)\dot{z}_j(t)+\bar{F}_t(z_j(t),t)
\end{equation}
The boundary value of $F_z$ is 
\begin{equation}\label{Fz}
F_z(z(\alpha,t),t)=\frac{\partial_{\alpha}F(z(\alpha,t),t)}{z_{\alpha}}=\frac{\bar{z}_{t\alpha}}{z_{\alpha}}-\sum_{j=1}^N\frac{\lambda_j i}{2\pi}\frac{1}{(z(\alpha,t)-z_j(t))^2}.
\end{equation}
So 
\begin{equation}
|F_z(z(\alpha,t),t)|\leq \norm{z_{t\alpha}}_{\infty}\norm{\frac{1}{z_{\alpha}}}_{\infty}+N\lambda_{max} d_I(t)^{-2}.
\end{equation}
The boundary value of $F_t$ is 
\begin{equation}\label{Ft}
\begin{split}
F_t(z(\alpha,t),t)=&\partial_t F(z(\alpha,t),t)-\frac{\partial_{\alpha}F(z(\alpha,t),t)}{z_{\alpha}}z_t\\
=&\partial_t \Big(\bar{z}_t+\sum_{j=1}^N \frac{\lambda_j i}{2\pi} \frac{1}{z(\alpha,t)-z_j(t)}\Big)-\frac{\partial_{\alpha}(\bar{z}_t+\sum_{j=1}^N \frac{\lambda_j i}{2\pi} \frac{1}{z(\alpha,t)-z_j(t)}}{z_{\alpha}}z_t\\
=&\bar{z}_{tt}-\sum_{j=1}^N \frac{\lambda_j i}{2\pi}\frac{z_t-\dot{z}_j(t)}{(z(\alpha,t)-z_j(t))^2}-\frac{\bar{z}_{t\alpha}}{z_{\alpha}}z_t+\sum_{j=1}^N \frac{\lambda_j i}{2\pi}\frac{z_t}{(z(\alpha,t)-z_j(t))^2}
\end{split}
\end{equation}
By maximum principle, 
\begin{align*}
\|F_t\|_{\infty}=& \norm{\bar{z}_{tt}-\sum_{j=1}^N \frac{\lambda_j i}{2\pi}\frac{z_t-\dot{z}_j(t)}{(z(\alpha,t)-z_j(t))^2}-\frac{\bar{z}_{t\alpha}}{z_{\alpha}}z_t+\sum_{j=1}^N \frac{\lambda_j i}{2\pi}\frac{z_t}{(z(\alpha,t)-z_j(t))^2}}_{\infty}\\
\leq & \norm{z_{tt}}_{\infty}+\|z_{t\alpha}\|_{\infty}\norm{\frac{1}{z_{\alpha}}}_{\infty}\|z_t\|_{\infty}+N\lambda_{max} \|z_t\|_{\infty} d_I(t)^{-2}+N\lambda_{max} |\dot{z}_j(t)| d_I(t)^{-2}.
\end{align*}
By a priori assumption (\ref{apriori}), for smooth free interface, we have 
\begin{equation}
\inf_{\alpha\in \mathbb{R}}|z_{\alpha}(\alpha,t)|\geq \frac{C_1}{2}.
\end{equation}

\noindent Substitute the estimate from the previous lemma for $\dot{z}_j(t)$, use Sobolev embedding $\|f\|_{L^{\infty}}\leq \|f\|_{H^1}$, we have 
\begin{equation}\label{secondtimederivative}
\begin{split}
|\ddot{z}_j(t)|\leq & \Big\{N\lambda_{max} d_P(t)^{-2}(\|z_t\|_{H^1}+  N\lambda_{max} (d_P(t)^{-1}+d_I(t)^{-1}))\Big\}+\Big\{\|z_{tt}\|_{H^1}+\frac{2}{C_1}\|z_{t}\|_{H^2}\|z_t\|_{H^1}\\
&+N\lambda_{max} \|z_t\|_{H^1} d_I(t)^{-2}+N\lambda_{max} (\|z_t\|_{H^1}+  N\lambda_{max} (d_P(t)^{-1}+d_I(t)^{-1})) d_I(t)^{-2}\Big\}\\
&+\Big\{ ((\|z_t\|_{H^1}+  N\lambda_{max} (d_P(t)^{-1}+d_I(t)^{-1}))\frac{2}{C_1}\|z_t\|_{H^2}+N\lambda_{max}d_I(t)^{-2})\Big\}
\end{split}
\end{equation}
(Here, the first bracket is the estimate for $-\sum_{k\neq j}\frac{\lambda_k i\overline{\dot{z}_j(t)-\dot{z}_k(t)}}{2\pi(\overline{z_j(t)-z_k(t)})^2}$, the second bracket is the estimate for $\bar{F}_z\dot{z}_j(t)$, and the third bracket is the estimate for $\bar{F}_t(z_j(t), t)$).

In abbreviate form, we write the estimate for $\dot{z}_j(t), \ddot{z}_j(t)$ as 
\begin{equation}\label{firstandsecond}
|\dot{z}_j(t)|+|\ddot{z}_j(t)|\leq C(N\lambda_{max}, \|z_t\|_{H^2}, \|z_{tt}\|_{H^1},d_I(t)^{-1}, d_P(t)^{-1}, C_1).
\end{equation}
\end{proof}

\subsection{Quasilinearization.}\label{quasilinearizationsubsection}
Take time derivative on both sides of $\bar{z}_{tt}+ia\bar{z}_{\alpha}=i$, let $u=\bar{z}_t$. We obtain
\begin{equation}\label{time_derivative}
u_{tt}+iau_{\alpha}=-ia_t\bar{z}_{\alpha}=-\frac{\bar{z}_{tt}-i}{|z_{tt}+i|}a_t|z_{\alpha}|:=g.
\end{equation}
To show that $a_t\bar{z}_{\alpha}$ is of lower order, we apply $I-\mathfrak{H}$ on both sides of (\ref{time_derivative}). Then we have 
\begin{equation}\label{quasi}
\begin{split}
&-i(I-\mathfrak{H})a_t \bar{z}_{\alpha}= (I-\mathfrak{H})(u_{tt}+iau_{\alpha})\\
=& [\partial_t^2+ia\partial_{\alpha}, \mathfrak{H}]u+(\partial_t^2+ia\partial_{\alpha})(I-\mathfrak{H})u.
\end{split}
\end{equation}
By lemma \ref{basic1},
\begin{equation}
\begin{split}
[\partial_t^2+ia\partial_{\alpha}, \mathfrak{H}]u= 2[z_{tt}, \mathfrak{H}]\frac{\bar{z}_{t\alpha}}{z_{\alpha}}+2[z_t, \mathfrak{H}] \frac{\bar{z}_{tt\alpha}}{z_{\alpha}}-\frac{1}{\pi i}\int \Big(\frac{z_t(\alpha,t)-z_t(\beta,t)}{z(\alpha,t)-z(\beta,t)}\Big)^2 \bar{z}_{t\beta}d\beta
\end{split}
\end{equation}
By lemma \ref{almost}, we have 
\begin{equation}
\begin{split}
(\partial_t^2+ia\partial_{\alpha})(I-\mathfrak{H})u=& -\frac{i}{\pi}\sum_{j=1}^N (\partial_t^2+ia\partial_{\alpha})\frac{\lambda_j}{z(\alpha,t)-z_j(t)}\\
=& \frac{i}{\pi}\sum_{j=1}^N \lambda_j\Big(\frac{2z_{tt}+i-\ddot{z}_j}{(z(\alpha,t)-z_j(t))^2}-2\frac{(z_t-\dot{z}_j(t))^2}{(z(\alpha,t)-z_j(t))^3}\Big)
\end{split}
\end{equation}
By (\ref{firstandsecond}), 
\begin{equation}
|\dot{z}_j(t)|+|\ddot{z}_j(t)|\leq C(N\lambda_{max},  \|z_t\|_{H^2}, \|z_{tt}\|_{H^1}, d_I(t)^{-1}, d_P(t)^{-1}, C_1).
\end{equation}
We rewrite $-i(I-\mathfrak{H})a_t\bar{z}_{\alpha}$ as 
\begin{equation}\label{goodgood}
-i(I-\mathfrak{H})a_t\bar{z}_{\alpha}=g_1+g_2,
\end{equation}
where
\begin{equation}\label{GG1}
g_1:=2[z_{tt}, \mathfrak{H}]\frac{\bar{z}_{t\alpha}}{z_{\alpha}}+2[z_t, \mathfrak{H}] \frac{\bar{z}_{tt\alpha}}{z_{\alpha}}-\frac{1}{\pi i}\int \Big(\frac{z_t(\alpha,t)-z_t(\beta,t)}{z(\alpha,t)-z(\beta,t)}\Big)^2 \bar{z}_{t\beta}d\beta.
\end{equation}
\begin{equation}\label{GG2}
g_2:=\frac{i}{\pi}\sum_{j=1}^N \lambda_j\Big(\frac{2z_{tt}+i-\ddot{z}_j}{(z(\alpha,t)-z_j(t))^2}-2\frac{(z_t-\dot{z}_j(t))^2}{(z(\alpha,t)-z_j(t))^3}\Big).
\end{equation}
Multiply both sides of (\ref{goodgood}) by $i\frac{z_{\alpha}}{|z_{\alpha}|}$ and take real parts, we obtain,
\begin{equation}
(I+\mathfrak{K}^{\ast})a_t|z_{\alpha}|=Re(\frac{iz_{\alpha}}{|z_{\alpha}|}(g_1+g_2)),
\end{equation}
where 
\begin{equation}
\mathfrak{K}^{\ast} f(\alpha)=p.v. \int Re\Big\{-\frac{1}{\pi i}\frac{z_{\alpha}}{|z_{\alpha}|}\frac{|z_{\beta}(\beta,t)|}{z(\alpha)-z(\beta)}   \Big\}f(\beta)d\beta.
\end{equation}
Both $g_1$ and $g_2$ are lower order terms.

Assuming the a priori assumptions (\ref{apriori}) and (\ref{apriori2}), for $0\leq t\leq T_0$, $(I+\mathfrak{K}^{\ast})$ is invertible on $L^2(\Sigma)$, so we have 
\begin{equation}
a_t|z_{\alpha}|=(I+\mathfrak{K}^{\ast})^{-1}\Big\{Re(\frac{iz_{\alpha}}{|z_{\alpha}|}(g_1+g_2))\Big\}.
\end{equation}
By Lemma \ref{layer}, we have 
\begin{equation}
    \|a_tz_{\alpha}\|_{H^s}\leq C\norm{\frac{iz_{\alpha}}{|z_{\alpha}|}(g_1+g_2)}_{H^s},
\end{equation}
for $C$  depends on $C_1, C_2$, and $\|z_{\alpha}-1\|_{H^{s-1}}$.

So (\ref{time_derivative}) can be written as 
\begin{equation}\label{time22}
u_{tt}+a|z_{\alpha}|\bold{n}\frac{u_{\alpha}}{z_{\alpha}}=g,
 \end{equation}
where
\begin{equation}
\bold{n}=\frac{iz_{\alpha}}{|z_{\alpha}|},\quad \quad a|z_{\alpha}|=|z_{tt}+i|=|\bar{u}_t+i|,
\end{equation}
and
\begin{equation}
g:=\frac{\bar{z}_{tt}-i}{|z_{tt}+i|}(I+\mathfrak{K}^{\ast})^{-1}\Big\{Re(\frac{iz_{\alpha}}{|z_{\alpha}|}(g_1+g_2))\Big\}
\end{equation}
Denote \footnote{This $A$ here is not the same as that in \S 3 and \S 5. }
\begin{equation}
A:=a|z_{\alpha}|,\quad \quad D:=\frac{\partial_{\alpha}}{z_{\alpha}}.
\end{equation}
Let $k\in \mathbb{N}$. Apply $\partial_{\alpha}^k$ on both sides of (\ref{time22}), we have 
\begin{equation}\label{alpha_k}
(\partial_{\alpha}^k u)_{tt}+A\bold{n}\partial_{\alpha}^k 
(\frac{\partial_{\alpha}}{z_{\alpha}}u)=\partial_{\alpha}^k g-[\partial_{\alpha}^k,  A\bold{n}]Du.
\end{equation}
We have 
\begin{equation}
\begin{split}
[\partial_{\alpha}^k, A\bold{n}]Du
=& \sum_{m=1}^k c_{m,k}\partial_{\alpha}^m(A\bold{n}) \partial_{\alpha}^{k-m}Du,
\end{split}
\end{equation}
where 
$$c_{m,k}=\frac{k!}{m!(k-m)!}.$$
So we obtain
\begin{equation}\label{alpha_k1}
\begin{cases}
\partial_t^2\partial_{\alpha}^k u+A\bold{n}\partial_{\alpha}^k \frac{\partial_{\alpha}}{z_{\alpha}}u=g_k,\\
A=a|z_{\alpha}|\\
\bold{n}=\frac{iz_{\alpha}}{|z_{\alpha}|}=\frac{\bar{u}_t+i}{|\bar{u}_t+i|}\\
g_k=\partial_{\alpha}^k g-\sum_{m=1}^k c_{m,k}\partial_{\alpha}^m (A\bold{n}) \partial_{\alpha}^{k-m}Du.
\end{cases}
\end{equation}

\subsection{Energy estimates.}
Decompose $u=f+p$ as in (\ref{decomposition}).
Let $s\in \mathbb{N}$. With quasilinearization (\ref{alpha_k1}), we define the energy $E(t)$ as 
\begin{equation}\label{energy}
E(t):=\sum_{k=0}^s \Big\{\int \frac{|z_{\alpha}|^{-2k+1}}{a|z_{\alpha}|}|\partial_{\alpha}^k u_t|^2 d\alpha + Re\int \bold{n} |z_{\alpha}|D^{k+1}f \overline{D^k f} d\alpha\Big\}
\end{equation}
Note that $f$ is holomorphic in $\Omega(t)$, so is $D^m f$, for any integer $m\geq 0$. 
So we have 
\begin{align*}
Re\int \bold{n}|z_{\alpha}| D^{k+1}f\overline{D^k f}d\alpha=Re \int i\partial_{\alpha} D^k f\overline{D^k f}=\int_{\Omega(t)} |\nabla D^k f|^2 dxdy\geq 0.
\end{align*}
So the energy $E$ is positive. 

We can bound $u_t$ by the energy $E$. Assume the bootstrap assumptions (\ref{apriori}) and (\ref{apriori2}), since $a|z_{\alpha}|=|z_{tt}+i|$, we have 
\begin{equation}
    a|z_{\alpha}|\leq |z_{tt}|+1\leq \|z_{tt}\|_{L^{\infty}}+1\leq \|z_{tt}\|_{H^1}+1\leq 2\|w_0\|_{H^s}+1.
\end{equation}

Without loss of generality, we assume $C_2>1$. By the definition of $E$ and the definition of $\mathcal{T}$, for $t\in [0,T]$, we have 
\begin{align}
    E(t)\geq &\sum_{k=0}^s\int \frac{\inf_{\alpha\in \mathbb{R}}|z_{\alpha}|^{-2k+1}}{\sup_{\alpha\in \mathbb{R}}a|z_{\alpha}|}|\partial_{\alpha}^k u_t(\alpha,t)|^2 d\alpha\\
    \geq & \sum_{k=0}^s \frac{(2C_2)^{-2k+1}}{2\|w_0\|_{H^s}+1}\int |\partial_{\alpha}^k u_t(\alpha,t)|^2 d\alpha\\
    \geq &\frac{(2C_2)^{-2s+1}}{2\|w_0\|_{H^s}+1}\|z_{tt}(\cdot,t)\|_{H^s}^2.
\end{align}
So we have 
\begin{equation}\label{boundaceleration}
    \|z_{tt}(\cdot,t)\|_{H^s}\leq \frac{(2\|w_0\|_{H^s}+1)^{1/2}}{(2C_2)^{-s+1/2}}E(t)^{1/2}.
\end{equation}

\noindent Let 
\begin{equation}
\mathcal{E}(t):=\max_{\tau\in [0,t]}E(\tau).
\end{equation}

Note that 
\begin{equation}
\begin{split}
&\frac{d}{dt}Re\int \bold{n}|z_{\alpha}|D^{k+1}f \overline{D^k f} d\alpha=\frac{d}{dt}Re\int \frac{iz_{\alpha}}{|z_{\alpha}|}|z_{\alpha}|\frac{\partial_{\alpha}}{z_{\alpha}}D^{k}f \overline{D^k f} d\alpha\\
=& Re \frac{d}{dt}\int i \partial_{\alpha} D^k f \overline{D^k f}\\
=& Re \frac{d}{dt}\Big\{\int i \partial_{\alpha} D^k u \overline{D^k u}-\int i \partial_{\alpha} D^k u \overline{D^k p}-\int i \partial_{\alpha} D^k p \overline{D^k u}+\int i \partial_{\alpha} D^k p \overline{D^k p}\Big\}\\
:=& Re\frac{d}{dt} (I_1+I_2+I_3+I_4).
\end{split}
\end{equation}
Note that  
\begin{equation}
p=-\frac{i}{2\pi}\sum_{j=1}^N \frac{\lambda_j}{z(\alpha,t)-z_j(t)}.
\end{equation}
So we have 
\begin{equation}
(p)_t=\frac{i}{2\pi}\sum_{j=1}^N \frac{\lambda_j(z_t-\dot{z}_j(t))}{(z(\alpha,t)-z_j(t))^2}.
\end{equation}
Observe that 
\begin{equation}
D^m\frac{1}{z(\alpha,t)-z_j(t)}=\frac{(-1)^m m!}{ (z(\alpha,t)-z_j(t))^{m+1}},
\end{equation}
\begin{equation}
D^m\frac{1}{(z(\alpha,t)-z_j(t))^2}=\frac{(-1)^{m}(m+1)!}{(z(\alpha,t)-z_j(t))^{m+2}},
\end{equation}

Therefore, for $k\geq 2$, by lemma \ref{integral} and the a priori assumptions (\ref{apriori}) and (\ref{apriori2}), we have 
\begin{equation}\label{part1}
\|D^{k} p\|_{L^2\cap L^{\infty}}+\|\partial_{\alpha}D^k p\|_{L^2\cap L^{\infty}}+
\|\partial_t D^{k} p\|_{L^2\cap L^{\infty}}+\|\partial_t\partial_{\alpha}D^k p\|_{L^2\cap L^{\infty}}\leq \tilde{C},
\end{equation}
for some 
$$\tilde{C}=\tilde{C}(\|z_{\beta}\|_{\infty},\|z_t\|_{L^2}, \|z_{tt}\|_{L^2}, d_I(t)^{-1}, d_P(t)^{-1}, N\lambda_{\max},C_1, C_2).$$
We can take $\tilde{C}$ to be a polynomial with positive coefficients.

We need to estimate $\|D^k u\|_{L^2}$ as well. Note that for $0\leq t\leq T_0\leq 1$,  
\begin{equation}
\begin{split}
\|z_t(t)\|_{H^s}\leq &\|v_0\|_{H^s}+\norm{\int_0^t z_{tt}(\cdot, \tau)d\tau}_{H^s}\leq \|v_0\|_{H^s}+t\sup_{\tau \in [0,t]}\|z_{tt}(t=\tau)\|_{H^s}\\
\leq & \|v_0\|_{H^s}+\frac{(2\|w_0\|_{H^s}+1)^{1/2}}{(2C_2)^{-s+1/2}}\mathcal{E}(t)^{1/2}\\
\leq &C(\|v_0\|_{H^s}, \|w_0\|_{H^s}, C_2, \mathcal{E}(t)).
\end{split}
\end{equation}
Similarly, we have 
\begin{equation}
\|z_{\alpha}(t)-1\|_{H^{s-1}}\leq C( \|\partial_{\alpha}\xi_0\|_{H^{s-1}},\|v_0\|_{H^s}, \|w_0\|_{H^s}, C_2, \mathcal{E}(t)).
\end{equation}
Therefore, under the a priori assumption (\ref{apriori}), using that $D^k=(\frac{\partial_{\alpha}}{z_{\alpha}})^k$, we have 
\begin{equation}
\|D^k u\|_{L^2}\leq C( \|\partial_{\alpha}\xi_0\|_{H^{s-1}},\|v_0\|_{H^s}, \|w_0\|_{H^s}, C_2, \mathcal{E}(t)),
\end{equation}
for some polynomial $C$ with positive coefficients.

From (\ref{part1}), integration by parts if necessary, we see that 
\begin{align*}
Re\frac{d}{dt} (I_2+I_3+I_4)\leq C(\|\partial_{\alpha}\xi_0\|_{H^{s-1}},\|v_0\|_{H^s}, \|w_0\|_{H^s}, \mathcal{E}(t), d_I(t)^{-1}, d_P(t)^{-1}, N\lambda_{max}, C_1, C_2)
\end{align*}
For $I_1$, 
\begin{align*}
&Re \frac{d}{dt}\int i \partial_{\alpha} D^k u \overline{D^k u}
= Re~ i\int \partial_{\alpha}\partial_t D^k u \overline{D^k u}+\partial_{\alpha} D^k u \overline{\partial_t D^k u}\\
=&2Re\int i \partial_{\alpha}D^k u\overline{\partial_t D^k u}.
\end{align*}
We have
\begin{equation}
\partial_{\alpha} D^k u=\frac{1}{z_{\alpha}^{k-1}}\partial_{\alpha}^k D u+\frac{1}{z_{\alpha}^{k-2}}\partial_{\alpha}^{k-1}(\frac{1}{z_{\alpha}})\partial_{\alpha}Du+F_k,
\end{equation}
where $F_k$ consists of lower order terms. We have 
\begin{equation}
\| F_k\|_{H^1}\leq C(\|\partial_{\alpha}\xi_0\|_{H^{s-1}},\|v_0\|_{H^s}, \|w_0\|_{H^s}, \mathcal{E}(t), d_I(t)^{-1}, d_P(t)^{-1}, N\lambda_{max}, C_1, C_2).
\end{equation}
We have also that 
\begin{equation}
\begin{split}
&\norm{\frac{1}{z_{\alpha}^{k-2}}\partial_{\alpha}^{k-1}(\frac{1}{z_{\alpha}})\partial_{\alpha}Du}_{L^2}\\
\leq  &C(\|\partial_{\alpha}\xi_0\|_{H^{s-1}},\|v_0\|_{H^s}, \|w_0\|_{H^s}, \mathcal{E}(t), d_I(t)^{-1}, d_P(t)^{-1}, N\lambda_{max}, C_1, C_2).
\end{split}
\end{equation}
Similarly, we write
\begin{equation}
\partial_t D^ku=\frac{1}{z_{\alpha}^k}\partial_t \partial_{\alpha}^ku+\frac{1}{z_{\alpha}^{k-1}}\partial_{\alpha}^{k-1}\partial_t (\frac{1}{z_{\alpha}})u_{\alpha}+G_k,
\end{equation}
where $G_k$ consists of lower order terms, and
\begin{equation}
\|G_k\|_{H^1}\leq C(\|\partial_{\alpha}\xi_0\|_{H^{s-1}},\|v_0\|_{H^s}, \|w_0\|_{H^s}, \mathcal{E}(t), d_I(t)^{-1}, d_P(t)^{-1}, N\lambda_{max}, C_1, C_2).
\end{equation}
Note that 
\begin{equation}
\partial_{\alpha}^{k-1}\partial_t z_{\alpha}^{-1}=-\frac{\partial_{\alpha}^k z_t}{z_{\alpha}^2}+H_k,
\end{equation}
where $H_k$ consists of lower order terms. So we can obtain
\begin{equation}
\begin{split}
&\norm{ \frac{1}{z_{\alpha}^{k-1}}\partial_{\alpha}^{k-1}\partial_t (\frac{1}{z_{\alpha}})u_{\alpha}}_{L^2}\\
\leq &C(\|\partial_{\alpha}\xi_0\|_{H^{s-1}},\|v_0\|_{H^s}, \|w_0\|_{H^s}, \mathcal{E}(t), d_I(t)^{-1}, d_P(t)^{-1}, N\lambda_{max}, C_1, C_2).
\end{split}
\end{equation}
Therefore, from the above estimates, we have 
\begin{align*}
2Re\int i \partial_{\alpha}D^k u\overline{\partial_t D^k u}=&2Re \int i \frac{1}{z_{\alpha}^{k-1}}\partial_{\alpha}^k Du \frac{1}{\bar{z}_{\alpha}^k}\overline{\partial_t \partial_{\alpha}^k u}+error_k\\
=&2Re\int \frac{iz_{\alpha}}{|z_{\alpha}|}\frac{1}{|z_{\alpha}|^{2k-1}}\partial_{\alpha}^k Du \overline{\partial_t \partial_{\alpha}^k u}+error_k\\
=& 2Re\int \bold{n} \frac{1}{|z_{\alpha}|^{2k-1}}\partial_{\alpha}^k Du \overline{\partial_t \partial_{\alpha}^k u}+error_k,
\end{align*}
where 
\begin{equation}
\begin{split}
error_k=&2Re\int i \partial_{\alpha}D^k u\overline{\partial_t D^k u}-2Re\int \bold{n} \frac{1}{|z_{\alpha}|^{2k-1}}\partial_{\alpha}^k Du \overline{\partial_t \partial_{\alpha}^k u}\\
\leq & C(\|\partial_{\alpha}\xi_0\|_{H^{s-1}},\|v_0\|_{H^s}, \|w_0\|_{H^s}, \mathcal{E}(t), d_I(t)^{-1}, d_P(t)^{-1}, N\lambda_{max}, C_1, C_2).
\end{split}
\end{equation}

Observe that 
\begin{equation}
\begin{split}
\frac{dE}{dt}=& \sum_{k=0}^s  \Big\{\int \Big(\frac{|z_{\alpha}|^{-2k+1}}{a|z_{\alpha}|}\Big)_t|\partial_{\alpha}^k u_t|^2   +2 Re\int  \frac{|z_{\alpha}|^{-2k+1}}{a|z_{\alpha}|}(\partial_{\alpha}^k u_{tt})\overline{\partial_{\alpha}^ku_t} +\frac{d}{dt}Re\int \bold{n} |z_{\alpha}|D^{k+1}f \overline{D^k f} d\alpha\\
=&\sum_{k=0}^s  \Big\{\int \Big(\frac{|z_{\alpha}|^{-2k+1}}{a|z_{\alpha}|}\Big)_t|\partial_{\alpha}^k u_t|^2   +2 Re\int  \frac{|z_{\alpha}|^{-2k+1}}{a|z_{\alpha}|}(\partial_{\alpha}^k u_{tt})\overline{\partial_{\alpha}^ku_t} +2Re\int \bold{n} \frac{1}{|z_{\alpha}|^{2k-1}}\partial_{\alpha}^k Du \overline{\partial_t \partial_{\alpha}^k u}\\
&\quad\quad+error_k  \Big\}\\
=& \sum_{k=0}^s  \Big\{\int \Big(\frac{|z_{\alpha}|^{-2k+1}}{a|z_{\alpha}|}\Big)_t|\partial_{\alpha}^k u_t|^2   +2 Re\int  \frac{|z_{\alpha}|^{-2k+1}}{a|z_{\alpha}|}\Big\{(\partial_{\alpha}^k u_{tt})+a|z_{\alpha}|\bold{n} \partial_{\alpha}^k Du\Big\}\overline{\partial_{\alpha}^ku_t}+error_k  \Big\}\\
=&\sum_{k=0}^s  \Big\{\int \Big(\frac{|z_{\alpha}|^{-2k+1}}{a|z_{\alpha}|}\Big)_t|\partial_{\alpha}^k u_t|^2   +2 Re\int  \frac{|z_{\alpha}|^{-2k+1}}{a|z_{\alpha}|}g_k\overline{\partial_{\alpha}^ku_t}+error_k  \Big\}\\
\leq & \sum_{k=0}^s \norm{ \Big(\frac{|z_{\alpha}|^{-2k+1}}{a|z_{\alpha}|}\Big)_t}_{\infty} \|\partial_{\alpha}u_t\|_{L^2}^2+2\norm{\frac{|z_{\alpha}|^{-2k+1}}{a|z_{\alpha}|}}_{\infty} \|g_k\|_{L^2} \|\partial_{\alpha}^k u_t\|_{L^2}+error_k
\end{split}
\end{equation}
It's easy to obtain the estimate that 
\begin{equation}
\|g_k\|_{L^2}\leq C(\|\partial_{\alpha}\xi_0\|_{H^{s-1}},\|v_0\|_{H^s}, \|w_0\|_{H^s}, \mathcal{E}(t), d_I(t)^{-1}, d_P(t)^{-1}, N\lambda_{max}, C_1, C_2, \alpha_0),
\end{equation}
and
\begin{equation}
\begin{split}
&\norm{ \Big(\frac{|z_{\alpha}|^{-2k}}{a}\Big)_t}_{\infty}+\norm{\frac{|z_{\alpha}|^{-2k}}{a}}_{\infty}\\
\leq  & C(\|\partial_{\alpha}\xi_0\|_{H^{s-1}},\|v_0\|_{H^s}, \|w_0\|_{H^s}, \mathcal{E}(t), d_I(t)^{-1}, d_P(t)^{-1}, N\lambda_{max}, C_1, C_2, \alpha_0)
\end{split}
\end{equation}

Then we obtain
\begin{equation}
\frac{dE}{dt}\leq C(\|\partial_{\alpha}\xi_0\|_{H^{s-1}},\|v_0\|_{H^s},  \|w_0\|_{H^s}, \mathcal{E}(t), d_I(t)^{-1}, d_P(t)^{-1}, N\lambda_{max}, C_1, C_2, \alpha_0)
\end{equation}
For some polynomial $C$ with positive coefficients.  So we obtain
\begin{equation}
\begin{split}
E(t & )\leq  E(0)\\
&+\int_0^t C(\|\partial_{\alpha}\xi_0\|_{H^{s-1}},\|v_0\|_{H^s}, \|w_0\|_{H^s}, \mathcal{E}(\tau), d_I(\tau)^{-1}, d_P(\tau)^{-1}, N\lambda
_{max}, C_1, C_2, \alpha_0) ds.
\end{split}
\end{equation}
Note that $E(0)=\mathcal{E}(0)$. Take $sup_{0\leq \tau\leq t}$,  we obtain
\begin{equation}\label{growE}
\begin{split}
\mathcal{E}(t &)\leq \mathcal{E}(0)\\
&+\int_0^t C(\|\partial_{\alpha}\xi_0\|_{H^{s-1}},\|v_0\|_{H^s}, \|w_0\|_{H^s}, \mathcal{E}(\tau), d_I(\tau)^{-1}, d_P(\tau)^{-1}, N\lambda
_{max}, C_1, C_2, \alpha_0) d\tau.
\end{split}
\end{equation}

\vspace*{2ex}

\noindent \underline{\textbf{Growth of } $d_P(t)^{-1}, d_I(t)^{-1}$.} To obtain a closed energy estimate, we need also to control the growth of $d_P(t)^{-1}$ and $d_I(t)^{-1}$.  Recall that $d_P(t)=\min_{j\neq k}\{|z_j(t)-z_k(t)|\}$, so we have 
\begin{equation}
    d_P(t)^{-1}=\max_{1\leq j\neq k\leq N}\frac{1}{|z_j(t)-z_k(t)|}.
\end{equation}
Note that
\begin{align*}
\Big|\frac{d}{dt}|z_j(t)-z_k(t)|^{-1}\Big|=&\Big|\frac{(z_j(t)-z_k(t))\cdot (\dot{z}_j(t)-\dot{z}_k(t))}{|z_j(t)-z_k(t)|^3}\Big|\leq |\dot{z}_j(t)-\dot{z}_k(t)|d_p(t)^{-2},
\end{align*}
Use (\ref{firstandsecond}), and control $\|z_t\|_{H^2}, \|z_{tt}\|_{H^1}$ by $\mathcal{E}$, we obtain
\begin{equation}\label{growP}
|\frac{d}{dt}d_P(t)^{-1}|\leq C(N\lambda_{max}, d_P(t)^{-1}, \mathcal{E}, d_I(t)^{-1}, C_1, C_2, \alpha_0).
\end{equation}
To estimate $\frac{d}{dt}d_I(t)^{-1}$, we estimate $\frac{d}{dt}|z_j(t)-z(\alpha,t)|^{-1}$. We have
$$\Big|\frac{d}{dt}|z_j(t)-z(\alpha,t)|^{-1}\Big|=\Big|\frac{(z_j(t)-z(\alpha,t))\cdot (\dot{z}_j(t)-z_t)}{|z_j(t)-z(\alpha,t)|^3}\Big|\leq |\dot{z}_j(t)-z_t|d_I(t)^{-2}.$$
Since $|\frac{d}{dt} d_I(t)^{-1}|\leq \max_{1\leq j\leq N}\sup_{\alpha}|\frac{d}{dt}|z_j(t)-z(\alpha,t)|^{-1}|$, use (\ref{firstandsecond}), and control $\|z_t\|_{H^2}, \|z_{tt}\|_{H^1}$ by $\mathcal{E}$, we obtain
\begin{equation}\label{growI}
|\frac{d}{dt}d_I(t)^{-1}|\leq C(N\lambda_{max}, d_P(t)^{-1}, \mathcal{E}, d_I(t)^{-1}, C_1, C_2, \alpha_0).
\end{equation}
Combine (\ref{growE}), (\ref{growP}), (\ref{growI}), we obtain
\begin{equation}\label{alltime}
\begin{split}
&\frac{d}{dt}\Big(d_P(t)^{-1}+d_I(t)^{-1}+\mathcal{E}(t)\Big)
\leq C,
\end{split}
\end{equation}
where
\begin{equation}\label{207}
    C=C(\|\partial_{\alpha}\xi_0\|_{H^{s-1}},\|v_0\|_{H^s},  \|w_0\|_{H^s}, \mathcal{E}(t), d_I(t)^{-1}, d_P(t)^{-1}, N\lambda_{max}, C_1, C_2, \alpha_0)
\end{equation}
is a polynomial  with positive coefficients (the coefficients are absolute constants which do not depend on $N\lambda_
{max}$, $d_P(0)^{-1}, d_I(0)^{-1}, \mathcal{E}, C_1, C_2, \alpha_0$, $\|\partial_{\alpha}\xi_0\|_{H^{s-1}}, \|v_0\|_{H^s}$, $\|w_0\|_{H^s}$).  So we can use bootstrap argument to obtain closed energy estimates.

\begin{lemma}
Assume the assumptions of Theorem \ref{theorem1}. There exists $T_0$ depends on $N\lambda_{max}$, $d_P(0)^{-1}$, $d_I(0)^{-1}$,  $\norm{(\partial_{\alpha}\xi_0,v_0,w_0)}_{H^{s-1}\times H^s\times H^s}$, $\mathcal{E}(0), C_1, C_2, \alpha_0, s$ such that for all $0\leq t\leq T_0$,
\begin{equation}\label{double}
\begin{cases}
\mathcal{E}(t)+d_P(t)^{-1}+d_I(t)^{-1}\leq 2(\mathcal{E}(0)+d_P(0)^{-1}+d_I(0)^{-1}),\\
\|z_{tt}(\cdot,t)\|_{H^s}\leq 2\|w_0\|_{H^s},\\
\frac{C_1}{2}|\alpha-\beta|\leq |z(\alpha,t)-z(\beta,t)|\leq 2C_2|\alpha-\beta|\\
\inf_{\alpha\in \mathbb{R}}a(\alpha,t)|z_{\alpha}|\geq \frac{1}{2}\alpha_0.
\end{cases}
\end{equation}
\end{lemma}
\begin{proof}
Let $T_0$ to be determined. Define
\begin{equation}
   \mathcal{T}:=\{T\in [0, T_0]:\quad   ~~(\ref{apriori})~and ~(\ref{apriori2}) \text{ hold for all } 0\leq t\leq T\}
\end{equation}
By continuity, $\mathcal{T}$ is closed. Moreover, since $0\in \mathcal{T}$, we have  $\mathcal{T}\neq \emptyset$. Let $T\in \mathcal{T}$. Then (\ref{alltime}) holds on $[0,T]$. So we have 
\begin{align}
   d_P(t)^{-1}+d_I(t)^{-1}+\mathcal{E}(t)\leq & d_P(0)^{-1}+d_I(0)^{-1}+\mathcal{E}(0)+\int_0^t Cd\tau,
\end{align}
where $C$ is given by (\ref{207}).
Since $C$ is a polynomial with positive coefficients, by taking $T_0$ sufficiently small, $T_0$ depends only on $N\lambda_{max}$, $d_P(0)^{-1}$, $d_I(0)^{-1}$,  $\norm{(\partial_{\alpha}\xi_0,v_0,w_0)}_{H^{s-1}\times H^s\times H^s}$, $\mathcal{E}(0), C_1, C_2, \alpha_0, s$, we have 
\begin{equation}
    \mathcal{E}(t)+d_P(t)^{-1}+d_I(t)^{-1}\leq \frac{3}{2}(\mathcal{E}(0)+d_P(0)^{-1}+d_I(0)^{-1}).
\end{equation}

\noindent  For $t\in [0,T]$, we have $\mathcal{E}(t)\leq 2\mathcal{E}(0)$, so we have by (\ref{boundaceleration}),
\begin{equation}
    \|z_{tt}(\cdot,t)\|_{H^s}\leq \sqrt{2}\frac{(2\|w_0\|_{H^s}+1)^{1/2}}{(2C_2)^{-s+1/2}}\mathcal{E}(0)^{1/2}:=M_1.
\end{equation}
Use $z_t(\cdot,t)=z_t(\cdot,0)+\int_0^t z_{tt}(\cdot,\tau)d\tau$, we have 
\begin{equation}
    \|z_t(\cdot,t)\|_{H^s}\leq \|z_t(\cdot,0)\|_{H^s}+T_0 M_1:=M_2.
\end{equation}
Use $z_{\alpha}(\cdot,t)-1=z_{\alpha}(\cdot,0)-1+\int_0^t z_{\tau\alpha}(\cdot, \tau) d\tau$, we obtain
\begin{equation}
    \|z_{\alpha}(\cdot,t)-1\|_{H^{s-1}}\leq \|\partial_{\alpha}\xi_0\|_{H^{s-1}}+T_0 M_2:=M_3,
\end{equation}
and
\begin{equation}
    \|z_{\alpha}(\cdot,t)-z_{\alpha}(\cdot,0)\|_{\infty}\leq \int_0^t \|z_{t\alpha}\|_{\infty}d\tau \leq T_0M_2.
\end{equation}
By choosing $T_0$ sufficiently small, we have 
\begin{equation}
    \sup_{t\in [0,T]}\|z_{\alpha}(\cdot,t)\|_{\infty}\leq \frac{3}{2}\|z_{\alpha}(\cdot,0)\|_{\infty}\leq  \frac{3}{2}C_2,
\end{equation}
and
\begin{equation}
    \inf_{t\in [0,T]}\inf_{\alpha\in \mathbb{R}}|z_{\alpha}(\alpha,t)|\geq \frac{2}{3}|z_{\alpha}(\alpha,0)|\geq \frac{2}{3}C_1.
\end{equation}
Therefore, since $|z(\alpha,t)-z(\beta,t)|=|z_{\alpha}(\gamma,t)(\alpha-\beta)|$ for some $\gamma$ between $\alpha$ and $\beta$,  we have 
\begin{equation}\label{chordarcchordarc}
    \frac{2}{3}C_1|\alpha-\beta|\leq |z(\alpha,t)-z(\beta,t)|\leq \frac{3}{2}|\alpha-\beta|,\quad \quad t\in [0,T].
\end{equation}

Multiply both sides of the equation $(\partial_t^2+ia\partial_{\alpha})u=-ia_t\bar{z}_{\alpha}$ by $\bar{u}_t$ and integrate in $\alpha$, then take real parts, we have
\begin{align}\label{accee}
    \frac{1}{2}\frac{d}{dt}\int |u_t|^2 d\alpha=Re\Big\{ -i\int au_{\alpha} \bar{u}_t d\alpha-i\int a_t\bar{z}_{\alpha} \bar{u}_t d\alpha  \Big\}.
\end{align}
For $0\leq t\leq T$, we have 
\begin{align}
    \Big |-i\int au_{\alpha} \bar{u}_t d\alpha-i\int a_t\bar{z}_{\alpha} \bar{u}_t d\alpha\Big|\leq & \|a|z_{\alpha}|\|_{\infty} \norm{\frac{u_{\alpha}}{z_{\alpha}}}_{L^2}\|u_t\|_{L^2}+\|a_tz_{\alpha}\|_{L^2}\|u_t\|_{L^2}\\
    \leq & C(\|w_0\|_{H^s} ,M_1, M_2, M_3, C_1, C_2, \alpha_0).
\end{align}
Similarly, 
\begin{equation}
   \frac{d}{dt} \|u_t\|_{\dot{H}^1}^2\leq C(\|w_0\|_{H^s}, M_1, M_2, M_3, C_1, C_2, \alpha_0).
\end{equation}
So we have for $0\leq t\leq T$,
\begin{equation}
    \|u_t(\cdot,t)\|_{H^1}^2= \|u_t(\cdot,0)\|_{H^1}^2+\int_0^t \frac{d}{d\tau}\|u_t(\cdot,\tau)\|_{H^1}^2 d\tau \leq \|u_t(\cdot,0)\|_{H^1}^2+T_0C.
\end{equation}
By choosing $T_0$ sufficiently small, we have for $0\leq t\leq T$,
\begin{equation}\label{utut}
    \|u_t(\cdot,t)\|_{H^1}^2\leq \frac{3}{2}\|u_t(\cdot,0)\|_{H^1}^2.
\end{equation}

Since $a|z_{\alpha}|=|z_{tt}+i|$, we have 
\begin{equation}\label{badc}
    \frac{d}{dt}a|z_{\alpha}|=\frac{(z_{tt}+i)\cdot z_{ttt}}{|z_{tt}+i|}=\frac{(z_{tt}+i)\cdot (iaz_{t\alpha}+ia_tz_{\alpha})}{|z_{tt}+i|}.
\end{equation}
Using (\ref{badc}), it's easy to obtain
\begin{equation}
    \| a|z_{\alpha}|(\cdot,t)-a|z_{\alpha}|(\cdot,0)\|_{\infty}\leq T_0C.
\end{equation}
By choosing $T_0$ sufficiently small, we have for $0\leq t\leq T$,
\begin{equation}
\| a|z_{\alpha}|(\cdot,t)-a|z_{\alpha}|(\cdot,0)\|_{\infty}\leq \frac{1}{3}\alpha_0.
\end{equation}
So we have for 
\begin{equation}\label{alpha0alpha0}
    \inf_{t\in [0,T]}\inf_{\alpha\in \mathbb{R}}a|z_{\alpha}|(\alpha,t)\geq \frac{2}{3}\alpha_0.
\end{equation}
Combining (\ref{chordarcchordarc}), (\ref{utut}), and (\ref{alpha0alpha0}), together with continuity of these quantities, there must exist $\delta>0$ such that for $0\leq t< T+\delta$,
\begin{equation}
    \begin{split}
        \frac{1}{2}C_1|\alpha-\beta|\leq |z(\alpha,t)-z(\beta,t)|\leq 2|\alpha-\beta|,\\
        \|u_t(\cdot,t)\|_{H^1}^2\leq 2\|u_t(\cdot,0)\|_{H^1}^2,\\
        \inf_{\alpha\in \mathbb{R}}a|z_{\alpha}|(\alpha,t)\geq \frac{1}{2}\alpha_0.
    \end{split}
\end{equation}
So $[0,T+\delta)\subset \mathcal{T})$ and therefore $\mathcal{T}=[0,T_0]$, provided that $T_0$ is sufficiently small, and $T_0$ depends only on $N\lambda_{max}$, $d_P(0)^{-1}$, $d_I(0)^{-1}$,  $\norm{(\partial_{\alpha}\xi,v_0, w_0)}_{H^{s-1}\times H^s\times H^s}$, $\mathcal{E}(0), C_1, C_2, \alpha_0, s$.


\end{proof}

\subsection{Proof of Theorem \ref{theorem1}.} Uniqueness is obtained by a similar argument as the energy estimate above. For local existence, one can use iteration method. We refer the readers to S. Wu's works \cite{Wu1997}\cite{Wu1999} for details of this iteration scheme. Moreover, if we let $T_0^{\ast}$ be the maximal lifespan, then either $T_0^{\ast}=\infty$, or $T_0^{\ast}<\infty$, but
\begin{equation}
\begin{split}
\lim_{T\rightarrow T_0^*-}\| (z_t, z_{tt})\|_{C([0,T];H^s\times H^s)}+\sup_{t\rightarrow T_0^*}(d_I(t)^{-1}+d_P(t)^{-1})=\infty.
\end{split}
\end{equation}
or 
\begin{equation}
   \lim_{t\rightarrow T_0^{\ast}-} \inf_{\alpha\in \mathbb{R}} a(\alpha,t)|z_{\alpha}(\alpha,t)|\leq 0,
\end{equation}
or 
\begin{equation}
    \sup_{\substack{\alpha\neq \beta\\ 0\leq t<T_0^*}}\Big |\frac{z(\alpha,t)-z(\beta,t)}{\alpha-\beta}\Big|+\sup_{\substack{\alpha\neq \beta\\ 0\leq t<T_0^*}}\Big |\frac{\alpha-\beta}{z(\alpha,t)-z(\beta,t)}\Big|=\infty.
\end{equation}
\vspace*{2ex}


\section{Long time behavior for small data}\label{sectionlongtime}
In this section we prove Theorem \ref{longtime}. 
\subsection{Derivation of the cubic structure.}
As was explained in the introduction, the main difficulty of studying long time behavior of the system (\ref{vortex_boundary}) is to find a cubic structure for this system. In \cite{Wu2009}, S. Wu uses $\theta:=(I-\mathfrak{H})(z-\bar{z})$ and shows that $(\partial_t^2-ia\partial_{\alpha})\theta$ is cubic for the irrotational case. We use the same $\theta$ here. 
Using lemma \ref{basic1},
\begin{align*}
&(\partial_t^2-ia\partial_{\alpha})\theta= (I-\mathfrak{H})(\partial_t^2-ia\partial_{\alpha})(z-\bar{z})-[\partial_t^2-ia\partial_{\alpha}, \mathfrak{H}](z-\bar{z})\\
=&-2(I-\mathfrak{H})\partial_t \bar{z}_t-2[z_t,\mathfrak{H}]\frac{\partial_{\alpha}(z_t-\bar{z}_t)}{z_{\alpha}}+\frac{1}{\pi i}\int \Big(\frac{z_t(\alpha,t)-z_t(\beta,t)}{z(\alpha,t)-z(\beta,t)}\Big)^2(z-\bar{z})_{\beta}d\beta\\
=&-2\partial_t(I-\mathfrak{H})\bar{z}_t-2[z_t,\mathfrak{H}]\frac{\partial_{\alpha}z_t}{z_{\alpha}}+\frac{1}{\pi i}\int \Big(\frac{z_t(\alpha,t)-z_t(\beta,t)}{z(\alpha,t)-z(\beta,t)}\Big)^2(z-\bar{z})_{\beta}d\beta.
\end{align*}
Decompose $\bar{z}_t=f+p$ as before, with $p=-\sum_{j=1}^2\frac{\lambda_j i}{2\pi} \frac{1}{z(\alpha,t)-z_j(t)}$. Since $(I-\mathfrak{H})p=2p$, we have 
\begin{align*}
-2\partial_t(I-\mathfrak{H})\bar{z}_t=-2\partial_t(I-\mathfrak{H})p=& -4p_t,
\end{align*}
and
\begin{align*}
-2[z_t,\mathfrak{H}]\frac{\partial_{\alpha}z_t}{z_{\alpha}}=& -2[\bar{f}, \mathfrak{H}]\frac{\partial_{\alpha}\bar{f}}{z_{\alpha}}-2[\bar{p}, \mathfrak{H}]\frac{\partial_{\alpha}\bar{f}}{z_{\alpha}}-2[\bar{f},\mathfrak{H}]\frac{\partial_{\alpha}\bar{p}}{z_{\alpha}}-2[\bar{p},\mathfrak{H}]\frac{\partial_{\alpha}\bar{p}}{z_{\alpha}}
\end{align*}
Since $f$ is holomorphic, we have $[f,\mathfrak{H}]\frac{f_{\alpha}}{z_{\alpha}}=0$, and hence $[\bar{f},\bar{\mathfrak{H}}]\frac{\bar{f}_{\alpha}}{\bar{z}_{\alpha}}=0$, so
\begin{equation}
-2[\bar{f}, \mathfrak{H}]\frac{\partial_{\alpha}\bar{f}}{z_{\alpha}}=-2[\bar{f}, \mathfrak{H}\frac{1}{z_{\alpha}}+\bar{\mathfrak{H}}\frac{1}{\bar{z}_{\alpha}}]\bar{f}_{\alpha},
\end{equation}
which is cubic. So we obtain
\begin{equation}
\begin{split}
(\partial_t^2-ia\partial_{\alpha})\theta
=&-2[\bar{f}, \mathfrak{H}\frac{1}{z_{\alpha}}+\bar{\mathfrak{H}}\frac{1}{\bar{z}_{\alpha}}]\bar{f}_{\alpha}+\frac{1}{\pi i}\int \Big(\frac{z_t(\alpha,t)-z_t(\beta,t)}{z(\alpha,t)-z(\beta,t)}\Big)^2(z-\bar{z})_{\beta}d\beta\\
&-2[\bar{p}, \mathfrak{H}]\frac{\partial_{\alpha}\bar{f}}{z_{\alpha}}-2[\bar{f},\mathfrak{H}]\frac{\partial_{\alpha}\bar{p}}{z_{\alpha}}-2[\bar{p},\mathfrak{H}]\frac{\partial_{\alpha}\bar{p}}{z_{\alpha}}-4p_t.
\end{split}
\end{equation}
Denote 
\begin{equation}
g_c:=-2[\bar{f}, \mathfrak{H}\frac{1}{z_{\alpha}}+\bar{\mathfrak{H}}\frac{1}{\bar{z}_{\alpha}}]\bar{f}_{\alpha}+\frac{1}{\pi i}\int \Big(\frac{z_t(\alpha,t)-z_t(\beta,t)}{z(\alpha,t)-z(\beta,t)}\Big)^2(z-\bar{z})_{\beta}d\beta.
\end{equation}
\begin{equation}
g_d:=-2[\bar{p}, \mathfrak{H}]\frac{\partial_{\alpha}\bar{f}}{z_{\alpha}}-2[\bar{f},\mathfrak{H}]\frac{\partial_{\alpha}\bar{p}}{z_{\alpha}}-2[\bar{p},\mathfrak{H}]\frac{\partial_{\alpha}\bar{p}}{z_{\alpha}}-4p_t.
\end{equation}

\vspace*{2ex}

To control $z_{tt}$, we consider the quantity 
$$\sigma:=(I-\mathfrak{H})\partial_t \theta=(I-\mathfrak{H})\partial_t (I-\mathfrak{H})(z-\bar{z}).$$
We have 
\begin{equation}
\begin{split}
(\partial_t^2-ia\partial_{\alpha})\partial_t(I-\mathfrak{H})(z-\bar{z})=&\partial_t (\partial_t^2-ia\partial_{\alpha})(I-\mathfrak{H})(z-\bar{z})+ia_t((I-\mathfrak{H})(z-\bar{z}))_{\alpha}\\
=&\partial_t g+ia_t((I-\mathfrak{H})(z-\bar{z}))_{\alpha}.
\end{split}
\end{equation}
Here,  $g=g_c+g_d$. Use lemma \ref{basic1} , 
\begin{equation}
\begin{split}
&(\partial_t^2-ia\partial_{\alpha})\sigma
=(I-\mathfrak{H})(\partial_t^2-ia\partial_{\alpha})\partial_t (I-\mathfrak{H})(z-\bar{z})-[\partial_t^2-ia\partial_{\alpha},\mathfrak{H}]\partial_t (I-\mathfrak{H})(z-\bar{z})\\
=&(I-\mathfrak{H})(\partial_t g+ia_t((I-\mathfrak{H})(z-\bar{z}))_{\alpha})-2[z_t,\mathfrak{H}]\frac{\partial_{\alpha}\partial_t^2(I-\mathfrak{H})(z-\bar{z})}{z_{\alpha}}\\
&+\frac{1}{\pi i}\int \Big(\frac{z_t(\alpha,t)-z_t(\beta,t)}{z(\alpha,t)-z(\beta,t)}\Big)^2 ((I-\mathfrak{H})(z-\bar{z}))_{t\beta}d\beta\\
:=&\tilde{g}_1+\tilde{g}_2+\tilde{g}_3.
\end{split}
\end{equation}
\begin{remark}
We have 
\begin{equation}
\begin{split}
    \tilde{g}_1=&(I-\mathfrak{H})\partial_t g_c+(I-\mathfrak{H})\partial_t g_d+(I-\mathfrak{H})ia_t(I-\mathfrak{H})(z-\bar{z})_{\alpha}\\
    :=&\tilde{g}_{11}+\tilde{g}_{12}+\tilde{g}_{13}.
    \end{split}
\end{equation}
Note that $\tilde{g}_{11}$ and $\tilde{g}_3$ are obvious cubic or enjoy nice time decay. As one can see later, $\tilde{g}_2$ is cubic as well. Since $a_t\bar{z}_{\alpha}$ consists of quadratic nonlinearities and terms with sufficiently fast time decay, as long as the point vortices move away from the interface at a speed which has a positive lower bound, so $(\partial_t^2-ia\partial_{\alpha})(I-\mathfrak{H})\partial_t(I-\mathfrak{H})(z-\bar{z})$ consists of cubic or higher order nonlinearities, or nonlinearities with rapid time decay, as long as the point vortices move away from the interface rapidly.
\end{remark}


\subsection{Change of coordinates.} Note that $(a-1)\theta_{\alpha}$ involves quadratic nonlinearities, which does not directly lead to cubic lifespan. To resolve the problem, we use the diffeomorphism $\kappa:\mathbb{R}\rightarrow \mathbb{R}$ such that $\bar{\zeta}-\alpha$ is holomorphic, where $\bar{\zeta}=z\circ\kappa^{-1}$. This $\kappa$ was used in \cite{Wu2009}\cite{Totz2012} for the irrotational case. 
Here we need to derive the formulae for $b$ and $A$ for the case with point vortices. Let $\Psi$ be the holomorphic function on $\Omega(t)$ such that 
$$\bar{\zeta}-\alpha=\Psi\circ \zeta.$$

We denote 
$$D_t\zeta=z_t\circ\kappa^{-1},\quad \quad A:=(a\kappa_{\alpha})\circ \kappa^{-1},\quad \quad b=\kappa_t\circ\kappa^{-1}.$$
Then 
\begin{equation}\label{recover}
\kappa_t=b\circ \kappa.
\end{equation}
Suppose we know $b$, then we can recover $\kappa$ by solving the ODE (\ref{recover}).

Recall that in (\ref{decomposition}), we decompose $\bar{z}_t$ as $\bar{z}_t=f+p$. We denote 
\begin{equation}
    \mathfrak{F}=f\circ\kappa^{-1},\quad \quad q=p\circ \kappa^{-1}.
\end{equation}
Since $f$ is the boundary value of the holomorphic function $F$ on $\Omega(t)$, we have 
\begin{equation}
    \mathfrak{F}(\alpha,t)=F(\zeta(\alpha,t),t),
\end{equation}

In new variables, the water wave system (\ref{vortex_boundary}) can be written as

\begin{equation}\label{vortex_boundary_new}
\begin{cases}
D_t^2\zeta-iA\zeta_{\alpha}=-i\\
\frac{d}{dt}z_j(t)=(v-\frac{\lambda_j i}{2\pi(\overline{z-z_j})})\Big |_{z=z_j}\\
(I-\mathcal{H})(D_t\bar{\zeta}+\sum_{j=1}^N \frac{\lambda_j i}{2\pi(\zeta(\alpha,t)-z_j(t))})=0\\
(I-\mathcal{H})(\bar{\zeta}-\alpha)=0.
\end{cases}
\end{equation}
Here, $\mathcal{H}$ is the Hilbert transform associates with $\zeta$, i.e.
\begin{equation}
\mathcal{H}f(\alpha):=\frac{1}{\pi i}p.v.\int_{-\infty}^{\infty}\frac{\zeta_{\beta}(\beta,t)}{\zeta(\alpha,t)-\zeta(\beta,t)}f(\beta)d\beta.
\end{equation}
To show that (\ref{vortex_boundary_new}) is a closed system, we need to derive formula for $b$ and $A$ in terms of the new variable. Once we have shown that this is a closed system, and prove wellposedness for this system, then in turn, this justifies the existence of such change of variable $\kappa^{-1}$.  Moreover,  if we let $\epsilon_0$ be sufficiently small, then in new variables, we have at $t=0$, 
\begin{equation}\label{initial_new}
\norm{|D|^{1/2}(\zeta(\alpha,0)-\alpha)}_{H^s}+\norm{\mathfrak{F}(\cdot, 0)}_{H^{s+1/2}}+\norm{D_t\mathfrak{F}(\cdot, 0)}_{H^{s}}\leq \frac{3}{2}\epsilon.
\end{equation}


\subsubsection{Formula for the quantities $b$ and $D_tb$}
Note that 

\begin{equation}\label{b11}
D_t\bar{\zeta}=F\circ \zeta-\frac{i}{2\pi}\sum_{j=1}^2 \frac{\lambda_j}{\zeta(\alpha,t)-z_j(t)},\quad \lambda_1=-\lambda_2=\lambda,
\end{equation}
Also, $D_t\bar{\zeta}$ can be written as 
\begin{equation}\label{b22}
\begin{split}
D_t\bar{\zeta}=&D_t(\bar{\zeta}-\alpha)+b=D_t\zeta \Psi_{\zeta}\circ \zeta+\Psi_t\circ \zeta+b.
\end{split}
\end{equation}
By (\ref{b11}) and (\ref{b22}), we have 
\begin{equation}
    F\circ \zeta-\frac{i}{2\pi}\sum_{j=1}^2 \frac{\lambda_j}{\zeta(\alpha,t)-z_j(t)}=D_t\zeta \Psi_{\zeta}\circ \zeta+\Psi_t\circ \zeta+b.
\end{equation}
Apply $I-\mathcal{H}$ on both sides of the above equation, use the fact that 
$$(I-\mathcal{H})\Psi_t\circ \zeta=0,\quad \quad (I-\mathcal{H})F\circ \zeta=0,\quad \quad \Psi_{\zeta}\circ \zeta=\frac{\bar{\zeta}_{\alpha}-1}{\zeta_{\alpha}},$$
we obtain
\begin{equation}\label{forb}
\begin{split}
(I-\mathcal{H})b=&-(I-\mathcal{H})D_t\zeta \frac{\bar{\zeta}_{\alpha}-1}{\zeta_{\alpha}}-(I-\mathcal{H})\frac{i}{2\pi} \sum_{j=1}^2 \frac{\lambda_j}{\zeta(\alpha,t)-z_j(t)}\\
=&-[D_t\zeta, \mathcal{H}]\frac{\bar{\zeta}_{\alpha}-1}{\zeta_{\alpha}}-2\frac{i}{2\pi} \sum_{j=1}^2 \frac{\lambda_j}{\zeta(\alpha,t)-z_j(t)}\\
=& -[D_t\zeta, \mathcal{H}]\frac{\bar{\zeta}_{\alpha}-1}{\zeta_{\alpha}}-\frac{i}{\pi} \sum_{j=1}^2 \frac{\lambda_j}{\zeta(\alpha,t)-z_j(t)},
\end{split}
\end{equation}
where we've used the fact that $\frac{1}{\zeta(\alpha,t)-z_j(t)}$ is boundary value of a holomorphic function in $\Omega(t)^c$, so 
\begin{equation}
(I-\mathcal{H})\frac{1}{\zeta(\alpha,t)-z_j(t)}=\frac{2}{\zeta(\alpha,t)-z_j(t)}.
\end{equation}
So $b$ is quadratic plus terms with sufficient rapid time decay, as long as $z_j(t)$ moves away from the interface rapidly.

We need a formula for $D_tb$ as well. Use $(I-\mathcal{H})b=-[D_t\zeta, \mathcal{H}]\frac{\bar{\zeta}_{\alpha}-1}{\zeta_{\alpha}}-\frac{i}{\pi} \sum_{j=1}^2 \frac{\lambda_j}{\zeta(\alpha,t)-z_j(t)}$, change of variables, we get
\begin{equation}
(I-\mathfrak{{H}})b\circ \kappa=-[z_t,\mathfrak{H}]\frac{\bar{z}_{\alpha}-1}{z_{\alpha}}-\frac{i}{\pi}\sum_{j=1}^2\frac{\lambda_j}{z(\alpha,t)-z_j(t)}.
\end{equation}
So we have 
\begin{equation}
\begin{split}
(I-\mathfrak{H})\partial_t b\circ\kappa=&[z_t,\mathfrak{H}]\frac{\partial_{\alpha}b\circ\kappa}{z_{\alpha}}-[z_{tt},\mathfrak{H}]\frac{\bar{z}_{\alpha}-1}{z_{\alpha}}-[z_{t},\mathfrak{H}]\frac{\bar{z}_{t\alpha}}{z_{\alpha}}\\
&+\frac{1}{\pi i}\int \Big(\frac{z_t(\alpha,t)-z_t(\beta,t)}{z(\alpha,t)-z(\beta,t)}\Big)^2(\bar{z}_{\beta}(\beta,t)-1)d\beta+\frac{i}{\pi}\sum_{j=1}^2\frac{\lambda_j(z_t-\dot{z}_j(t))}{(z(\alpha,t)-z_j(t))^2}.
\end{split}
\end{equation}
Changing coordinates by precomposing with $\kappa^{-1}$, we obtain
\begin{equation}
\begin{split}
(I-\mathcal{H})D_t b=&[D_t\zeta,\mathcal{H}]\frac{\partial_{\alpha}b}{\zeta_{\alpha}}-[D_t^2\zeta,\mathcal{H}]\frac{\bar{\zeta}_{\alpha}-1}{\zeta_{\alpha}}-[D_t\zeta,\mathcal{H}]\frac{\partial_{\alpha}D_t\bar{\zeta}}{\zeta_{\alpha}}\\
&+\frac{1}{\pi i}\int \Big(\frac{D_t\zeta(\alpha,t)-D_t\zeta(\beta,t)}{\zeta(\alpha,t)-\zeta(\beta,t)}\Big)^2(\bar{\zeta}_{\beta}(\beta,t)-1)d\beta+\frac{i}{\pi}\sum_{j=1}^2\frac{\lambda_j(D_t\zeta-\dot{z}_j(t))}{(\zeta(\alpha,t)-z_j(t))^2}.
\end{split}
\end{equation}
So $D_tb$ is quadratic plus terms with sufficient rapid time decay, as long as $z_j(t)$ moves away from the interface rapidly.

\subsection{The quantity $A$} Since $\partial_{\alpha}\mathfrak{F}=\partial_{\alpha}F(\zeta(\alpha,t),t)=F_{\zeta}\circ \zeta\zeta_{\alpha}$, we have 
\begin{equation}
    F_{\zeta}\circ \zeta=\frac{\partial_{\alpha}\mathfrak{F}}{\zeta_{\alpha}}.
\end{equation}
Use $D_t^2\bar{\zeta}+iA\bar{\zeta}_{\alpha}=i$. We have 
\begin{equation}
\begin{split}
D_t^2\bar{\zeta}=&D_t(D_t\bar{\zeta})=D_tF\circ \zeta-\frac{i}{2\pi} D_t\sum_{j=1}^2 \frac{\lambda_j}{\zeta(\alpha,t)-z_j(t)}\\
=& D_t\zeta F_{\zeta}\circ \zeta+\frac{i}{2\pi}\sum_{j=1}^2 \frac{\lambda_j(D_t\zeta(\alpha,t)-\dot{z}_j(t))}{(\zeta(\alpha,t)-z_j(t))^2}\\
=& D_t\zeta \frac{\partial_{\alpha} \mathfrak{F}}{\zeta_{\alpha}}+\frac{i}{2\pi}\sum_{j=1}^2 \frac{\lambda_j(D_t\zeta(\alpha,t)-\dot{z}_j(t))}{(\zeta(\alpha,t)-z_j(t))^2}.
\end{split}
\end{equation}
Also,
\begin{equation}
iA\bar{\zeta}_{\alpha}=iA+iA\partial_{\alpha}(\bar{\zeta}-\alpha)=iA+iA\zeta_{\alpha}\Psi_{\zeta}\circ \zeta=iA+(D_t^2\zeta+i)\Psi_{\zeta}\circ \zeta.
\end{equation}
So we have 
\begin{equation}\label{aaaa}
iA=i-D_t\zeta \frac{\partial_{\alpha} \mathfrak{F}}{\zeta_{\alpha}}-\frac{i}{2\pi}\sum_{j=1}^2 \frac{\lambda_j(D_t\zeta(\alpha,t)-\dot{z}_j(t))}{(\zeta(\alpha,t)-z_j(t))^2}-(D_t^2\zeta+i)\Psi_{\zeta}\circ \zeta.
\end{equation}
Apply $I-\mathcal{H}$ on both sides of (\ref{aaaa}), use the fact that $(I-\mathcal{H})\frac{\mathfrak{F}_{\alpha}}{\zeta_{\alpha}}=0$, $(I-\mathcal{H})\Psi_{\zeta}\circ \zeta=0$, we obtain
\begin{equation}
\begin{split}
&i(I-\mathcal{H})A=i-[D_t\zeta,\mathcal{H}]\frac{\partial_{\alpha}\mathfrak{F}}{\zeta_{\alpha}}-[D_t^2\zeta,\mathcal{H}]\frac{\bar{\zeta}_{\alpha}-1}{\zeta_{\alpha}}-(I-\mathcal{H})\frac{i}{2\pi}\sum_{j=1}^2 \frac{\lambda_j(D_t\zeta(\alpha,t)-\dot{z}_j(t))}{(\zeta(\alpha,t)-z_j(t))^2}\\
\end{split}
\end{equation}
So we obtain
\begin{equation}
\begin{split}
(I-\mathcal{H})A=&1+i[D_t\zeta,\mathcal{H}]\frac{\partial_{\alpha}\mathfrak{F}}{\zeta_{\alpha}}+i[D_t^2\zeta,\mathcal{H}]\frac{\bar{\zeta}_{\alpha}-1}{\zeta_{\alpha}}-(I-\mathcal{H})\frac{1}{2\pi}\sum_{j=1}^2 \frac{\lambda_j(D_t\zeta(\alpha,t)-\dot{z}_j(t))}{(\zeta(\alpha,t)-z_j(t))^2}.
\end{split}
\end{equation}
So $A-1$ is quadratic plus terms with rapid time decay, as long as the point vortices move away from the interface with a speed that has a positive lower bound.

\subsection{The quantity $\frac{a_t}{a}\circ\kappa^{-1}$} We need a formula for $\frac{a_t}{a}\circ\kappa^{-1}$ as well.  Use (\ref{goodgood}) and (\ref{GG1}), (\ref{GG2}), then change of variables, we obtain
\begin{equation}\label{atanew}
\begin{split}
&(I-\mathcal{H})\frac{a_t}{a}\circ\kappa^{-1}A\bar{\zeta}_{\alpha}\\=&2i[D_t^2\zeta, \mathcal{H}]\frac{\partial_{\alpha}D_t\bar{\zeta}}{\zeta_{\alpha}}+2i[D_t\zeta, \mathcal{H}] \frac{\partial_{\alpha}D_t^2\bar{\zeta}}{\zeta_{\alpha}}-\frac{1}{\pi }\int \Big(\frac{D_t\zeta(\alpha,t)-D_t\zeta(\beta,t)}{\zeta(\alpha,t)-\zeta(\beta,t)}\Big)^2 (D_t\bar{\zeta})_{\beta}d\beta\\
&-\frac{1}{\pi}\sum_{j=1}^2 \lambda_j\Big(\frac{2D_t^2\zeta+i-\partial_t^2 z_j}{(\zeta(\alpha,t)-z_j(t))^2}-2\frac{(D_t\zeta-\dot{z}_j(t))^2}{(\zeta(\alpha,t)-z_j(t))^3}\Big)
\end{split}
\end{equation}
So $\frac{a_t}{a}\circ\kappa^{-1}$ is quadratic plus terms with rapid time decay, as long as the point vortices move away from the interface with a speed that has a positive lower bound.

\subsection{Cubic structure in new variables} 
Denote 
\begin{equation}
\tilde{\theta}:=(I-\mathcal{H})(\zeta-\bar{\zeta}),\quad \quad \tilde{\sigma}:=(I-\mathcal{H})D_t\tilde{\theta}.
\end{equation}
We sum up the calculations above, which show that $(D_t^2-iA\partial_{\alpha})\tilde{\theta}$ and $(D_t^2-iA\partial_{\alpha})\tilde{\sigma}$ consist of cubic terms and terms with rapid time decay, as long as the point vortices move away from the interface rapidly.
Recall that $q=p\circ\kappa^{-1}$, so $q$ is given by
\begin{equation}
q=-\sum_{j=1}^2\frac{\lambda_j i}{2\pi}\frac{1}{\zeta(\alpha,t)-z_j(t)}.
\end{equation}
We have
\begin{equation}
\begin{cases}
(D_t^2-iA\partial_{\alpha})\tilde{\theta}=G\\
(D_t^2-iA\partial_{\alpha})\tilde{\sigma}=\tilde{G}\\
\end{cases}
\end{equation}
where $G=G_c+G_d$, with
\begin{equation}
G_c:=-2[\bar{\mathfrak{F}}, \mathcal{H}\frac{1}{\zeta_{\alpha}}+\bar{\mathcal{H}}\frac{1}{\bar{\zeta}_{\alpha}}]\bar{\mathfrak{F}}_{\alpha}+\frac{1}{\pi i}\int \Big(\frac{D_t\zeta(\alpha,t)-D_t\zeta(\beta,t)}{\zeta(\alpha,t)-\zeta(\beta,t)}\Big)^2(\zeta-\bar{\zeta})_{\beta}d\beta.
\end{equation}
\begin{equation}
G_d:=-2[\bar{q}, \mathcal{H}]\frac{\partial_{\alpha}\bar{\mathfrak{F}}}{\zeta_{\alpha}}-2[\bar{\mathfrak{F}},\mathcal{H}]\frac{\partial_{\alpha}\bar{q}}{\zeta_{\alpha}}-2[\bar{q},\mathcal{H}]\frac{\partial_{\alpha}\bar{q}}{\zeta_{\alpha}}-4D_t q,
\end{equation}

\noindent and
\begin{equation}\label{systemnew}
\begin{split}
\tilde{G}=& (I-\mathcal{H})(D_t G+i\frac{a_t}{a}\circ\kappa^{-1}A((I-\mathcal{H})(\zeta-\bar{\zeta}))_{\alpha})-2[D_t\zeta,\mathcal{H}]\frac{\partial_{\alpha}D_t^2(I-\mathcal{H})(\zeta-\bar{\zeta})}{\zeta_{\alpha}}\\
&+\frac{1}{\pi i}\int \Big(\frac{D_t\zeta(\alpha,t)-D_t\zeta(\beta,t)}{\zeta(\alpha,t)-\zeta(\beta,t)}\Big)^2 \partial_{\beta} D_t(I-\mathcal{H})(\zeta-\bar{\zeta})d\beta\\
\end{split}
\end{equation}
\subsubsection{Evolution equation for higher order derivatives} Apply $\partial_{\alpha}^k$ on both sides of (\ref{systemnew}), we have
\begin{equation}\label{systemnewk}
\begin{split}
(D_t^2-iA\partial_{\alpha})\theta_k=G_k^{\theta}\\
(D_t^2-iA\partial_{\alpha})\sigma_k=G_k^{\sigma},\\
\end{split}
\end{equation}
where for $0\leq k\leq s$,
\begin{equation}
\theta_k=(I-\mathcal{H})\partial_{\alpha}^k\tilde{\theta},\quad \quad \sigma_k=(I-\mathcal{H})\partial_{\alpha}^k\tilde{\sigma}.
\end{equation}
\begin{equation}
G_k^{\theta}=(I-\mathcal{H})(\partial_{\alpha}^kG+[D_t^2-iA\partial_{\alpha}, \partial_{\alpha}^k]\tilde{\theta})-[D_t^2-iA\partial_{\alpha},\mathcal{H}]\partial_{\alpha}^k\tilde{\theta},
\end{equation}
and
\begin{equation}
G_k^{\sigma}=(I-\mathcal{H})(\partial_{\alpha}^k\tilde{G}+[D_t^2-iA\partial_{\alpha}, \partial_{\alpha}^k]\tilde{\sigma})-[D_t^2-iA\partial_{\alpha},\mathcal{H}]\partial_{\alpha}^k\tilde{\sigma}.
\end{equation}

\subsection{Energy functional} Define
\begin{equation}
E_k^{\theta}:=\int \frac{1}{A}|D_t\theta_k|^2+i\theta_k \overline{\partial_{\alpha}\theta_k}d\alpha.
\end{equation}
By Wu's basic energy lemma (lemma 4.1, \cite{Wu2009}), we have
\begin{equation}\label{energy1}
\frac{d}{dt}E_k^{\theta}=\int\frac{2}{A}Re D_t\theta_k \bar{G}_k-\int \frac{1}{A}\frac{a_t}{a}\circ\kappa^{-1}|D_t\theta_k|^2 
\end{equation}

\noindent Define
\begin{equation}
E_k^{\sigma}:=\int \frac{1}{A}|D_t\sigma_k|^2+i\sigma_k \overline{\partial_{\alpha}\sigma_k}d\alpha.
\end{equation}
Then we have 
\begin{equation}\label{energy2}
\frac{d}{dt}E_k^{\sigma}=\int\frac{2}{A}Re D_t\sigma_k \overline{\tilde{G}}_k-\int \frac{1}{A}\frac{a_t}{a}\circ\kappa^{-1}|D_t\sigma_k|^2 
\end{equation}
Define
\begin{equation}
\mathcal{E}_s:=\sum_{k=0}^s(E_k^{\theta}+E_k^{\sigma}).
\end{equation}

\subsection{The bootstrap assumption and some preliminary estimates}
 To obtain a priori energy estimates, we make the following bootstrap assumption:  Let $T_0\geq 0$, we assume 
\begin{equation}\label{longtimeapriori}
\|\zeta_{\alpha}-1\|_{H^s}\leq 5\epsilon,\quad\quad \|\mathfrak{F}\|_{H^{s+1/2}}\leq 5\epsilon,\quad \quad  \|D_t \mathfrak{F}\|_{H^s}\leq 5\epsilon,\quad \quad \forall~~t\in [0,T_0].
\end{equation}

\begin{remark}
The assumptions of Theorem \ref{longtime} imply that the bootstrap assumption holds at $T_0=0$. 
\end{remark}

As a consequence of (\ref{longtimeapriori}), we have 
\begin{lemma}[Chord-arc condition]\label{chordarclongtime}
Assume the assumptions of Theorem \ref{longtime} holds. Assume also the bootstrap assumption (\ref{longtimeapriori}), we have 
\begin{equation}
(1-5\epsilon)|\alpha-\beta|\leq |\zeta(\alpha,t)-\zeta(\beta,t)|\leq (1+5\epsilon)|\alpha-\beta|,\quad \quad \forall~t\in [0,T_0].
\end{equation}
\end{lemma}
\begin{proof}
\begin{equation}
|\zeta(\alpha,t)-\zeta(\beta,t)|=|\alpha-\beta+(\zeta(\alpha,t)-\alpha)-(\zeta(\beta,t)-\beta)|.
\end{equation}
Note that 
\begin{equation}
|\zeta(\alpha,t)-\alpha-(\zeta(\beta,t)-\beta)|\leq \|
\zeta_{\alpha}-1\|_{\infty}|\alpha-\beta|\leq 5\epsilon |\alpha-\beta|.
\end{equation}
So the conclusion follows by Triangle inequality.

\end{proof}

\begin{lemma}\label{maximum_principle}
Assume the assumptions of Theorem \ref{longtime} hold. Assume also the bootstrap assumption (\ref{longtimeapriori}), we have for $\epsilon$ sufficiently small,
\begin{equation}\label{max_F}
    \|F(\cdot,t)\|_{L^{\infty}(\Omega(t))}\leq 5\epsilon,
\end{equation}
\begin{equation}
    \|F_{\zeta}(\cdot,t)\|_{L^{\infty}(\Omega(t))}\leq 6\epsilon,
\end{equation}
\begin{equation}\label{max_real}
    |Re F(x+iy,t)|\leq 6\epsilon |x|,
\end{equation}
\begin{equation}
    \|F_t\|_{L^{\infty}(\Omega(t))}\leq 6\epsilon,
\end{equation}
\begin{equation}
    \|F_{\zeta\zeta}\|_{L^\infty(\Omega(t))}\leq 10\epsilon,
\end{equation}
\begin{equation}
    \|F_{t\zeta}\|_{L^{\infty}(\Omega(t))}\leq 10\epsilon.
\end{equation}
\end{lemma}
\begin{proof}
For (\ref{max_F}), by maximum principle, we have 
\begin{equation}
    \|F\|_{L^{\infty}(\Omega(t))}= \|F\|_{L^{\infty}(\Sigma(t))}=\|\mathfrak{F}(t)\|_{\infty}\leq 5\epsilon.
\end{equation}
By bootstrap assumption (\ref{longtimeapriori}), for $\epsilon$ sufficiently small, we have
\begin{equation}\label{Fzeta}
\|F_{\zeta}(\zeta(\alpha,t),t)\|_{\infty}=\norm{\frac{\partial_{\alpha}
\mathfrak{F}(\alpha,t)}{\zeta_{\alpha}}}_{\infty}\leq \frac{\|\mathfrak{F}\|_{H^s}}{\|\zeta_{\alpha}\|_{\infty}}\leq \frac{5\epsilon}{1-5\epsilon}\leq 6\epsilon.
\end{equation}
By maximum principle, we have 
\begin{equation}
\|F_{\zeta}(\cdot,t)\|_{L^{\infty}(\Omega(t))}\leq \|F_{\zeta}(\zeta(\alpha,t),t)\|_{\infty}\leq 6\epsilon.
\end{equation}
Note that 
\begin{equation}\label{FtHs}
    D_t\mathfrak{F}(\alpha,t)=D_t F(\zeta(\alpha,t),t)=F_t\circ\zeta+D_t\zeta F_{\zeta}\circ \zeta.
\end{equation}
So we have for $\epsilon$ sufficiently small (say, $\epsilon<1/36$),
\begin{equation}
    \|F_t(\zeta(\cdot,t),t)\|_{H^s}\leq \|D_t\mathfrak{F}\|_{H^s}+\|D_t\zeta F_{\zeta}\circ\zeta\|_{H^s}\leq 5\epsilon+6\epsilon (6\epsilon)\leq 6\epsilon.
\end{equation}
By Sobolev embedding and maximum principle, we have 
\begin{equation}
\|F_{t}(\cdot,t)\|_{L^{\infty}(\Omega(t))}\leq \|F_{t}(\zeta(\alpha,t),t)\|_{L^{\infty}}\leq \|F_t\circ\zeta\|_{H^1}\leq 6\epsilon.
\end{equation}

\noindent Use the fact that $Re F$ is odd, and the estimate $\|F_{\zeta}\|_{H^1}\leq 6\epsilon$, we have 
\begin{align*}
|Re F(x+iy,t)|=|Re F(x+iy,t)-Re F(0+iy,t)|\leq \|F_{\zeta}\|_{\infty}|x|\leq 6\epsilon |x|.
\end{align*}

Note that 
\begin{equation}
F_{\zeta\zeta}(\zeta(\alpha,t),t)=(\frac{\partial_{\alpha}}{\zeta_{\alpha}})^2F(\zeta(\alpha,t),t)=\frac{1}{\zeta_{\alpha}^2}\mathfrak{F}_{\alpha\alpha}-\frac{\zeta_{\alpha\alpha}}{\zeta_{\alpha}^3}\mathfrak{F}_{\alpha}.
\end{equation}
For $\epsilon$ sufficiently small, by maximum principle, we have
\begin{equation}
\|F_{\zeta\zeta}(\cdot,t)\|_{L^\infty(\Omega(t))}\leq \frac{1}{\inf_{\alpha\in \mathbb{R}}\zeta_{\alpha}^2}\|\mathfrak{F}\|_{H^3}+\frac{\|\zeta_{\alpha\alpha}\|_{\infty}}{\inf_{\alpha\in \mathbb{R}}\zeta_{\alpha}^3}\|\mathfrak{F}\|_{H^1} \leq \frac{1}{(1-5\epsilon)^2}5\epsilon+\frac{5\epsilon}{(1-5\epsilon)^3}5\epsilon\leq 10\epsilon.
\end{equation}

Maximum principle implies $\|F_{t\zeta}\|_{L^{\infty}(\Omega(t))}\leq \|F_{t\zeta}\|_{L^{\infty}(\partial \Omega(t))}$. Since $F_{t\zeta}(\zeta(\alpha,t),t)=\frac{\partial_{\alpha}F_t(\zeta(\alpha,t),t)}{\zeta_{\alpha}}$, by (\ref{FtHs}), we have 
\begin{align*}
\norm{F_{t\zeta}\circ\zeta}_{L^{\infty}}=&\norm{\frac{\partial_{\alpha}F_t(\zeta(\alpha,t),t)}{\zeta_{\alpha}}}_{L^{\infty}}\leq \norm{\partial_{\alpha}F_t(\zeta(\cdot,t),t)}_{\infty}\norm{\frac{1}{\zeta_{\alpha}}}_{\infty}\\
= & \norm{D_t\mathfrak{F}(\alpha,t)-D_t\zeta F_{\zeta}\circ \zeta}_{H^s}\norm{\frac{1}{\zeta_{\alpha}}}_{\infty}\leq 10\epsilon,\\
\end{align*}
\end{proof}

Assume the bootstrap assumption (\ref{longtimeapriori}), we can obtain control of various characteristics of the point vortices. 

\vspace*{2ex}

\noindent \textbf{\underline{Convention.}} We use $K_s$ to denote a constant that depends on $s$. We'll use $K_s\sim \frac{((s+12)!)^2}{((s+7)!)^2}$. $K_s$ can be different at different places, up to an absolute multiplicity constant. We also use $C$ to represent an absolute constant.


\vspace*{2ex}

We'll need the following lemma. Similar versions of this lemma have been appeared in \cite{Wu2009}.
\begin{lemma}\label{realinverse}
Assume the bootstrap assumption (\ref{longtimeapriori}), let $f, h$ be real functions. Assume 
$$(I-\mathcal{H})h\bar{\zeta}_{\alpha}=g\quad \quad or\quad \quad (I-\mathcal{H})h=g.$$
Then we have for any $t\in [0,T_0]$,
\begin{equation}
\|h\|_{H^s}\leq 2\|g\|_{H^s}.
\end{equation}
\end{lemma}

We'll use the following  estimate a lot.
\begin{lemma}\label{smallestimate}
Assume the assumptions of Theorem \ref{longtime} hold. Assume also the bootstrap assumption (\ref{longtimeapriori}), and assume a priori that $d_I(t)\geq 1$, $\frac{1}{2}\leq \frac{x(t)}{x(0)}\leq 2$, $\forall~t\in [0,T_0]$. Then we have $\forall~t\in [0,T_0]$,
\begin{equation}\label{qone}
\|q\|_{H^s}\leq K_s^{-1}\epsilon d_I(t)^{-3/2}.
\end{equation}
\begin{equation}\label{qtwo}
\norm{\sum_{j=1}^2\frac{\lambda_j i}{2\pi}\frac{1}{(\zeta(\alpha,t)-z_j(t))^2}}_{H^s}\leq K_s^{-1}\epsilon d_I(t)^{-5/2}.
\end{equation}
\end{lemma}
\begin{proof}
We prove (\ref{qone}). The proof of (\ref{qtwo}) is similar . Let $s$  be a positive integer, we have 
\begin{align*}
\|q\|_{H^s}^2\leq \sum_{n=0}^{s}\int_{-\infty}^{\infty}\Big| \partial_{\alpha}^n \sum_{j=1}^2\frac{\lambda_j i}{2\pi}\frac{1}{\zeta(\alpha,t)-z_j(t)}\Big|^2 d\alpha
\end{align*}
Denote $f_j(\alpha,t):=\frac{\lambda_j i}{2\pi}\frac{1}{\alpha-z_j(t)}$, $g:=\zeta(\alpha,t)$. Then 
$\frac{\lambda_j i}{2\pi}\frac{1}{\zeta(\alpha,t)-z_j(t)}=f_j(g(\alpha,t),t)$. By chain rule for composite functions, we have 
\begin{align}
    \partial_{\alpha}^n f_j(g)=\sum_{k=1}^n\sum \frac{n!}{(k_1)!...(k_n)!}\partial_{\alpha}^k f_j(\cdot,t)\circ g\prod_{l=1}^n \Big(\frac{\partial_{\alpha}^l g}{l!}\Big)^{k_l},
\end{align}
where the second summation is over all non-negative integers $(k_1,...,k_n)$ such that 
\begin{equation}
\begin{cases}
    \sum_{l=1}^n k_l=k\\
    \sum_{l=1}^n lk_l=n.
    \end{cases}
\end{equation}
So we have 
\begin{align}
    \partial_{\alpha}^n q=\sum_{k=1}^n\sum \frac{n!}{(k_1)!...(k_n)!}\Big(\sum_{j=1}^2\partial_{\alpha}^k f_j(\cdot,t)\circ g\Big)\prod_{l=1}^n \Big(\frac{\partial_{\alpha}^l g}{l!}\Big)^{k_l}
\end{align}
Note that 
\begin{align}
    (\sum_{j=1}^2\partial_{\alpha}^k f_j(\cdot,t))\circ g=& \sum_{j=1}^2 \frac{\lambda_j i}{2\pi}\frac{(-1)^k k!}{(\zeta(\alpha,t)-z_j(t))^{k+1}}\\
    =& \frac{\lambda i(-1)^k k!}{2\pi}\sum_{m=0}^k \frac{z_1-z_2}{(\zeta(\alpha,t)-z_1(t))^{k+1-m}(\zeta(\alpha,t)-z_2(t))^{m+1}}
\end{align}

use $z_1-z_2=2x(t)$, 
similar to the proof of lemma \ref{integral}, we have 
\begin{equation}
\norm{\sum_{j=1}^2\partial_{\alpha}^k f_j(\cdot,t)\circ g}_{L^2}\leq 100 (k+1)!|\lambda x(t)|d_I(t)^{-3/2}.
\end{equation}

Therefore, 
\begin{align*}
\|\partial_{\alpha}^n q\|_{L^2}=& \norm{\sum_{k=1}^n\sum \frac{n!}{(k_1)!...(k_n)!}\Big(\sum_{j=1}^2\partial_{\alpha}^k f_j(\cdot,t)\circ g\Big)\prod_{l=1}^n \Big(\frac{\partial_{\alpha}^l g}{l!}\Big)^{k_l}}_{L^2}\\
\leq & \sum_{k=1}^n \sum \frac{n!}{(k_1)!...(k_n)!}\prod_{l=1}^n \frac{\|\partial_{\alpha}^l g\|_{\infty}^{k_l}}{(l!)^{k_l}}\norm{\sum_{j=1}^2\partial_{\alpha}^k f_j(\cdot,t)\circ g}_{L^2}
\end{align*}
For $l=1$, we bound $\partial_{\alpha}^lg$ by $1+5\epsilon$. For $l\geq 2$, we bound $\|\partial_{\alpha}^lg\|_{\infty}$ by $5\epsilon$. We choose $\epsilon$ small so that $(1+5\epsilon)^s\leq 2$. We bound $(k+1)!$ by $(n+1)!$. Use the assumption $x(t)\leq 2x(0)$. Use 
\begin{equation}
    \prod_{l=1}^n\|\partial_{\alpha}^lg\|_{\infty}^{k_l}\leq \prod_{j=1}^n (1+5\epsilon)^{k_l}\leq (1+5\epsilon)^s,
\end{equation}
we obtain
\begin{align}
    \|\partial_{\alpha}^n q\|_{L^2}\leq &\sum_{k=1}^n \sum \frac{n!}{(k_1)!...(k_n)!}\prod_{l=1}^n \frac{(1+5\epsilon)^{k_l}}{(l!)^{k_l}}\times (100 (k+1)!|\lambda x(t)|d_I(t)^{-3/2})\\
    \leq & 400S(n)|\lambda| x(0) (n+1)! d_I(t)^{-3/2},
\end{align}
where 
\begin{equation}
    S(n)=\sum_{k=1}^n \sum \frac{n!}{(k_1)!...(k_n)!}\prod_{l=1}^n \frac{1}{(l!)^{k_l}}
\end{equation}
is called the bell number. We can bound $S(n)$ by 
\begin{equation}
    S(n)\leq n!.
\end{equation}
So we have 
\begin{equation}
    \|\partial_{\alpha}^nq\|_{L^2}\leq 400|\lambda x(0)| n! (n+1)! d_I(t)^{-3/2}.
\end{equation}
Therefore, 
\begin{align}
    \|q\|_{H^s}\leq &\Big(\sum_{n=0}^s \|\partial_{\alpha}^n q\|_{L^2}^2\Big)^{1/2}\\
    \leq & \Big(\sum_{n=0}^s (400|\lambda x(0)| n! (n+1)! d_I(t)^{-3/2})^2\Big)
    ^{1/2}\\
    \leq & 400((s+2)!)^2 |\lambda x(0)| d_I(t)^{-3/2}\\
    \leq & K_s^{-1}\epsilon d_I(t)^{-3/2}.
\end{align}
\end{proof}
\begin{corollary}
Assume the assumptions of Theorem \ref{longtime} hold and assume the bootstrap assumption (\ref{longtimeapriori}), and assume a priori that $d_I(t)\geq 1$, $\frac{1}{2}\leq \frac{x(t)}{x(0)}\leq 2$, $\forall~t\in [0,T_0]$. Then we have 
\begin{equation}
\sup_{t\in [0,T]}\|D_t\zeta\|_{H^s}\leq 6\epsilon,\quad \quad \forall~t\in [0,T_0].
\end{equation}
\end{corollary}

\begin{corollary}\label{estimateforbb}
Assume the assumptions of Theorem \ref{longtime} hold and assume the bootstrap assumption (\ref{longtimeapriori}), and assume a priori that $d_I(t)\geq 1$, $\frac{1}{2}\leq \frac{x(t)}{x(0)}\leq 2$, $\forall~t\in [0,T_0]$. Then we have 
\begin{equation}
    \|b\|_{H^s}\leq C\epsilon^2+K_s^{-1}\epsilon d_I(t)^{-3/2},\quad \quad \forall~t\in [0,T_0].
\end{equation}
for some absolute constant $C>0$.
\end{corollary}
\begin{proof}
$$(I-\mathcal{H})b=-[D_t\zeta, \mathcal{H}]\frac{\bar{\zeta}_{\alpha}-1}{\zeta_{\alpha}}-\frac{i}{\pi} \sum_{j=1}^2 \frac{\lambda_j}{\zeta(\alpha,t)-z_j(t)}.$$
By lemma \ref{lemmacommutator1} and lemma \ref{smallestimate}, we have 
\begin{align*}
\norm{(I-\mathcal{H})b}_{H^s}\leq &\norm{[D_t\zeta, \mathcal{H}]\frac{\bar{\zeta}_{\alpha}-1}{\zeta_{\alpha}}}_{H^s}+\norm{\frac{i}{\pi} \sum_{j=1}^2 \frac{\lambda_j}{\zeta(\alpha,t)-z_j(t)}}_{H^s}\\
\leq & C\epsilon^2+K_s^{-1}\epsilon d_I(t)^{-3/2}.
\end{align*}
So we have 
\begin{equation}\label{estimateforb}
\|b\|_{H^s}\leq C\epsilon^2+K_s^{-1}\epsilon d_I(t)^{-3/2}.
\end{equation}
\end{proof}

\begin{remark}
Again, the a priori assumption $d_I(t)\geq 1, \frac{1}{2}\leq \frac{x(t)}{x(0)}\leq 2$,  $\forall~t\in [0,T_0]$ will be justified by a bootstrap argument.
\end{remark}

We need to estimate $\dot{z}_j$ and $\ddot{z}_j$ in a more precise way rather than using the rough estimates in lemma \ref{velocity}. Let's first derive the estimate for $\dot{z}_j$, then we use this estimate to control $x(t)$ over time. We use the control of $x(t)$ to estimate $\ddot{z}_j$.

\begin{lemma}\label{refinederivative}
Assume the assumptions of Theorem \ref{longtime} and the bootstrap assumption (\ref{longtimeapriori}), and assume a priori that $d_I(t)\geq 1$, $\frac{1}{2}\leq \frac{x(t)}{x(0)}\leq 2$, $\forall~t\in [0,T_0]$. Then we have 
\begin{equation}
\sup_{t\in [0,T_0]}|\dot{z}_1(t)-\dot{z}_2(t)|\leq 10 \epsilon.
\end{equation}
\begin{equation}
\sup_{t\in [0,T_0]}|\dot{z_j}(t)-\frac{\lambda i}{4\pi x(t)}|\leq 5\epsilon.
\end{equation}
\begin{equation}
\sup_{t\in [0,T_0]}|\dot{z}_1^2-\dot{z}_2^2|\leq 6|\lambda| \epsilon+120\epsilon^2 x(t).
\end{equation}
\end{lemma}
\begin{proof}
Note that 
\begin{equation}\label{firstderivative1}
\dot{z}_1=\frac{\lambda_2 i}{2\pi(z_1-z_2)}+\bar{F}(z_1,t)=\frac{\lambda i}{4\pi x(t)}+\bar{F}(z_1,t).
\end{equation}
Similarly,
\begin{align*}
\dot{z}_2=\frac{\lambda i}{4\pi x(t)}+\bar{F}(z_2,t).
\end{align*}
By lemma \ref{maximum_principle}, we have  
\begin{align*}
|\dot{z}_1(t)-\dot{z}_2(t)|=|F(z_1,t)-F(z_2,t)|\leq 2\|F\|_{L^{\infty}(\partial \Omega(t))}\leq 10\epsilon,
\end{align*}
and
\begin{equation}
|\dot{z}_j(t)-\frac{\lambda i}{4\pi x(t)}|=|\bar{F}(z_1(t),t)|\leq 5\epsilon.
\end{equation}
We have 
$$\dot{z}_j^2= (\frac{\lambda i}{4\pi x(t)})^2+2\frac{\lambda i}{4\pi x(t)}\bar{F}(z_j(t),t)+(\bar{F}(z_j(t),t))^2,$$
By mean value theorem, bootstrap assumption (\ref{longtimeapriori}), lemma \ref{maximum_principle}, we have 
\begin{equation}\label{meanmean}
\begin{split}
    |F(z_1(t),t)^2-F(z_2(t),t)^2|=&|(F(z_1(t),t)+F(z_2(t),t))F_{\zeta}(\tilde{x}+iy(t),t)(z_1(t)-z_2(t))|\\
    \leq & 120\epsilon^2 x(t).
    \end{split}
\end{equation}
Here, $\tilde{x}\in (0,x(t))$.

By (\ref{max_real}) of  lemma \ref{maximum_principle}, we have 
\begin{align*}
|\dot{z}_1^2-\dot{z}_2^2|=&|\frac{2\lambda i}{4\pi x(t)}(F(z_1(t),t)-F(z_2(t),t))+F(z_1(t),t)^2-F(z_2(t),t)^2|\\
=&|\frac{\lambda i}{\pi x(t)}Re F(z_1(t),t)+F(z_1(t),t)^2-F(z_2(t),t)^2|\\
\leq & \frac{|\lambda|}{\pi x(t)}\|F_{\zeta}\|_{\infty}x(t)+|F(z_1(t),t)^2-F(z_2(t),t)^2|\\
\leq &6|\lambda| \epsilon+120\epsilon^2x(t).
\end{align*}
\end{proof}

Another consequence of the bootstrap assumption (\ref{longtimeapriori}) is the following description of the motion of the point vortices, which is the key control of this paper.
\begin{proposition}[key control]\label{parallel}
Assume the assumptions of Theorem \ref{longtime} and assume the bootstrap assumption (\ref{longtimeapriori}), we have
\begin{equation}\label{goodbound}
\frac{1}{2}\leq \frac{x(t)}{x(0)}\leq 2,\quad \quad 0\leq t\leq T_0.
\end{equation}
\end{proposition}
\begin{proof}
It suffices to prove the case that $x(t)$ is increasing on $0\leq t\leq T_0$. The case that $x(t)$ is decreasing follows in a similar way, and other cases are controlled by these two cases. 

Denote 
\begin{equation}\label{bootstrap}
\mathcal{T}:=\Big\{T\in [0,T_0] | \quad \dot{y}(t)\leq -\frac{|\lambda|}{20\pi x(0)}, \quad \frac{1}{2}\leq \frac{x(t)}{x(0)}\leq 2,\quad \hat{d}_I(t)\geq 1+\frac{|\lambda|}{20\pi x(0)}t, \quad \forall~t\in [0,T] \Big \}.
\end{equation}
Let's assume $F=F_1+iF_2$, where $F_1, F_2$ are real (we remind the readers that $F$ is the holomorphic extension of $f$.  Recall also the notations that $z_1(t)=-x(t)+iy(t)$, $z_2(t)=x(t)+iy(t)$, $x(t)>0, y(t)<0$). From the proof of lemma \ref{refinederivative}, we have  
\begin{equation}\label{velocityvortex}
\begin{cases}
\dot{z}_1(t)=\bar{F}(z_1(t),t)+\frac{\lambda_2 i}{2\pi}\frac{1}{\overline{z_1(t)-z_2(t)}}=\bar{F}(z_1(t))-\frac{|\lambda|i}{4\pi x(t)}\\
\dot{z}_2(t)=\bar{F}(z_2(t),t)+\frac{\lambda_1 i}{2\pi}\frac{1}{\overline{z_2(t)-z_1(t)}}=\bar{F}(z_2(t))-\frac{|\lambda|i}{4\pi x(t)}.
\end{cases}
\end{equation}
So we have 
\begin{equation}
\dot{y}(t)=-F_2(z_2(t),t)-\frac{|\lambda| }{4\pi x(t)}.
\end{equation}
By maximum principle and the bootstrap assumption (\ref{longtimeapriori}), we have 
\begin{equation}
|F_2(z_2(t),t)|\leq |F(z_2(t),t)|\leq \|F(\cdot,t)\|_{L^{\infty}(\Omega(t))}=\|\mathfrak{F}(\cdot,t)\|_{L^{\infty}(\mathbb{R})}\leq 5\epsilon.
\end{equation}
For $M$ relatively large (we take $M=200\pi$), at $t=0$, we have $\frac{|\lambda|}{4\pi x(0)}\geq \frac{200\pi\epsilon}{4\pi}=50\epsilon$. So we have 
\begin{equation}
\dot{y}(0)=-F_2(z_2(0), 0)-\frac{|\lambda|}{4\pi x(0)}\leq -\frac{9|\lambda|}{40\pi x(0)}.
\end{equation}
So $0\in\mathcal{T}$ and therefore $\mathcal{T}\neq \emptyset$. Clearly, by the definition of $\mathcal{T}$, since $\dot{y}(t), x(t)$ and $\hat{d}_I(t)$ are continuous, so $\mathcal{T}$ is closed in $[0,T_0]$. To prove $\mathcal{T}=[0,T_0]$, it suffices to prove that if (\ref{bootstrap}) holds on $[0,T]$ with $T<T_0$, then there exists $\delta>0$ such that (\ref{bootstrap}) holds on $[T,T+\delta)$.

Let $T\in \mathcal{T}$.

By (\ref{velocityvortex}), we have $\dot{x}(t)=ReF(z_2(t),t)$. Use (\ref{max_real}) of lemma \ref{maximum_principle}, use the fact that $Re F$ is odd, by mean value theorem, we have 
\begin{equation}\label{oddcancel}
\dot{x}(t)=ReF(z_2(t),t)-ReF(0+iy(t),t)=Re~F_x(\tilde{x}+iy(t),t)x(t),
\end{equation}
for some $\tilde{x}\in (0,x(t))$.

Since
\begin{equation}
F(z,t)=\frac{1}{2\pi i}\int \frac{\zeta_{\beta}}{z-\zeta(\beta,t)}\mathfrak{F}(\beta,t)d\beta.
\end{equation}
So we have $\forall~t\in [0,T_0]$,
\begin{equation}
\partial_x F(z,t)=-\frac{1}{2\pi i}\int \frac{\zeta_{\beta}}{(z-\zeta(\beta,t))^2}\mathfrak{F}(\beta,t)d\beta.
\end{equation}
By Cauchy-Schwartz inequality and lemma \ref{integral}, we have 
\begin{equation}\label{estimatetwo}
\begin{split}
|\partial_x F(\tilde{x}+iy(t),t)|\leq &\frac{1}{2\pi}\Big(\int \frac{1}{|\tilde{x}+iy(t)-\zeta(\beta,t)|^4}d\beta\Big)^{1/2}\|\zeta_{\beta}\|_{L^{\infty}}\|\mathfrak{F}\|_{L^2}\\
\leq & C\epsilon \hat{d}_I(t)^{-3/2}, \quad \quad \forall~t\in [0,T_0],
\end{split}
\end{equation}
for some absolute constant $C>0$. By direct calculation, we can see that $C\leq 2$.
So we obtain
\begin{equation}
\dot{x}(t)\leq 2\epsilon \hat{d}_I(t)^{-3/2}x(t). 
\end{equation}
So we have 
\begin{equation}
\frac{d}{dt}\ln\frac{x(t)}{x(0)}\leq 2\hat{d}_I(t)^{-3/2}\epsilon\leq 2(1+\frac{|\lambda|}{20\pi  x(0)}t)^{-3/2}\epsilon, \quad \quad \forall~ t\in [0, T].
\end{equation}
Then we have for all $t\in [0, T]$, 
\begin{equation}
\begin{split}
    x(t)\leq & x(0)exp\Big\{2\epsilon\int_0^t (1+\frac{|\lambda|}{20\pi x(0)}\tau)^{-3/2}d\tau\Big\}\\
    \leq & x(0)exp\Big\{   4\epsilon \frac{20\pi x(0)}{|\lambda|}\Big\}\\
    \leq & x(0)e^{2\epsilon  \frac{40\pi}{200\pi\epsilon}}=e^{\frac{2}{5}}x(0)\leq \frac{3}{2}x(0).
    \end{split}
\end{equation}
By the continuity of $x(t)$,  there exists $\delta>0$ such that 
\begin{equation}
    \sup_{t\in [0, T+\delta)} \frac{x(t)}{x(0)} \leq 2.
\end{equation}

\noindent Next we show that by choosing $\delta>0$ smaller if necessary, we have 
\begin{equation}\label{bootboot}
    \dot{y}(t)\leq -\frac{|\lambda|}{20\pi x(0)}, \quad \quad \hat{d}_I(t)\geq 1+\frac{|\lambda|}{20\pi x(0)}t, \quad \quad \forall~t\in [0, T+\delta).
\end{equation}
\noindent \textbf{Proof of (\ref{bootboot}):}
By the definition of $\mathcal{T}$, the bootstrap assumption (\ref{longtimeapriori}), and the fact that $\frac{|\lambda|}{x(0)}\geq 200\pi\epsilon$, we have 
\begin{equation}\label{dotyt}
\dot{y}(t)\leq 5\epsilon-\frac{|\lambda|}{8\pi x(0)}\leq -\frac{|\lambda|}{10\pi x(0)},\quad \quad \forall~ t\in [0,T].
\end{equation}
By Fundamental theorem of calculus, we have 
\begin{equation}
y(t)=y(0)+\int_0^t \dot{y}(\tau)d\tau.
\end{equation}
We have 
\begin{equation}
    \begin{split}
        \zeta(\alpha,t)-z_j(t)=&\zeta(\alpha,0)-z_j(0)+\int_0^t \partial_{\tau}(\zeta(\alpha,\tau)-z_j(\tau))d\tau\\
        =&\zeta(\alpha,0)-z_j(0)+\int_0^t D_{\tau}\zeta(\alpha,\tau)d\tau-\int_0^t b(\alpha,\tau)\partial_{\alpha}\zeta(\alpha,\tau)d\tau-\int_0^t \dot{z}_j(\tau)d\tau
    \end{split}
\end{equation}
So we have 
\begin{equation}
\begin{split}
    Im\{\zeta(\alpha,t)-z_j(t)\}=& Im\Big\{\zeta(\alpha,0)-z_j(0)-\int_0^t \dot{z}_j(\tau)d\tau\Big\}+Im\int_0^t D_{\tau}\zeta(\alpha,\tau)d\tau\\
    &-\int_0^t b(\alpha,\tau)Im\{\partial_{\alpha}\zeta(\alpha,\tau)\}d\tau.
    \end{split}
\end{equation}
By Sobolev embedding lemma \ref{sobolevembedding} and the bootstrap assumption, we have 
\begin{equation}
    \Big|\int_0^t D_{\tau}\zeta(\alpha,\tau)d\tau\Big|\leq 6\epsilon t,\quad \quad \forall ~t\in [0,T].
\end{equation}
By Corollary \ref{estimateforbb} and Sobolev embedding, we have 
\begin{equation}
 \Big| \int_0^t b(\alpha,\tau)\partial_{\alpha}\zeta(\alpha,\tau)d\tau \Big|\leq    (C\epsilon^2+K_s^{-1}\epsilon d_I(t)^{-3/2})(1+5\epsilon)t\leq (C\epsilon^2+K_s^{-1}\epsilon)t.
\end{equation}
Note that $Im\{\zeta(\alpha,0)-z_j(0)\}\geq \hat{d}_I(0)\geq 1$, $-Im\{\dot{z}_j(\tau)\}\geq \frac{|\lambda|}{10\pi x(0)}>0$, so we have for all $t\in [0,T]$,
\begin{equation}
    \begin{split}
        Im\{\zeta(\alpha,t)-z_j(t)\}
        \geq & \inf_{\alpha\in\mathbb{R}}Im\{\zeta(\alpha,0)-z_j(0)\}+(\frac{|\lambda|}{10\pi x(0)}-6\epsilon-(C\epsilon^2+K_s^{-1}\epsilon))t\\
        \geq &  1+\frac{|\lambda|}{18\pi x(0)}t.
    \end{split}
\end{equation}
So we have 
\begin{equation}\label{dit}
    \hat{d}_I(t)=\min_{j=1,2}\inf_{\alpha\in \mathbb{R}}Im \{\zeta(\alpha,t)-z_j(t)\}\geq  1+\frac{|\lambda|}{18\pi x(0)}t,\quad \quad\forall~t\in [0,T].
\end{equation}
By (\ref{dotyt}), (\ref{dit}), the continuity of $\hat{d}_I(t)$, and the continuity of $y'(t)$, choosing $\delta>0$ smaller if necessary, we have (\ref{bootboot}). 

\vspace*{1ex}

Since $\mathcal{T}$ is both closed and open as a subspace of $[0,T_0]$,  so we must have 
\begin{equation}
    \mathcal{T}=[0,T_0],
\end{equation}
which concludes the proof of the lemma.
\end{proof}

\noindent Because 
$$d_I(t)=\min_{j=1,2}\inf_{\alpha\in \mathbb{R}}|\zeta(\alpha,t)-z_j(t)|\geq \hat{d}_I(t),$$ we have the following estimate.
\begin{corollary}[Decay estimate]\label{decayestimate}
Assume the assumptions of Theorem \ref{longtime} and assume the bootstrap assumption (\ref{longtimeapriori}), we have $\forall~t\in [0,T_0]$,
\begin{equation}
d_I(t)^{-1}\leq (1+\frac{|\lambda|}{20\pi x(0)} t)^{-1}.
\end{equation}
\end{corollary}

We need to estimate $\ddot{z}_j$ and $\ddot{z}_1-\ddot{z}_2$ as well.

\vspace*{2ex}

\noindent \textbf{Convention:} From now on, if the domain of $t$ is not specified, we assume $t\in [0, T_0]$ by default. 

\vspace*{2ex}

\begin{lemma}\label{refinesecondderivative}
Assume the assumptions of Theorem \ref{longtime} and assume the bootstrap assumption (\ref{longtimeapriori}), we have $\forall~t\in [0,T_0]$,
\begin{equation}\label{zjddot}
|\ddot{z}_j(t)|\leq 10\epsilon+\frac{6|\lambda|}{x(t)}\epsilon.
\end{equation}
\begin{equation}
|\ddot{z}_1(t)-\ddot{z}_2(t)|\leq 220\epsilon^2x(t)+\epsilon(20x(t)+\frac{5|\lambda|}{\pi})
\end{equation}
\end{lemma}
\begin{proof}
Take time derivative of (\ref{firstderivative1}), we have 
\begin{equation}
\ddot{z}_j(t)=-\frac{\lambda i x'(t)}{4\pi x(t)^2}+\bar{F}_\zeta(z_j(t),t)\dot{z}_j(t)+\bar{F}_t(z_j,t).
\end{equation}
We have $x'(t)=Re F(z_2(t),t)$.
By lemma \ref{maximum_principle}, we have 
\begin{equation}
|F_{\zeta}(z_j(t),t)|\leq \|F_{\zeta}(\cdot,t)\|_{L^{\infty}(\Omega(t))}\leq \|F_{\zeta}(\zeta(\alpha,t),t)\|_{\infty}\leq 6\epsilon.
\end{equation}
By Sobolev embedding and lemma \ref{maximum_principle}, we have 
\begin{equation}
\|F_{t}(\cdot,t)\|_{L^{\infty}(\Omega(t))}\leq \|F_{t}(\zeta(\alpha,t),t)\|_{L^{\infty}}\leq \|F_t\circ\zeta\|_{H^1}\leq 6\epsilon.
\end{equation}
Apply lemma \ref{maximum_principle} again, we have 
\begin{align*}
|Re F(z_j(t),t)|\leq 6\epsilon x(t).
\end{align*}
So we obtain
\begin{align*}
|\ddot{z}_j(t)|\leq &\frac{|\lambda| |Re F(z_j,t)|}{4\pi x(t)^2}+|F_{\zeta}(z_j,t)| |\dot{z}_j(t)|+|F_t(z_j,t)|\\
\leq &\frac{6|\lambda|\epsilon}{4\pi x(t)}+6\epsilon (\frac{|\lambda|}{4\pi x(t)}+6\epsilon)+6\epsilon\\
\leq & \frac{6 |\lambda| \epsilon}{\pi x(t)}+10\epsilon.
\end{align*}
Here, we assume $\epsilon^2$ sufficiently small such that $36\epsilon^2\leq 4\epsilon$. We have 
\begin{align*}
&|\ddot{z}_1(t)-\ddot{z}_2(t)|\\
=&|\bar{F}_{\zeta}(z_1(t),t)\dot{z}_1(t)+\bar{F}_t(z_1(t),t)-\bar{F}_{\zeta}(z_2(t),t)\dot{z}_2(t)-\bar{F}_t(z_2(t),t)|\\
\leq & |F_{\zeta}(z_1(t),t)-F_{\zeta}(z_2(t),t)||\dot{z}_1(t)|+|F_{\zeta}(z_2(t),t)||\dot{z}_1-\dot{z}_2|+|F_t(z_1,t)-F_t(z_2,t)|\\
\leq & \|F_{\zeta\zeta}\|_{\infty}|z_1-z_2| |\dot{z}_1(t)|+|F_{\zeta}(z_2(t),t)||F(z_1(t),t)-F(z_2(t),t)|+\|(Re F)_{t\zeta}\|_{L^\infty(\Omega(t))}|z_1-z_2|.
\end{align*} 
By lemma \ref{maximum_principle},
\begin{equation}
\|F_{\zeta\zeta}\|_{\infty}\leq 10\epsilon,
\end{equation}
and
\begin{equation}
    \|F_{t\zeta}\|_{L^{\infty}(\Omega(t)}\leq 10\epsilon.
\end{equation}
Since $Re F_t$ is odd in $x$ and $Im F_t$ is even in $x$, by mean value theorem, we have 
\begin{align}
|F_t(z_1,t)-F_t(z_2,t)|=&2|Re F_t(z_2,t)-Re F_t(0,y,t)|x=2|Re F_{tx}(\tilde{x},y,t)|x(t)\\
\leq &2\|F_{t\zeta}\|_{L^{\infty}(\Omega(t))} x(t)\leq 20\epsilon x(t).
\end{align}
for some $\tilde{x}\in (0,x(t))$.

So we obtain
\begin{equation}
\begin{split}
&|\ddot{z}_1(t)-\ddot{z}_2(t)|\\
\leq & 10\epsilon(2x(t))(\frac{|\lambda|}{4\pi x(t)}+5\epsilon)+(6\epsilon) 20\epsilon x(t)+(10\epsilon)2x(t)\\
=& 220\epsilon^2x(t)+\epsilon(20x(t)+\frac{5|\lambda|}{\pi}).
\end{split}
\end{equation}
\end{proof}

Next, we estimate the quantity $\norm{\sum_{j=1}^2 \frac{\lambda_j i}{2\pi}\frac{\dot{z}_j}{(\zeta(\alpha,t)-z_j(t))^2}}_{H^s}$, the quantity $\norm{\sum_{j=1}^2 \frac{\lambda_j i}{2\pi}\frac{\ddot{z}_j}{(\zeta(\alpha,t)-z_j(t))^2}}_{H^s}$, and the quantity $\norm{\sum_{j=1}^2 \frac{\lambda_j i}{2\pi}\frac{(\dot{z}_j)^2}{(\zeta(\alpha,t)-z_j(t))^2}}_{H^s}$. These quantities arise from the energy estimates.
\begin{lemma}\label{dotzjdotzj}
Assume the assumptions of Theorem \ref{longtime} and assume the bootstrap assumption (\ref{longtimeapriori}). Then we have
\begin{equation}\label{veryrough}
\norm{\sum_{j=1}^2 \frac{\lambda_j i}{2\pi}\frac{\dot{z}_j}{(\zeta(\alpha,t)-z_j(t))^2}}_{H^s}\leq K_s^{-1}\epsilon \frac{|\lambda|}{x(0)}d_I(t)^{-5/2}+C\epsilon^2.
\end{equation}
\end{lemma}
\begin{proof}
Replace $\dot{z}_j$ by
$$\dot{z}_j(t)=\frac{\lambda i}{4\pi x(t)}+\bar{F}(z_j(t),t).$$
We have 
\begin{align*}
\sum_{j=1}^2 \frac{\lambda_j i}{2\pi}&\frac{\dot{z}_j}{(\zeta(\alpha,t)-z_j(t))^2}=\sum_{j=1}^2 \frac{\lambda_j i}{2\pi}\frac{\frac{\lambda i}{4\pi x(t)}}{(\zeta(\alpha,t)-z_j(t))^2}+\bar{F}(z_1(t),t)\sum_{j=1}^2 \frac{\lambda_j i}{2\pi}\frac{1}{(\zeta(\alpha,t)-z_j(t))^2}\\
&+\frac{\lambda i (\bar{F}(z_1(t),t)-\bar{F}(z_2(t),t))}{2\pi}\frac{1}{(\zeta(\alpha,t)-z_2(t))^2}
\end{align*}

By lemma \ref{smallestimate} (and use the proof of lemma \ref{smallestimate} to estimate the term $\norm{\frac{1}{(\zeta(\alpha,t)-z_2(t))^2}}_{H^s}$), we have 
\begin{align*}
&\norm{\sum_{j=1}^2 \frac{\lambda_j i}{2\pi}\frac{\dot{z}_j}{(\zeta(\alpha,t)-z_j(t))^2}}_{H^s}\\
\leq & \frac{|\lambda|}{4\pi x(t)}\norm{\sum_{j=1}^2\frac{\lambda_j i}{2\pi} \frac{1}{(\zeta(\alpha,t)-z_j(t))^2}}_{H^s}+\norm{\sum_{j=1}^2 \frac{\lambda_j i}{2\pi}\frac{1}{(\zeta(\alpha,t)-z_j(t))^2}}_{H^s}\|F\|_{L^\infty(\Omega(t))}\\
&+\norm{\frac{\lambda i (\bar{F}(z_1(t),t)-\bar{F}(z_2(t),t))}{2\pi}\frac{1}{(\zeta(\alpha,t)-z_2(t))^2}}_{H^s}\\
\leq &\frac{|\lambda|}{4\pi x(t)}K_s^{-1}\epsilon d_I(t)^{-5/2}+K_s^{-1}\epsilon d_I(t)^{-5/2}(5\epsilon)+\Big|\frac{F(z_1(t),t)-F(z_2(t),t)}{x(t)}\Big| \norm{\frac{\lambda x(t)}{(\zeta(\alpha,t)-z_2(t))^2}}_{H^s}\\
\leq & K_s^{-1}\epsilon \frac{|\lambda|}{x(0)}d_I(t)^{-5/2}+K_s^{-1}\epsilon^2 d_I(t)^{-5/2}+K_s^{-1}\epsilon^2 d_I(t)^{-3/2}.
\end{align*}
Here, we use lemma \ref{maximum_principle} to estimate 
\begin{equation}
    \Big|\frac{F(z_1(t),t)-F(z_2(t),t)}{x(t)}\Big|\leq 12\epsilon,
\end{equation}
and we use the proof of lemma \ref{smallestimate} to estimate 
\begin{equation}
    \norm{\frac{\lambda x(t)}{(\zeta(\alpha,t)-z_2(t))^2}}_{H^s}\leq K_s^{-1}\epsilon d_I(t)^{-3/2}.
\end{equation}
Since $d_I(t)\geq 1$, we simply estimate $\norm{\sum_{j=1}^2 \frac{\lambda_j i}{2\pi}\frac{\dot{z}_j}{(\zeta(\alpha,t)-z_j(t))^2}}_{H^s}$ by (\ref{veryrough}).
\end{proof}

\begin{lemma}\label{dotzjsquare}
Assume the assumptions of Theorem \ref{longtime} and assume the bootstrap assumption (\ref{longtimeapriori}). Then we have
\begin{equation}
\norm{\sum_{j=1}^2 \frac{\lambda_j i}{2\pi}\frac{(\dot{z}_j)^2}{(\zeta(\alpha,t)-z_j(t))^3}}_{H^s}\leq K_s^{-1}\epsilon^2+K_s^{-1}\epsilon\frac{|\lambda|}{x(0)}d_I(t)^{-5/2}.
\end{equation}
\end{lemma}
\begin{proof}
\begin{align*}
\sum_{j=1}^2 \frac{\lambda_j i}{2\pi}\frac{(\dot{z}_j)^2}{(\zeta(\alpha,t)-z_j(t))^3}
=&\frac{\lambda i (\dot{z}_1)^2}{2\pi}(\frac{1}{(\zeta(\alpha,t)-z_1(t))^3}-\frac{1}{(\zeta(\alpha,t)-z_2(t))^3})\\
&+\frac{\lambda i ((\dot{z}_1)^2-(\dot{z}_2)^2)}{2\pi}\frac{1}{(\zeta(\alpha,t)-z_2(t))^3}:=I+\it{II}.
\end{align*}
We have 
\begin{align*}
|I|=\Big|\frac{\lambda i(\dot{z}_1)^2}{2\pi}2x(t)\Big\{&\frac{1}{(\zeta(\alpha,t)-z_1(t))^3(\zeta(\alpha,t)-z_2(t))}+\frac{1}{(\zeta(\alpha,t)-z_1(t))^2(\zeta(\alpha,t)-z_2(t)^2)}\\
&+\frac{1}{(\zeta(\alpha,t)-z_1(t))(\zeta(\alpha,t)-z_2(t))^3}\Big\}\Big|,
\end{align*}
and
$$\dot{z}_1^2=-\frac{\lambda^2}{16\pi^2 x(t)^2}+\frac{\lambda i}{2\pi x(t)}\bar{F}(z_1(t),t)+(\bar{F}(z_1(t),t))^2.$$
Use the proof of lemma \ref{smallestimate}, it's easy to see that 
\begin{equation}
\norm{\frac{1}{(\zeta(\alpha,t)-z_1(t))^3(\zeta(\alpha,t)-z_2(t))}}_{H^s}\leq ((s+6)!)^2 d_I(t)^{-7/2}
\end{equation}
\begin{equation}
\norm{\frac{1}{(\zeta(\alpha,t)-z_1(t))^2(\zeta(\alpha,t)-z_2(t)^2)}}_{H^s}\leq ((s+6)!)^2 d_I(t)^{-7/2}
\end{equation}
\begin{equation}
\norm{\frac{1}{(\zeta(\alpha,t)-z_1(t))(\zeta(\alpha,t)-z_2(t)^3)}}_{H^s}\leq ((s+6)!)^2 d_I(t)^{-7/2}
\end{equation}
Use the assumption that $\lambda^2+|\lambda x(0)|\leq \frac{1}{((s+12)!)^2}\epsilon$ and the fact that $\frac{1}{2}x(0)\leq x(t)\leq 2x(0)$, we have 

\begin{align*}
\|I\|_{H^s}\leq & \frac{|\lambda | x(t)}{\pi}(\frac{\lambda^2}{16\pi^2 x(t)^2}+\frac{|\lambda|}{2\pi x(t)}\times 5\epsilon+25\epsilon^2) ((s+6)!)^2 d_I(t)^{-7/2}\\
\leq &K_s^{-1}\epsilon^2d_I(t)^{-7/2}+K_s^{-1}\epsilon\frac{|\lambda|}{x(0)}d_I(t)^{-7/2}.
\end{align*}
By lemma \ref{refinederivative}, we have 
\begin{align*}
\|\it{II}\|_{H^s}\leq & \frac{|\lambda| |\dot{z}_1^2-\dot{z}_2^2|}{2\pi}\norm{\frac{1}{(\zeta(\alpha,t)-z_2(t))^3}}_{H^s}\\
\leq & \frac{|\lambda| (6|\lambda| \epsilon+120\epsilon^2x(t))}{2\pi}((s+6)!)^2 d_I(t)^{-5/2}\\
\leq & K_s^{-1} \epsilon^2 d_I(t)^{-5/2}.
\end{align*}
Here, we use the assumption 
\begin{equation}
\lambda^2+|\lambda x(0)|\leq c_0\epsilon,\quad \quad c_0=\frac{1}{((s+12)!)^2}.
\end{equation}
So we we have 
\begin{align*}
\norm{\sum_{j=1}^2 \frac{\lambda_j i}{2\pi}\frac{(\dot{z}_j)^2}{(\zeta(\alpha,t)-z_j(t))^3}}_{H^s}\leq  K_s^{-1}\epsilon^2+K_s^{-1}\epsilon\frac{|\lambda|}{x(0)}d_I(t)^{-5/2}.
\end{align*}
\end{proof}

\begin{lemma}\label{ddotzjnorm}
Assume the assumptions of Theorem \ref{longtime} and assume the bootstrap assumption (\ref{longtimeapriori}). Then we have 
\begin{equation}
\norm{\sum_{j=1}^2 \frac{\lambda_j i}{2\pi}\frac{\ddot{z}_j}{(\zeta(\alpha,t)-z_j(t))^2}}_{H^s}\leq K_s^{-1}\epsilon^2d_I(t)^{-5/2}.
\end{equation}
\end{lemma}
\begin{proof}
We have
\begin{align*}
&\norm{\sum_{j=1}^2 \frac{\lambda_j i}{2\pi}\frac{\ddot{z}_j}{(\zeta(\alpha,t)-z_j(t))^2}}_{H^s}\\
\leq & \norm{\sum_{j=1}^2\frac{\lambda_j i \ddot{z}_1(t)}{2\pi}\frac{1}{(\zeta(\alpha,t)-z_j(t))^2}}_{H^s}
+\norm{\frac{\lambda i (\ddot{z}_1(t)-\ddot{z}_2(t))}{2\pi}\frac{1}{(\zeta(\alpha,t)-z_2(t))^2}}_{H^s}:=I+\it{II}.
\end{align*}
By the proof of lemma \ref{smallestimate} and by lemma \ref{refinesecondderivative}, we have 
\begin{align*}
I\leq &|\ddot{z}_1(t)|\|\sum_{j=1}^2\frac{\lambda_j i }{2\pi}\frac{1}{(\zeta(\alpha,t)-z_j(t))^2}\|_{H^s}\\
\leq & (10\epsilon+\frac{6|\lambda|}{x(t)}\epsilon) |\lambda x(0)|((s+6)!)^2 d_I(t)^{-5/2}\\
\leq & K_s^{-1}\epsilon^2 d_I(t)^{-5/2}.
\end{align*}
By lemma \ref{smallestimate} and lemma \ref{refinesecondderivative}, we have 
\begin{align*}
\it{II}\leq &|\ddot{z}_1(t)-\ddot{z}_2(t)|\norm{\frac{\lambda i }{2\pi}\frac{1}{(\zeta(\alpha,t)-z_2(t))^2}}_{H^s}\\
\leq &\Big(220\epsilon^2x(t)+\epsilon(20x(t)+\frac{5|\lambda|}{\pi})\Big)|\lambda| ((s+6)!)^2 d_I(t)^{-5/2}\\
\leq & K_s^{-1}\epsilon^2 d_I(t)^{-5/2}.
\end{align*}
Here, we've used the fact that $\lambda^2+|\lambda x(0)|\leq \frac{1}{((s+12)!)^2}\epsilon$.
So we obtain
\begin{equation}
\norm{\sum_{j=1}^2 \frac{\lambda_j i}{2\pi}\frac{\ddot{z}_j}{(\zeta(\alpha,t)-z_j(t))^2}}_{H^s}\leq K_s^{-1}\epsilon^2 d_I(t)^{-5/2}.
\end{equation}
\end{proof}

\subsection{Estimates for quantities involved in the energy estimates.}\label{estimateinvolve} In this subsection, we derive estimates for various quantities that show up in energy estimates. 

\vspace*{2ex}

\subsubsection{Control $\|\partial_{\alpha}\tilde{\theta}\|_{H^s}$ by $\|\zeta_{\alpha}-1\|_{H^s}$.}
\begin{lemma}\label{equivalent_1}
Assume the assumptions of Theorem \ref{longtime} and assume the assumption (\ref{longtimeapriori}), for $1\leq k\leq s+1$, we have 
\begin{equation}\label{okokokok}
\|\partial_{\alpha}^k \tilde{\theta}- 2\partial_{\alpha}^{k-1}(\zeta_{\alpha}-1)\|_{L^2}\leq C\epsilon^2.
\end{equation}
\end{lemma}
\begin{proof}
Since $(I-\mathcal{H})(\bar{\zeta}-\alpha)=0$, we have $(I+\bar{\mathcal{H}})(\zeta-\alpha)=2(\zeta-\alpha)$. Therefore,
\begin{align*}
\partial_{\alpha}^k \tilde{\theta}=&\partial_{\alpha}^k (I-\mathcal{H})(\zeta-\bar{\zeta})=\partial_{\alpha}^k (I-\mathcal{H})(\zeta-\alpha)\\
=& \partial_{\alpha}^k(I+\bar{\mathcal{H}}-(\bar{\mathcal{H}}+\mathcal{H}))(\zeta-\alpha)\\
=& 2\partial_{\alpha}^{k-1}(\zeta_{\alpha}-1)-\partial_{\alpha}^k(\bar{\mathcal{H}}+\mathcal{H}))(\zeta-\alpha).
\end{align*}
It's easy to obtain that for $1\leq k\leq s+1$,
\begin{equation}
\norm{\partial_{\alpha}^k(\bar{\mathcal{H}}+\mathcal{H})(\zeta-\alpha)}_{L^2}\leq C\norm{\zeta_{\alpha}-1}_{H^s}^2\leq C\epsilon^2.
\end{equation}

\noindent So we have 
\begin{equation}
\begin{split}
\norm{\partial_{\alpha}^k \tilde{\theta}-2\partial_{\alpha}^{k-1}(\zeta_{\alpha}-1)}_{L^2}\leq C\epsilon^2.
\end{split}
\end{equation}
So we obtain (\ref{okokokok}).
\end{proof}
\begin{corollary}\label{onetotwo}
Assume the assumptions of Theorem \ref{longtime} and the bootstrap assumption (\ref{longtimeapriori}), we have 
\begin{equation}
\norm{\partial_{\alpha} \tilde{\theta}}_{H^s}\leq 11\epsilon.
\end{equation}
\end{corollary}

\subsubsection{Compare $\norm{D_t\tilde{\theta}}_{H^s}$ with $\norm{D_t\zeta}_{H^s}$ and $\norm{D_t\tilde{\sigma}}_{H^s}$ with $\norm{D_t^2\zeta}_{H^s}$} We need to show that $D_t\tilde{\theta}$ and $D_t\zeta$ are equivalent in certain sense. 
We have the following:
\begin{lemma}\label{comparetransform}
Assume the assumptions of Theorem \ref{longtime} and a priori assumption (\ref{longtimeapriori}), we have 
\begin{equation}\label{equivtheta}
\norm{D_t\tilde{\theta}-2(\bar{\mathfrak{F}}-q)}_{H^{s+1/2}}\leq C\epsilon^2.
\end{equation}
\begin{equation}\label{equivsigma}
\norm{D_t\tilde{\sigma}-4(D_t \bar{\mathfrak{F}}-D_tq)}_{H^s}\leq C\epsilon^2.
\end{equation}
\end{lemma}
\begin{proof}
Recall that $D_t\zeta=\bar{\mathfrak{F}}+\bar{q}$, where $(I-\mathcal{H})\mathfrak{F}=0$, $(I+\mathcal{H})q=0$. So we have
\begin{align*}
(I+\bar{\mathcal{H}}))\bar{\mathfrak{F}}=2\bar{\mathfrak{F}}, \quad \quad (I+\bar{\mathcal{H}})\bar{q}=0.
\end{align*}
We have 
\begin{equation}\label{comparetheta}
\begin{split}
D_t\tilde{\theta}=&D_t (I-\mathcal{H})(\zeta-\bar{\zeta})=(I-\mathcal{H})(D_t\zeta-D_t\bar{\zeta})-[D_t\zeta,\mathcal{H}]\frac{\partial_{\alpha}(\zeta-\bar{\zeta})}{\zeta_{\alpha}}\\
=&(I-\mathcal{H})(\bar{\mathfrak{F}}+\bar{q}-\mathfrak{F}-q)-[D_t\zeta,\mathcal{H}]\frac{\partial_{\alpha}(\zeta-\bar{\zeta})}{\zeta_{\alpha}}\\
=&(I+\bar{\mathcal{H}})\bar{\mathfrak{F}}+(I+\bar{\mathcal{H}})\bar{q}-(\mathcal{H}+\bar{\mathcal{H}})D_t\zeta-2q-[D_t\zeta,\mathcal{H}]\frac{\partial_{\alpha}(\zeta-\bar{\zeta})}{\zeta_{\alpha}}\\
=&2\bar{\mathfrak{F}}-2q-(\mathcal{H}+\bar{\mathcal{H}})D_t\zeta-[D_t\zeta,\mathcal{H}]\frac{\partial_{\alpha}(\zeta-\bar{\zeta})}{\zeta_{\alpha}}.
\end{split}
\end{equation}
It's easy to obtain that under a priori assumption (\ref{longtimeapriori}),
\begin{equation}
\norm{-(\mathcal{H}+\bar{\mathcal{H}})D_t\zeta-[D_t\zeta,\mathcal{H}]\frac{\partial_{\alpha}(\zeta-\bar{\zeta})}{\zeta_{\alpha}}}_{H^{s+1/2}}\leq C\epsilon\norm{D_t\zeta}_{H^{s+1/2}}\leq C\epsilon^2,
\end{equation}
for some absolute constant $C>0$.

By triangle inequality, 
\begin{equation}
\begin{split}
\norm{D_t\tilde{\theta}-2(\bar{\mathfrak{F}}-q)}_{H^{s+1/2}}\leq &  \norm{-(\mathcal{H}+\bar{\mathcal{H}})D_t\zeta-[D_t\zeta,\mathcal{H}]\frac{\partial_{\alpha}(\zeta-\bar{\zeta})}{\zeta_{\alpha}}}_{H^{s+1/2}}\\
\leq & C\epsilon^2.
\end{split}
\end{equation}
So we obtain (\ref{equivtheta}).

By (\ref{comparetheta}), use
\begin{equation}
    (I-\mathcal{H})\bar{\mathfrak{F}}=2\bar{\mathfrak{F}}-(\bar{\mathcal{H}}+\mathcal{H})\bar{\mathfrak{F}},\quad \quad (I-\mathcal{H})q=2q,
\end{equation}
we have 
\begin{equation}
\begin{split}
    D_t\tilde{\sigma}=&D_t(I-\mathcal{H})D_t\tilde{\theta}=D_t(I-\mathcal{H})\Big\{2\bar{\mathfrak{F}}-2q-(\mathcal{H}+\bar{\mathcal{H}})D_t\zeta-[D_t\zeta,\mathcal{H}]\frac{\partial_{\alpha}(\zeta-\bar{\zeta})}{\zeta_{\alpha}}\Big\}
    \\
    =&4D_t\bar{\mathfrak{F}}-4D_tq+D_t(I-\mathcal{H})\Big\{-(\mathcal{H}+\bar{\mathcal{H}})D_t\zeta-[D_t\zeta,\mathcal{H}]\frac{\partial_{\alpha}(\zeta-\bar{\zeta})}{\zeta_{\alpha}}\Big\}-D_t(\bar{\mathcal{H}}+\mathcal{H})\bar{\mathfrak{F}}.
    \end{split}
\end{equation}
Therefore,
\begin{align}
   & \norm{D_t\tilde{\sigma}-4(D_t\bar{\mathfrak{F}}- D_tq)}_{H^s}\\
   = & \norm{D_t(I-\mathcal{H})\Big\{-(\mathcal{H}+\bar{\mathcal{H}})D_t\zeta-[D_t\zeta,\mathcal{H}]\frac{\partial_{\alpha}(\zeta-\bar{\zeta})}{\zeta_{\alpha}}\Big\}-D_t(\bar{\mathcal{H}}+\mathcal{H})\bar{\mathfrak{F}}}_{H^s}\\
    \leq & C\epsilon^2.
\end{align}
\end{proof}

\begin{corollary}\label{aprioriequiv}
Assume the bootstrap assumption (\ref{longtimeapriori}), we have 
\begin{equation}
\norm{D_t\tilde{\theta}}_{H^{s+1/2}}\leq 11\epsilon,\quad \quad \norm{D_t\tilde{\sigma}}_{H^s}\leq 21\epsilon.
\end{equation}
\end{corollary}

\subsubsection{Estimate the quantity $\frac{a_t}{a}\circ\kappa^{-1}$.} Recall that 
\begin{equation}
\begin{split}
&(I-\mathcal{H})\frac{a_t}{a}\circ\kappa^{-1}A\bar{\zeta}_{\alpha}\\=&2i[D_t^2\zeta, \mathcal{H}]\frac{\partial_{\alpha}D_t\bar{\zeta}}{\zeta_{\alpha}}+2i[D_t\zeta, \mathcal{H}] \frac{\partial_{\alpha}D_t^2\bar{\zeta}}{\zeta_{\alpha}}-\frac{1}{\pi }\int \Big(\frac{D_t\zeta(\alpha,t)-D_t\zeta(\beta,t)}{\zeta(\alpha,t)-\zeta(\beta,t)}\Big)^2 (D_t\bar{\zeta})_{\beta}d\beta\\
&-\frac{1}{\pi}\sum_{j=1}^2 \lambda_j\Big(\frac{2D_t^2\zeta+i-\partial_t^2 z_j}{(\zeta(\alpha,t)-z_j(t))^2}-2\frac{(D_t\zeta-\dot{z}_j(t))^2}{(\zeta(\alpha,t)-z_j(t))^3}\Big)
\end{split}
\end{equation}
By lemma \ref{lemmacommutator1}, the a priori assumption (\ref{longtimeapriori}), we have 
\begin{equation}
\norm{2i[D_t^2\zeta, \mathcal{H}]\frac{\partial_{\alpha}D_t\bar{\zeta}}{\zeta_{\alpha}}}_{H^s}\leq C\|D_t^2\zeta\|_{H^s}\|D_t\zeta\|_{H^s}\leq C\epsilon^2.
\end{equation}
\begin{equation}
\norm{2i[D_t\zeta, \mathcal{H}] \frac{\partial_{\alpha}D_t^2\bar{\zeta}}{\zeta_{\alpha}}}_{H^s}\leq C\|D_t\zeta\|_{H^s}\|D_t^2\zeta\|_{H^s}\leq C\epsilon^2.
\end{equation}
\begin{equation}
\norm{\frac{1}{\pi }\int \Big(\frac{D_t\zeta(\alpha,t)-D_t\zeta(\beta,t)}{\zeta(\alpha,t)-\zeta(\beta,t)}\Big)^2 (D_t\bar{\zeta})_{\beta}d\beta}_{H^s}\leq \|D_t\zeta\|_{H^s}^3\leq C(5\epsilon)^3\leq C\epsilon^2.
\end{equation}
\begin{equation}
\norm{\frac{1}{\pi}\sum_{j=1}^2 \frac{2\lambda_j D_t^2\zeta}{(\zeta(\alpha,t)-z_j(t))^2}}_{H^s}\leq \|D_t^2\zeta\|_{H^s}\norm{\frac{2}{\pi}\sum_{j=1}^2\frac{\lambda_j}{(\zeta(\alpha,t)-z_j(t))^2}}_{H^s}
\end{equation}
By lemma \ref{smallestimate}, we have 
\begin{equation}
\norm{\frac{2}{\pi}\sum_{j=1}^2\frac{\lambda_j }{(\zeta(\alpha,t)-z_j(t))^2}}_{H^s}\leq K_s^{-1}\epsilon d_I(t)^{-5/2}
\end{equation}
So we have 
\begin{equation}
\norm{\frac{1}{\pi}\sum_{j=1}^2 \frac{2\lambda_j D_t^2\zeta}{(\zeta(\alpha,t)-z_j(t))^2}}_{H^s}\leq \|D_t^2\zeta\|_{H^s}\norm{\frac{1}{\pi}\sum_{j=1}^2 \frac{2\lambda_j }{(\zeta(\alpha,t)-z_j(t))^2}}_{H^s}\leq K_s^{-1}\epsilon^2.
\end{equation}
By lemma \ref{dotzjdotzj}, we have 
\begin{equation}
\begin{split}
\norm{\sum_{j=1}^2\frac{\lambda_j i}{\pi}\frac{2D_t\zeta \dot{z}_j}{(\zeta(\alpha,t)-z_j(t))^3}}_{H^s}\leq &2\|D_t\zeta\|_{H^s}\norm{\sum_{j=1}^2\frac{\lambda_j i}{\pi}\frac{\dot{z}_j}{(\zeta(\alpha,t)-z_j(t))^3}}_{H^s}\\
\leq & 12\epsilon K_s^{-1}\epsilon d_I(t)^{-5/2}\leq K_s^{-1}\epsilon^2 d_I(t)^{-5/2}.
\end{split}
\end{equation}
By lemma \ref{dotzjsquare} and lemma \ref{ddotzjnorm}, we have 
\begin{equation}
\begin{split}
&\norm{\frac{1}{\pi}\sum_{j=1}^2 \frac{\lambda_j \ddot{z}_j(t)}{(\zeta(\alpha,t)-z_j(t))^2}}_{H^s}+\norm{\frac{1}{\pi}\sum_{j=1}^2 \frac{\lambda_j 2(\dot{z}_j(t))^2}{(\zeta(\alpha,t)-z_j(t))^3}}_{H^s}\\
\leq & K_s^{-1}\epsilon\frac{|\lambda|}{x(0)}d_I(t)^{-5/2}+K_s^{-1}\epsilon^2d_I(t)^{-5/2}.
\end{split}
\end{equation}
So we obtain
\begin{equation}
\norm{(I-\mathcal{H})\frac{a_t}{a}\circ\kappa^{-1}A\bar{\zeta}_{\alpha}}_{H^s}\leq C\epsilon^2+K_s^{-1}\epsilon\frac{|\lambda|}{x(0)}d_I(t)^{-5/2}.
\end{equation}
By lemma \ref{realinverse} and Sobolev embedding, we have 
\begin{equation}\label{ataA}
\norm{\frac{a_t}{a}\circ\kappa^{-1}}_{\infty}\leq C\epsilon^2+K_s^{-1}\epsilon\frac{|\lambda|}{x(0)}d_I(t)^{-5/2}.
\end{equation}

\subsubsection{Estimate the quantity $A$.}  Recall that 
\begin{equation}
\begin{split}
(I-\mathcal{H})A=&1+i[D_t\zeta,\mathcal{H}]\frac{\partial_{\alpha}\mathfrak{F}}{\zeta_{\alpha}}+i[D_t^2\zeta,\mathcal{H}]\frac{\bar{\zeta}_{\alpha}-1}{\zeta_{\alpha}}-(I-\mathcal{H})\frac{1}{2\pi}\sum_{j=1}^2 \frac{\lambda_j(D_t\zeta(\alpha,t)-\dot{z}_j(t))}{(\zeta(\alpha,t)-z_j(t))^2}.
\end{split}
\end{equation}
By lemma \ref{lemmacommutator1}, lemma \ref{smallestimate}, lemma \ref{dotzjdotzj}, we have 
\begin{align*}
\|(I-\mathcal{H})(A-1)\|_{H^s}\leq &\|D_t\zeta\|_{H^s}\|\mathfrak{F}\|_{H^s}+\|D_t^2\zeta\|_{H^s}\|\zeta_{\alpha}-1\|_{H^s}+\frac{1}{\pi}\|D_t\zeta\|_{H^s}\norm{\sum_{j=1}^2\frac{\lambda_j}{(\zeta(\alpha,t)-z_j(t))^2}}_{H^s}\\
&+\norm{\frac{1}{\pi}\sum_{j=1}^2 \frac{\lambda_j \dot{z}_j(t)}{(\zeta(\alpha,t)-z_j(t))^2}}_{H^s}\\
\leq &C\epsilon^2+K_s^{-1}\epsilon d_I(t)^{-5/2}.
\end{align*}
So we have 
\begin{equation}\label{estimateAminusOne}
\|A-1\|_{H^s}\leq C\epsilon^2+K_s^{-1}\epsilon d_I(t)^{-5/2}.
\end{equation}
\begin{corollary}\label{lowerboundA}
Assume the assumptions of Theorem \ref{longtime} and assume the bootstrap assumption \ref{longtimeapriori}. For $\epsilon$ sufficiently small, we have 
\begin{equation}
\inf_{\alpha\in \mathbb{R}}A(\alpha,t)\geq \frac{9}{10},\quad \quad \forall t\in [0,T_0].
\end{equation}
\begin{equation}
\sup_{\alpha\in \mathbb{R}}A(\alpha,t)\leq \frac{10}{9},\quad \quad \forall t\in [0,T_0].
\end{equation}
\end{corollary}

\subsubsection{Estimate the quantity $D_tb$.} Recall that 
\begin{equation}
\begin{split}
(I-\mathcal{H})D_t b=&[D_t\zeta,\mathcal{H}]\frac{\partial_{\alpha}b}{\zeta_{\alpha}}-[D_t^2\zeta,\mathcal{H}]\frac{\bar{\zeta}_{\alpha}-1}{\zeta_{\alpha}}-[D_t\zeta,\mathcal{H}]\frac{\partial_{\alpha}D_t\bar{\zeta}}{\zeta_{\alpha}}\\
&+\frac{1}{\pi i}\int \Big(\frac{D_t\zeta(\alpha,t)-D_t\zeta(\beta,t)}{\zeta(\alpha,t)-\zeta(\beta,t)}\Big)^2(\bar{\zeta}_{\beta}(\beta,t)-1)d\beta+\frac{i}{\pi}\sum_{j=1}^2\frac{\lambda_j(D_t\zeta-\dot{z}_j(t))}{(\zeta(\alpha,t)-z_j(t))^2}.
\end{split}
\end{equation}
By lemma \ref{lemmacommutator1}, lemma \ref{smallestimate}, lemma \ref{dotzjdotzj}, estimate (\ref{estimateforb}), we have 
\begin{align*}
&\|(I-\mathcal{H})D_tb\|_{H^s}\\
\leq &\norm{[D_t\zeta,\mathcal{H}]\frac{\partial_{\alpha}b}{\zeta_{\alpha}}}_{H^s}+\norm{[D_t^2\zeta,\mathcal{H}]\frac{\bar{\zeta}_{\alpha}-1}{\zeta_{\alpha}}}_{H^s}+\norm{[D_t\zeta,\mathcal{H}]\frac{\partial_{\alpha}D_t\bar{\zeta}}{\zeta_{\alpha}}}_{H^s}\\
&+\norm{\frac{1}{\pi i}\int \Big(\frac{D_t\zeta(\alpha,t)-D_t\zeta(\beta,t)}{\zeta(\alpha,t)-\zeta(\beta,t)}\Big)^2(\bar{\zeta}_{\beta}(\beta,t)-1)d\beta}_{H^s}+\norm{\frac{i}{\pi}\sum_{j=1}^2\frac{\lambda_j(D_t\zeta-\dot{z}_j(t))}{(\zeta(\alpha,t)-z_j(t))^2}}_{H^s}\\
\leq & C\|D_t\zeta\|_{H^s}\|b\|_{H^s}+C\|D_t^2\zeta\|_{H^s}\|\zeta_{\alpha}-1\|_{H^s}+C\|D_t\zeta\|_{H^s}^2\\
&+\|D_t\zeta\|_{H^s}\norm{\sum_{j=1}^2\frac{\lambda_j}{\pi(\zeta(\alpha,t)-z_j(t))^2}}_{H^s}+\norm{\sum_{j=1}^2\frac{\lambda_j \dot{z}_j}{\pi(\zeta(\alpha,t)-z_j(t))^2}}_{H^s}+C\|D_t\zeta\|_{H^s}^2\|\zeta_{\alpha}-1\|_{H^s}\\
\leq &C\epsilon (C\epsilon^2+K_s^{-1}\epsilon d_I(t)^{-5/2})+C\epsilon^2+K_s^{-1}\epsilon^2 d_I(t)^{-5/2}+K_s^{-1}\epsilon d_I(t)^{-5/2}+C\epsilon^3\\
\leq &C\epsilon^2+K_s^{-1}\epsilon d_I(t)^{-5/2}.
\end{align*}
By lemma \ref{realinverse}, we  have 
\begin{equation}
\|D_tb\|_{H^s}\leq C\epsilon^2+K_s^{-1}\epsilon d_I(t)^{-5/2}.
\end{equation}

\subsubsection{Estimate $\|G\|_{H^s}$.}\label{Gk1} Recall that $G=G_c+G_d$, with
\begin{equation}
G_c:=-2[\bar{\mathfrak{F}}, \mathcal{H}\frac{1}{\zeta_{\alpha}}+\bar{\mathcal{H}}\frac{1}{\bar{\zeta}_{\alpha}}]\bar{\mathfrak{F}}_{\alpha}+\frac{1}{\pi i}\int \Big(\frac{D_t\zeta(\alpha,t)-D_t\zeta(\beta,t)}{\zeta(\alpha,t)-\zeta(\beta,t)}\Big)^2(\zeta-\bar{\zeta})_{\beta}d\beta :=G_{c1}+G_{c2}.
\end{equation}
\begin{equation}
G_d:=-2[\bar{q}, \mathcal{H}]\frac{\partial_{\alpha}\bar{\mathfrak{F}}}{\zeta_{\alpha}}-2[\bar{\mathfrak{F}},\mathcal{H}]\frac{\partial_{\alpha}\bar{q}}{\zeta_{\alpha}}-2[\bar{q},\mathcal{H}]\frac{\partial_{\alpha}\bar{q}}{\zeta_{\alpha}}-4D_t q:=G_{d1}+G_{d2}+G_{d3}+G_{d4}.
\end{equation}
We rewrite $G_{c1}$ as 
\begin{equation}
G_{c1}=-\frac{4}{\pi}\int \frac{(D_t\bar{\mathfrak{F}}(\alpha,t)-D_t\bar{\mathfrak{F}}(\beta,t))Im\{\zeta(\alpha,t)-\zeta(\beta,t)\}}{|\zeta(\alpha,t)-\zeta(\beta,t)|^2}\partial_{\beta}\bar{\mathfrak{F}}(\beta,t)d\beta.
\end{equation}
By lemma \ref{lemmacommutator1},  we have 
\begin{equation}
\|G_{c1}\|_{H^s}\leq C\|\mathfrak{F}\|_{H^s}\|\zeta_{\alpha}-1\|_{H^s}\|\mathfrak{F}\|_{H^s}\leq C\epsilon^3,
\end{equation}
for some constant $C$ depends on $s$ only. Similarly, 
\begin{align*}
\|G_{c2}\|_{H^s}\leq C\|D_t\zeta\|_{H^s}^2\|\zeta_{\alpha}-1\|_{H^s}\leq C\epsilon^3.
\end{align*}
By lemma \ref{lemmacommutator1}, we have 
\begin{equation}\label{cs}
\|G_{d1}\|_{H^s}+\|G_{d2}\|_{H^s}\leq C\|q\|_{H^s}\|\mathfrak{F}\|_{H^s}\leq K_s^{-1}\epsilon^2 d_I(t)^{-3/2}.
\end{equation}
Similarly,
\begin{align*}
\|G_{d3}\|_{H^s}\leq C\|q\|_{H^s}^2\leq K_s^{-1}\epsilon^2 d_I(t)^{-3/2}.
\end{align*}
Use 
\begin{equation}
D_t\bar{q}=\sum_{j=1}^2\frac{\lambda_j i}{2\pi}\frac{D_t\zeta-\dot{z}_j}{(\zeta(\alpha,t)-z_j(t))^2},
\end{equation}
by lemma \ref{dotzjdotzj}, lemma \ref{smallestimate}, we have 
\begin{equation}\label{estimateGd4}
\begin{split}
\|G_{d4}\|_{H^s}\leq &4\|D_t\zeta\|_{H^s}\norm{\sum_{j=1}^2\frac{\lambda_j}{2\pi(\zeta(\alpha,t)-z_j(t))^2}}_{H^s}+4\norm{\sum_{j=1}^2\frac{\lambda_j \dot{z}_j}{2\pi(\zeta(\alpha,t)-z_j(t))^2}}_{H^s}\\
\leq & C\epsilon^2  d_I(t)^{-5/2}+K_s^{-1}\epsilon \frac{|\lambda|}{x(0)}d_I(t)^{-5/2}.
\end{split}
\end{equation}
So we obtain
\begin{equation}
\|G\|_{H^s}\leq C\epsilon^3+K_s^{-1}\epsilon^2 d_I(t)^{-3/2}+K_s^{-1}\epsilon \frac{|\lambda|}{x(0)} d_I(t)^{-5/2}.
\end{equation}
As a consequence, 
\begin{equation}\label{Gk11}
\|(I-\mathcal{H})G\|_{H^s}\leq 3\|G\|_{H^s}\leq C\epsilon^3+K_s^{-1}\epsilon^2 d_I(t)^{-3/2}+K_s^{-1}\epsilon \frac{|\lambda|}{x(0)}d_I(t)^{-5/2}.
\end{equation}

\subsubsection{Estimate  $\|(I-\mathcal{H})[D_t^2-iA\partial_{\alpha}, \partial_{\alpha}^k]\tilde{\theta})\|_{L^2}$.}\label{Gk2}
By lemma \ref{commutehigh}, we have 
\begin{align*}
[D_t^2,\partial_{\alpha}^k]\tilde{\theta}=&-\sum_{m=0}^{k-1}\Big[\partial_{\alpha}^m(D_tb_{\alpha})\partial_{\alpha}^{k-m}\tilde{\theta}+\partial_{\alpha}^m (b_{\alpha}\partial_{\alpha}^{k-m}D_t\tilde{\theta})+\partial_{\alpha}^m (b_{\alpha}[b\partial_{\alpha},\partial_{\alpha}^{k-m}]\tilde{\theta})+\partial_{\alpha}^m b_{\alpha}\partial_{\alpha}^{k-m}D_t\tilde{\theta}\\
&+\partial_{\alpha}^mb_{\alpha}[b\partial_{\alpha},\partial_{\alpha}]\partial_{\alpha}^{k-m-1}\tilde{\theta}\Big]\\
\end{align*}

\vspace*{1ex}

\noindent \underline{The quantity $\|\partial_{\alpha}^m(D_tb_{\alpha})\partial_{\alpha}^{k-m}\tilde{\theta}\|_{L^2}$.}
For $0\leq m\leq k-1, k\leq s$, we have 
\begin{equation}
\begin{split}
\|\partial_{\alpha}^m(D_tb_{\alpha})\partial_{\alpha}^{k-m}\tilde{\theta}\|_{L^2}\leq \|D_tb_{\alpha}\|_{H^m}\|\partial_{\alpha}^{k-m}\tilde{\theta}\|_{H^m}
\end{split}
\end{equation}
Since $D_tb_{\alpha}=\partial_{\alpha}D_tb+bb_{\alpha}$, we have 
\begin{align*}
\|D_tb_{\alpha}\|_{H^m}\leq &\|\partial_{\alpha}D_tb\|_{H^m}+\|b_{\alpha}b\|_{H^m}\leq \|D_tb\|_{H^s}+\|b\|_{H^s}^2\\
\leq & C\epsilon^2+K_s^{-1}\epsilon d_I(t)^{-5/2}+(C\epsilon^2+K_s^{-1}\epsilon d_I(t)^{-3/2})^2\\
\leq & C\epsilon^2+K_s^{-1}\epsilon d_I(t)^{-5/2}.
\end{align*}
and since $k-m\geq 1$, by Corollary \ref{onetotwo}, we have 
\begin{equation}
\|\partial_{\alpha}^{k-m}\tilde{\theta}\|_{H^s}\leq 11\epsilon,
\end{equation}
we obtain
\begin{equation}
\|\partial_{\alpha}^m(D_tb_{\alpha})\partial_{\alpha}^{k-m}\tilde{\theta}\|_{L^2}\leq \|D_tb_{\alpha}\|_{H^s}\|\partial_{\alpha}\tilde{\theta}\|_{H^s}\leq C\epsilon^3+K_s^{-1}\epsilon^2 d_I(t)^{-3/2}.
\end{equation}

\vspace*{1ex}

\noindent \underline{The quantity $\partial_{\alpha}^m (b_{\alpha}\partial_{\alpha}^{k-m}D_t\tilde{\theta})$.} Similar to the previous case, we have for $0\leq m\leq k-1, k\leq s$, and assume bootstrap assumption (\ref{longtimeapriori}),
\begin{equation}
\|\partial_{\alpha}^m (b_{\alpha}\partial_{\alpha}^{k-m}D_t\tilde{\theta})\|_{L^2}\leq \|b_{\alpha}\|_{H^{k-1}}\|D_t\tilde{\theta}\|_{H^k}\leq C\epsilon^3+K_s^{-1}\epsilon^2 d_I(t)^{-3/2}.
\end{equation}

\vspace*{1ex}

\noindent \underline{The quantity $\partial_{\alpha}^m (b_{\alpha}[b\partial_{\alpha},\partial_{\alpha}^{k-m}]\tilde{\theta})$.} 
We have  for $0\leq m\leq k-1, k\leq s$, and assume bootstrap assumption (\ref{longtimeapriori}),
\begin{equation}
\|\partial_{\alpha}^m (b_{\alpha}[b\partial_{\alpha},\partial_{\alpha}^{k-m}]\tilde{\theta})\|_{L^2}\leq C\|b_{\alpha}\|_{H^{s-1}}\|b\|_{H^s}\|\theta_{\alpha}\|_{H^{s-1}}\leq C\epsilon^3.
\end{equation}

\vspace*{2ex}

\noindent \underline{The quantity $\partial_{\alpha}^m b_{\alpha}\partial_{\alpha}^{k-m}D_t\tilde{\theta}$.} 
We have  for $0\leq m\leq k-1, k\leq s$, 
\begin{equation}
\|\partial_{\alpha}^m b_{\alpha}\partial_{\alpha}^{k-m}D_t\tilde{\theta}\|_{L^2}\leq C\|b_{\alpha}\|_{H^{k-1}}\|D_t\tilde{\theta}\|_{H^{k}}\leq C\epsilon^3+K_s^{-1}\epsilon^2  d_I(t)^{-3/2}.
\end{equation}

\vspace*{2ex}

\noindent \underline{The quantity $\partial_{\alpha}^mb_{\alpha}[b\partial_{\alpha},\partial_{\alpha}]\partial_{\alpha}^{k-m-1}\tilde{\theta}$.} 
We have  for $0\leq m\leq k-1, k\leq s$,
\begin{equation}
\|\partial_{\alpha}^mb_{\alpha}[b\partial_{\alpha},\partial_{\alpha}]\partial_{\alpha}^{k-m-1}\tilde{\theta}\|_{L^2}\leq C\|b\|_{H^{s}}^2\|D_t\theta\|_{H^{s}}\leq C\epsilon^3.
\end{equation}
So we obtain
\begin{equation}
\|[D_t^2, \partial_{\alpha}^k]\tilde{\theta}\|_{L^2}\leq C\epsilon^3+K_s^{-1}C\epsilon^2 d_I(t)^{-3/2}.
\end{equation}

\vspace*{2ex}

\noindent \underline{The quantity $\|[iA\partial_{\alpha},\partial_{\alpha}^k]\tilde{\theta}\|_{L^2}$} Use similar argument, we obtain
\begin{equation}
\|[iA\partial_{\alpha},\partial_{\alpha}^k]\theta\|_{L^2}\leq C\epsilon^3+K_s^{-1}\epsilon^2 d_I(t)^{-3/2}.
\end{equation}

\vspace*{2ex}
So we obtain
\begin{equation}\label{Gk22}
\|(I-\mathcal{H})[D_t^2-iA\partial_{\alpha}, \partial_{\alpha}^k]\tilde{\theta})\|_{L^2}\leq C\epsilon^3+K_s^{-1}C\epsilon^2  d_I(t)^{-3/2}.
\end{equation}

\subsubsection{Estimate $\|[D_t^2-iA\partial_{\alpha},\mathcal{H}]\partial_{\alpha}^k\tilde{\theta}\|_{L^2}$}\label{Gk3}
Note that by identity (\ref{abcde}),
\begin{equation}
[D_t^2-iA\partial_{\alpha},\mathcal{H}]\partial_{\alpha}^k\tilde{\theta}=2[D_t\zeta,\mathcal{H}]\frac{\partial_{\alpha}\partial_{\alpha}^k\tilde{\theta}}{\zeta_{\alpha}}-\frac{1}{\pi i}\int \Big(\frac{\zeta(\alpha,t)-\zeta(\beta,t)}{\zeta(\alpha,t)-\zeta(\beta,t)}\Big)^2 \partial_{\beta} \partial_{\beta}^k \tilde{\theta} d\beta
\end{equation}
Clearly, for $k\leq s$, and assume (\ref{longtimeapriori}), we have 
\begin{equation}
\norm{\frac{1}{\pi i}\int \Big(\frac{\zeta(\alpha,t)-\zeta(\beta,t)}{\zeta(\alpha,t)-\zeta(\beta,t)}\Big)^2 \partial_{\beta} \partial_{\beta}^k \tilde{\theta} d\beta}_{L^2}\leq C\epsilon^3.
\end{equation}

\vspace*{2ex}

\noindent $\bullet$ Estimate $\|[D_t\zeta,\mathcal{H}]\frac{\partial_{\alpha}\partial_{\alpha}^k\tilde{\theta}}{\zeta_{\alpha}}\|_{L^2}$.

$[D_t\zeta,\mathcal{H}]\frac{\partial_{\alpha}\partial_{\alpha}^k\tilde{\theta}}{\zeta_{\alpha}}$ is not obvious cubic. However, since $\partial_{\alpha}^k\tilde{\theta}$ is almost anti-holomorphic, and $D_t\zeta=\bar{\mathfrak{F}}+\bar{q}$, with $\bar{\mathfrak{F}}$ anti-holomorphic and $\bar{q}$ decays rapidly in time as long as the point vortices move away from the free interface rapidly,  we expect this quantity consists of cubic terms and quadratic terms which decay rapidly. To see this,  decompose 
$$\partial_{\alpha}^k\tilde{\theta}:=\frac{1}{2}(I-\mathcal{H})\partial_{\alpha}^k\tilde{\theta}+\frac{1}{2}(I+\mathcal{H})\partial_{\alpha}^k\tilde{\theta}.$$
Note that for $k\geq 1$,
\begin{align*}
(I+\mathcal{H})\partial_{\alpha}^k\tilde{\theta}=&(I+\mathcal{H})\partial_{\alpha}^k(I-\mathcal{H})(\zeta-\bar{\zeta})=-[\partial_{\alpha}^k,\mathcal{H}]\tilde{\theta}\\
=&-\sum_{m=0}^{k-1}\partial_{\alpha}^m [\zeta_{\alpha}-1,\mathcal{H}]\frac{\partial_{\alpha}\partial_{\alpha}^{k-m-1}\tilde{\theta}}{\zeta_{\alpha}}.
\end{align*}
By lemma \ref{lemmacommutator1} and lemma \ref{equivalent_1},
\begin{equation}
\|(I+\mathcal{H})\partial_{\alpha}^k\tilde{\theta}\|_{L^2}\leq C\|\zeta_{\alpha}-1\|_{H^k} \|\partial_{\alpha}\tilde{\theta}\|_{H^{s-1}}\leq C\epsilon^2.
\end{equation}
Therefore, by lemma \ref{lemmacommutator1}, we have 
\begin{equation}
\norm{[D_t\zeta,\mathcal{H}]\frac{\partial_{\alpha}\frac{1}{2}(I+\mathcal{H})\partial_{\alpha}^k\tilde{\theta}}{\zeta_{\alpha}}}_{L^2}\leq C\epsilon^3.
\end{equation}
We rewrite $[D_t\zeta,\mathcal{H}]\frac{\partial_{\alpha}\frac{1}{2}(I-\mathcal{H})\partial_{\alpha}^k\tilde{\theta}}{\zeta_{\alpha}}$ as 
\begin{align*}
&[D_t\zeta,\mathcal{H}]\frac{\partial_{\alpha}\frac{1}{2}(I-\mathcal{H})\partial_{\alpha}^k\tilde{\theta}}{\zeta_{\alpha}}\\
=&[\frac{1}{2}(I+\mathcal{H})D_t\zeta,\mathcal{H}]\frac{\partial_{\alpha}\frac{1}{2}(I-\mathcal{H})\partial_{\alpha}^k\tilde{\theta}}{\zeta_{\alpha}}+[\frac{1}{2}(I-\mathcal{H})D_t\theta,\mathcal{H}]\frac{\partial_{\alpha}\frac{1}{2}(I-\mathcal{H})\partial_{\alpha}^k\tilde{\theta}}{\zeta_{\alpha}}:=I+\it{II}.
\end{align*}
Clearly, $\it{II}=0$. Since 
\begin{equation}
    \frac{1}{2}(I+\mathcal{H})D_t\zeta=\frac{1}{2}(I+\mathcal{H})\bar{q}+\frac{1}{2}(\mathcal{H}+\bar{\mathcal{H}})\bar{\mathfrak{F}},
\end{equation}

Use lemma \ref{lemmacommutator1}, lemma \ref{smallestimate}, and similar to the estimate of $\|G_{d1}\|_{H^s}$ in \S \ref{Gk1},   we have 
\begin{equation}
\begin{split}
\norm{[\frac{1}{2}(I+\mathcal{H})\bar{q},\mathcal{H}]\frac{\partial_{\alpha}\frac{1}{2}(I-\mathcal{H})\partial_{\alpha}^k\tilde{\theta}}{\zeta_{\alpha}}}_{L^2}\leq & C\|q\|_{H^{k}}\|\partial_{\alpha}\theta\|_{H^{k-1}}\leq K_s^{-1}\epsilon^2 d_I(t)^{-3/2}.
\end{split}
\end{equation}
It's easy to obtain
\begin{equation}
    \norm{(\mathcal{H}+\bar{\mathcal{H}})\mathfrak{F}}_{H^s}\leq C\epsilon^2.
\end{equation}
So we obtain
\begin{equation}
    \norm{[(\mathcal{H}+\bar{\mathcal{H}})\mathfrak{F},\mathcal{H}]\frac{\partial_{\alpha}\frac{1}{2}(I-\mathcal{H})\partial_{\alpha}^k\tilde{\theta}}{\zeta_{\alpha}}}_{H^s}\leq C\epsilon^3.
\end{equation}
Therefore,
\begin{equation}\label{Gk33}
\norm{(I-\mathcal{H})[D_t^2-iA\partial_{\alpha},\mathcal{H}]\partial_{\alpha}^k\tilde{\theta}}_{L^2}\leq 3\norm{[D_t^2-iA\partial_{\alpha},\mathcal{H}]\partial_{\alpha}^k\tilde{\theta}}_{L^2}\leq C\epsilon^3+K_s^{-1}\epsilon^2d_I(t)^{-3/2}.
\end{equation}

\subsubsection{Estimate for $\norm{G_k^{\theta}}_{L^s}$.} Collect the estimates from  (\ref{Gk11}),  (\ref{Gk22}),  (\ref{Gk33}), we obtain
\begin{equation}\label{Gktheta}
\norm{G_k^{\theta}}_{L^2}\leq C\epsilon^3+K_s^{-1}\epsilon^2d_I(t)^{-3/2}+K_s^{-1}\epsilon \frac{|\lambda|}{x(0)}d_I(t)^{-5/2}.
\end{equation}

\subsection{Estimate $\norm{(I-\mathcal{H})\partial_{\alpha}^k \tilde{G}}_{L^2}$ }\label{estimateinvolve1} Recall that 
\begin{equation}\label{systemnew}
\begin{split}
\tilde{G}=& (I-\mathcal{H})(D_t G+i\frac{a_t}{a}\circ\kappa^{-1}A((I-\mathcal{H})(\zeta-\bar{\zeta}))_{\alpha})-2[D_t\zeta,\mathcal{H}]\frac{\partial_{\alpha}D_t^2(I-\mathcal{H})(\zeta-\bar{\zeta})}{\zeta_{\alpha}}\\
&+\frac{1}{\pi i}\int \Big(\frac{D_t\zeta(\alpha,t)-D_t\zeta(\beta,t)}{\zeta(\alpha,t)-\zeta(\beta,t)}\Big)^2 (D_t(I-\mathcal{H})(\zeta-\bar{\zeta}))_{\beta}d\beta\\
\end{split}
\end{equation}

\subsubsection{Estimate $\norm{D_tG}_{H^k}$} $D_tG$ is given by
\begin{align*}
D_tG=(\partial_t g)\circ \kappa^{-1}.
\end{align*}
$g=g_c+g_d$, and 
\begin{align*}
\partial_t g_c=&\partial_t \Big\{-2[\bar{f}, \mathfrak{H}\frac{1}{z_{\alpha}}+\bar{\mathfrak{H}}\frac{1}{\bar{z}_{\alpha}}]\bar{f}_{\alpha}+\frac{1}{\pi i}\int \Big(\frac{z_t(\alpha,t)-z_t(\beta,t)}{z(\alpha,t)-z(\beta,t)}\Big)^2(z-\bar{z})_{\beta}d\beta\Big\}\\
=& -2[\bar{f}_t, \mathfrak{H}\frac{1}{z_{\alpha}}+\bar{\mathfrak{H}}\frac{1}{\bar{z}_{\alpha}}]\bar{f}_{\alpha} -2[\bar{f}, \mathfrak{H}\frac{1}{z_{\alpha}}+\bar{\mathfrak{H}}\frac{1}{\bar{z}_{\alpha}}]\bar{f}_{t\alpha}\\
&-\frac{4}{\pi}\int (\bar{f}(\alpha,t)-\bar{f}(\beta,t)) \Big(\partial_t Im\Big\{ \frac{z_t(\alpha,t)-z_t(\beta,t)}{(z(\alpha,t)-z(\beta,t))^2} \Big\}\Big)\partial_{\beta}\bar{f}(\beta,t)d\beta\\
&+\frac{2}{\pi i}\int \frac{z_t(\alpha,t)-z_t(\beta,t)}{z(\alpha,t)-z(\beta,t)}\Big\{\frac{z_{tt}(\alpha,t)-z_{tt}(\beta,t)}{z(\alpha,t)-z(\beta,t)}-\frac{(z_t(\alpha,t)-z_t(\beta,t))^2}{(z(\alpha,t)-z(\beta,t))^2} \Big\} (z-\bar{z})_{\beta}d\beta \\
& +\frac{1}{\pi i}\int \Big(\frac{z_t(\alpha,t)-z_t(\beta,t)}{z(\alpha,t)-z(\beta,t)}\Big)^2(z_t-\bar{z}_t)_{\beta}d\beta
\end{align*}
So we have 
\begin{align*}
D_tG_c=&
-2[D_t\bar{\mathfrak{F}}, \mathcal{H}\frac{1}{\zeta_{\alpha}}+\bar{\mathcal{H}}\frac{1}{\bar{\zeta}_{\alpha}}]\bar{\mathfrak{F}}_{\alpha} -2[\bar{\mathfrak{F}}, \mathcal{H}\frac{1}{\zeta_{\alpha}}+\bar{\mathcal{H}}\frac{1}{\bar{\zeta}_{\alpha}}]\partial_{\alpha}D_t\bar{\mathfrak{F}}\\
&-\frac{4}{\pi}\int (\bar{\mathfrak{F}}(\alpha,t)-\bar{\mathfrak{F}}(\beta,t)) \Big(D_t Im\Big\{ \frac {\zeta(\alpha,t)-\zeta(\beta,t)}{(\zeta(\alpha,t)-\zeta(\beta,t))^2} \Big\}\Big)\partial_{\beta}\bar{\mathfrak{F}}(\beta,t)d\beta\\
&+\frac{2}{\pi i}\int \frac{D_t\zeta(\alpha,t)-D_t\zeta(\beta,t)}{\zeta(\alpha,t)-\zeta(\beta,t)}\Big\{\frac{D_t^2\zeta(\alpha,t)-D_t^2\zeta(\beta,t)}{\zeta(\alpha,t)-\zeta(\beta,t)}-\frac{(D_t\zeta(\alpha,t)-D_t\zeta(\beta,t))^2}{(\zeta(\alpha,t)-\zeta(\beta,t))^2} \Big\} (\zeta-\bar{\zeta})d\beta \\
& +\frac{1}{\pi i}\int \Big(\frac{D_t\zeta(\alpha,t)-D_t\zeta(\beta,t)}{\zeta(\alpha,t)-\zeta(\beta,t)}\Big)^2(D_t\zeta-D_t\bar{\zeta})_{\beta}d\beta
\end{align*}
Recall that 
\begin{equation}
g_d:=-2[\bar{p}, \mathfrak{H}]\frac{\partial_{\alpha}\bar{f}}{z_{\alpha}}-2[\bar{f},\mathfrak{H}]\frac{\partial_{\alpha}\bar{p}}{z_{\alpha}}-2[\bar{p},\mathfrak{H}]\frac{\partial_{\alpha}\bar{p}}{z_{\alpha}}-4p_t.
\end{equation}
So 
\begin{align*}
\partial_t g_d=&-2[\bar{p}, \mathfrak{H}]\frac{\partial_{\alpha t}\bar{f}}{z_{\alpha}}-\frac{2}{\pi i}\int \Big(\frac{\bar{p}(\alpha,t)-\bar{p}(\beta,t)}{z(\alpha,t)-z(\beta,t))} \Big)_t \partial_{\beta}\bar{f}(\beta,t)d\beta\\
&-2[\bar{f},\mathfrak{H}]\frac{\partial_{\alpha}\bar{p}_t}{z_{\alpha}}-\frac{2}{\pi i}\int \Big(\frac{\bar{f}(\alpha,t)-\bar{f}(\beta,t)}{z(\alpha,t)-z(\beta,t)}  \Big)_t\partial_{\beta}\bar{p}(\beta,t)d\beta\\
&-2[\bar{p},\mathfrak{H}]\frac{\partial_{\alpha}\bar{p}_t}{z_{\alpha}}-\frac{2}{\pi i}\int \Big(\frac{\bar{p}(\alpha,t)-\bar{p}(\beta,t)}{z(\alpha,t)-z(\beta,t)} \Big)_t \partial_{\beta}\bar{p}(\beta,t)d\beta\\
&-4p_{tt}.
\end{align*}
So we have 
\begin{equation}
\begin{split}
D_t G_d=& -2[\bar{q}, \mathcal{H}]\frac{\partial_{\alpha }D_t\bar{\mathfrak{F}}}{\zeta_{\alpha}}-\frac{2}{\pi i}\int \Big(D_t\frac{\bar{q}(\alpha,t)-\bar{q}(\beta,t)}{\zeta(\alpha,t)-\zeta(\beta,t)}  \Big)\partial_{\beta}\bar{\mathfrak{F}}(\beta,t)d\beta\\
&-2[\bar{\mathfrak{F}},\mathcal{H}]\frac{\partial_{\alpha}D_t\bar{q}}{\zeta_{\alpha}}-\frac{2}{\pi i}\int \Big(D_t\frac{\bar{\mathfrak{F}}(\alpha,t)-\bar{\mathfrak{F}}(\beta,t)}{\zeta(\alpha,t)-\zeta(\beta,t)}  \Big)\partial_{\beta}\bar{q}(\beta,t)d\beta\\
&-2[\bar{q},\mathcal{H}]\frac{\partial_{\alpha}D_t\bar{q}}{\zeta_{\alpha}}-\frac{2}{\pi i}\int \Big(D_t\frac{\bar{q}(\alpha,t)-\bar{q}(\beta,t)}{\zeta(\alpha,t)-\zeta(\beta,t)}\Big)  \partial_{\beta}\bar{q}(\beta,t)d\beta\\
&-4D_t^2q.
\end{split}
\end{equation}

\noindent $D_tG_c$ is cubic, we have
\begin{align*}
\norm{D_tG_c}_{H^k}\leq & C_s(\norm{D_t\mathfrak{F}}_{H^k}\norm{\zeta_{\alpha}-1}_{H^k}\norm{\mathfrak{F}}_{H^k}+\norm{\mathfrak{F}}_{H^k}^2\norm{D_t\zeta}_{H^k}+\norm{D_t\zeta}_{H^k}\norm{D_t^2\zeta}_{H^k}\norm{\zeta_{\alpha}-1}_{H^k}\\
&+\norm{D_t\zeta}_{H^k}^3\norm{\zeta_{\alpha}-1}_{H^k}+\|D_t\zeta\|_{H^k}^3)\\
\leq & C\epsilon^3.
\end{align*}
$D_tG_d$ consists of cubic terms or terms with rapid time decay.  By (\ref{estimateGd4}), we have 
\begin{equation}
\norm{D_tq}_{H^k}\leq C\epsilon^2 d_I(t)^{-5/2}+K_s^{-1}\epsilon d_I(t)^{-5/2}.
\end{equation}
Note that $D_t G_d+4D_t^2q$ is at least quadratic. Use lemma \ref{lemmacommutator1}, lemma \ref{smallestimate}, and similar to the estimate of $\norm{G_{d1}}_{H^s}$ in \S \ref{Gk1}, we obtain
\begin{equation}
\norm{D_tG_d+4D_t^2q}_{H^k}\leq C\epsilon^3+K_s^{-1}\epsilon^2d_I(t)^{-3/2}.
\end{equation}
Note that
\begin{equation}
\begin{split}
D_t^2q=&\sum_{j=1}^2\frac{\lambda_j i}{2\pi}\frac{D_t^2\zeta-\ddot{z}_j(t)}{(\zeta(\alpha,t)-z_j(t))^2}-\sum_{j=1}^2\frac{\lambda_j i}{\pi}\frac{(D_t\zeta)^2-2D_t\zeta \dot{z}_j}{(\zeta(\alpha,t)-z_j(t))^3}-\sum_{j=1}^2\frac{\lambda_j i}{\pi}\frac{\dot{z}_j^2}{(\zeta(\alpha,t)-z_j(t))^3}.
\end{split}
\end{equation}
Use lemma \ref{dotzjdotzj}, lemma \ref{dotzjsquare}, lemma \ref{ddotzjnorm}, we have 
\begin{equation}
\norm{4D_t^2q}_{H^k}\leq C\epsilon^3+K_s^{-1}\epsilon^2d_I(t)^{-3/2}+K_s^{-1}\epsilon \frac{|\lambda|}{x(0)}d_I(t)^{-3/2}.
\end{equation}
Then we have 
\begin{equation}
\norm{D_tG}_{H^k}\leq C\epsilon^3+K_s^{-1}\epsilon^2d_I(t)^{-3/2}+K_s^{-1}\epsilon \frac{|\lambda|}{x(0)}d_I(t)^{-3/2}.
\end{equation}
Therefore,
\begin{equation}\label{Gk111}
\norm{(I-\mathcal{H})D_tG}_{H^k}\leq C\epsilon^3+K_s^{-1}\epsilon^2d_I(t)^{-3/2}+K_s^{-1}\epsilon \frac{|\lambda|}{x(0)}d_I(t)^{-3/2}.
\end{equation}

\subsubsection{Estimate $\norm{2[D_t\zeta,\mathcal{H}]\frac{\partial_{\alpha}D_t^2(I-\mathcal{H})(\zeta-\bar{\zeta})}{\zeta_{\alpha}}}_{H^k}$} The way that we estimate for this quantity is the same as that for $[D_t\zeta,\mathcal{H}]\frac{\partial_{\alpha}\partial_{\alpha}^k\tilde{\theta}}{\zeta_{\alpha}}$. We obtain
\begin{equation}
\norm{2[D_t\zeta,\mathcal{H}]\frac{\partial_{\alpha}D_t^2(I-\mathcal{H})(\zeta-\bar{\zeta})}{\zeta_{\alpha}}}_{H^k}\leq C\epsilon^3+K_s^{-1}\epsilon^2 d_I(t)^{-3/2}.
\end{equation}

\vspace*{2ex}

\noindent So we obtain
\begin{equation}\label{Gk222}
\norm{(I-\mathcal{H})\partial_{\alpha}^k\tilde{G}}_{L^2}\leq C\epsilon^3+K_s^{-1}\epsilon^2 d_I(t)^{-3/2}+K_s^{-1}\epsilon\frac{|\lambda|}{x(0)}d_I(t)^{-3/2}.
\end{equation}

\subsubsection{Estimate $[D_t^2-iA\partial_{\alpha},\mathcal{H}]\partial_{\alpha}^k\tilde{\sigma}$} Use 
\begin{align*}
[D_t^2-iA\partial_{\alpha},\mathcal{H}]\partial_{\alpha}^k\tilde{\sigma}=&2[D_t\zeta, \mathcal{H}]\frac{\partial_{\alpha}D_t\partial_{\alpha}^k\tilde{\sigma}}{\zeta_{\alpha}}-\frac{1}{\pi i}\int \Big(\frac{D_t\zeta(\alpha,t)-D_t\zeta(\beta,t)}{\zeta(\alpha,t)-\zeta(\beta,t)}\Big)^2 \partial_{\beta}^{k+1}\tilde{\sigma}(\beta,t)d\beta\\
:=& I_1+I_2.
\end{align*}
Clearly, 
\begin{equation}
\|I_2\|_{L^2}\leq C\|D_t\zeta\|_{H^k}^2\|\tilde{\sigma}\|_{H^k}\leq C\epsilon^3.
\end{equation}
Note that 
\begin{align*}
[D_t\zeta, \mathcal{H}]\frac{\partial_{\alpha}D_t\partial_{\alpha}^k\tilde{\sigma}}{\zeta_{\alpha}}=[D_t\zeta, \mathcal{H}]\frac{\partial_{\alpha}\partial_{\alpha}^kD_t\tilde{\sigma}}{\zeta_{\alpha}}+[D_t\zeta, \mathcal{H}]\frac{\partial_{\alpha}[D_t,\partial_{\alpha}^k]\tilde{\sigma}}{\zeta_{\alpha}}
\end{align*}
The second term $[D_t\zeta, \mathcal{H}]\frac{\partial_{\alpha}[D_t,\partial_{\alpha}^k]\tilde{\sigma}}{\zeta_{\alpha}}$ is cubic, it's easy to obtain
\begin{equation}
\norm{[D_t\zeta, \mathcal{H}]\frac{\partial_{\alpha}[D_t,\partial_{\alpha}^k]\tilde{\sigma}}{\zeta_{\alpha}}}_{L^2}\leq C\epsilon^3.
\end{equation}
The way that we estimate for this quantity is the same as that for $[D_t\zeta,\mathcal{H}]\frac{\partial_{\alpha}\partial_{\alpha}^k\tilde{\theta}}{\zeta_{\alpha}}$. We obtain
\begin{equation}
\norm{[D_t\zeta, \mathcal{H}]\frac{\partial_{\alpha}\partial_{\alpha}^kD_t\tilde{\sigma}}{\zeta_{\alpha}}}_{H^k}\leq C\epsilon^3+K_s^{-1}\epsilon^2d_I(t)^{-3/2}.
\end{equation}
So we obtain
\begin{equation}\label{Gk333}
\norm{[D_t^2-iA\partial_{\alpha},\mathcal{H}]\partial_{\alpha}^k\tilde{\sigma}}_{L^2}\leq C\epsilon^3+K_s^{-1}\epsilon^2d_I(t)^{-3/2}.
\end{equation}

\subsubsection{Estimate $\norm{(I-\mathcal{H})[D_t^2-iA\partial_{\alpha}, \partial_{\alpha}^k]\tilde{\sigma}}_{L^2}$}
The way that we estimate this quantity is the same as that for $\norm{(I-\mathcal{H})[D_t^2-iA\partial_{\alpha}, \partial_{\alpha}^k]\tilde{\theta})}_{L^2}$. We obtain
\begin{equation}\label{Gk444}
\norm{(I-\mathcal{H})[D_t^2-iA\partial_{\alpha}, \partial_{\alpha}^k]\tilde{\sigma}}_{L^2}\leq C\epsilon^3+K_s^{-1}\epsilon^2d_I(t)^{-3/2}.
\end{equation}

\subsubsection{Estimate for $\norm{G_k^{\sigma}}_{L^2}$.} Collect the estimates from  (\ref{Gk111}),  (\ref{Gk222}),  (\ref{Gk333}), (\ref{Gk444}), we obtain
\begin{equation}\label{Gksigma}
\norm{G_k^{\sigma}}_{L^2}\leq C\epsilon^3+K_s^{-1}\epsilon^2d_I(t)^{-3/2}+K_s^{-1}\epsilon\frac{|\lambda|}{x(0)}d_I(t)^{-3/2}.
\end{equation}

\subsection{A priori energy estimates} We derive energy estimates in this subsection. We'll prove the following.
\begin{proposition}\label{energy_estimate}
Assume the assumptions of Theorem \ref{longtime}, assume the bootstrap assumption (\ref{longtimeapriori}), we have for all $t\in [0,T_0]$,
\begin{equation}
\frac{d}{dt}\mathcal{E}_s(t)\leq C\epsilon^4+K_s^{-1}\epsilon^3 d_I(t)^{-3/2}+K_s^{-1}\epsilon^2\frac{|\lambda|}{x(0)}d_I(t)^{-5/2}.
\end{equation}
\end{proposition}
\begin{proof}
From (\ref{energy1}) and (\ref{energy2}), we have 
\begin{equation}
\begin{split}
&\frac{d}{dt}\mathcal{E}_s(t)=\frac{d}{dt}\sum_{k=0}^s (E_k^{\theta}+E_k^{\sigma})\\
=&\sum_{k=0}^s \Big(\int\frac{2}{A}Re D_t\theta_k \overline{{G}_k^{\theta}}-\int \frac{1}{A}\frac{a_t}{a}\circ\kappa^{-1}|D_t\theta_k|^2+\int\frac{2}{A}Re D_t\sigma_k \overline{G_k^{\sigma}}-\int \frac{1}{A}\frac{a_t}{a}\circ\kappa^{-1}|D_t\sigma_k|^2 \Big)
\end{split}
\end{equation}
By Corollary \ref{aprioriequiv}, we have 
\begin{equation}
\|D_t\tilde{\theta}\|_{H^s}\leq 11\epsilon,\quad \quad \|D_t\tilde{\sigma}\|_{H^s}\leq 21\epsilon.
\end{equation}

By Corollary \ref{lowerboundA}, (\ref{ataA}), (\ref{Gktheta}), (\ref{Gksigma}), we have 
\begin{align*}
\frac{d}{dt}\mathcal{E}_s(t)\leq &\sum_{k=0}^s \Big(2\norm{\frac{1}{A}}_{\infty}\norm{D_t\theta_k}_{L^2}\norm{G_k^{\theta}}_{L^2}+\norm{\frac{1}{A}}_{\infty}\norm{\frac{a_t}{a}\circ \kappa^{-1}}_{\infty}\norm{D_t\theta_k}_{L^2}^2\\
&+ 2\norm{\frac{1}{A}}_{\infty}\norm{D_t\sigma_k}_{L^2}\norm{G_k^{\sigma}}_{L^2}+\norm{\frac{1}{A}}_{\infty}\norm{\frac{a_t}{a}\circ \kappa^{-1}}_{\infty}\norm{D_t\sigma_k}_{L^2}^2\Big)\\
\leq & \sum_{k=0}^s \Big(4\times 11\epsilon (C\epsilon^3+K_s^{-1}\epsilon^2d_I(t)^{-3/2}+K_s^{-1}\epsilon \frac{|\lambda|}{x(0)}d_I(t)^{-5/2})\\
&+4\times (C\epsilon^2+K_s^{-1}\epsilon\frac{|\lambda|}{x(0)}d_I(t)^{-5/2})\times (11\epsilon)^2\\
&+4\times 21\epsilon (C\epsilon^3+K_s^{-1}\epsilon^2d_I(t)^{-3/2}+K_s^{-1}\epsilon\frac{|\lambda|}{x(0)}d_I(t)^{-3/2})\\
&+4\times  (C\epsilon^2+K_s^{-1}\epsilon\frac{|\lambda|}{x(0)}d_I(t)^{-5/2})\times (21\epsilon)^2\Big)\\
\leq & C\epsilon^4+K_s^{-1}\epsilon^3 d_I(t)^{-3/2}+K_s^{-1} \epsilon^2\frac{|\lambda|}{x(0)}d_I(t)^{-5/2}.
\end{align*}
Here, we simply bound $\norm{\frac{1}{A}}_{\infty}$ by $\frac{1}{2}$.
\end{proof}

Before we use the bootstrap argument to complete the proof of Theorem \ref{longtime}, we need to show that the energy $\mathcal{E}_s$ is equivalent to $4(\norm{D_t\tilde{\theta}}_{H^s}^2+\norm{D_t\tilde{\sigma}}_{H^s}^2+\norm{|D|^{1/2}\tilde{\theta}}_{H^s}^2+\norm{|D|^{1/2}\tilde{\sigma}}_{H^s}^2)$.
\begin{lemma}\label{connect}
Assume the assumptions of Theorem \ref{longtime}, and assume the bootstrap assumption (\ref{longtimeapriori}). Then we have 
\begin{equation}
 \Big|\mathcal{E}_s-4(\norm{D_t\tilde{\theta}}_{H^s}^2+\norm{D_t\tilde{\sigma}}_{H^s}^2+\norm{|D|^{1/2}\tilde{\theta}}_{H^s}^2+\norm{|D|^{1/2}\tilde{\sigma}}_{H^s}^2)\Big|\leq C\epsilon^3.
\end{equation}
\end{lemma}
\begin{proof}
Recall that 
\begin{equation}
    \mathcal{E}_s=\sum_{k=0}^s \Big\{\int \frac{1}{A}|D_t\theta_k|^2+i\theta_k \overline{\partial_{\alpha}\theta_k}d\alpha+\int \frac{1}{A}|D_t\sigma_k|^2+i\sigma_k \overline{\partial_{\alpha}\sigma_k}d\alpha\Big\},
\end{equation}
where 
\begin{equation}
\theta_k=(I-\mathcal{H})\partial_{\alpha}^k\tilde{\theta},\quad  \sigma_k=(I-\mathcal{H})\partial_{\alpha}^k\tilde{\sigma},\quad \tilde{\theta}:=(I-\mathcal{H})(\zeta-\bar{\zeta}), \quad \tilde{\sigma}:=(I-\mathcal{H})D_t\tilde{\theta}.
\end{equation}
It's easy to obtain that 
\begin{equation}
    \|A-1\|_{H^s}\leq C\epsilon.
\end{equation}
So 
\begin{equation}
    \mathcal{E}_s=\sum_{k=0}^s \Big\{\int |D_t\theta_k|^2+i\theta_k \overline{\partial_{\alpha}\theta_k}d\alpha+\int |D_t\sigma_k|^2+i\sigma_k \overline{\partial_{\alpha}\sigma_k}d\alpha\Big\}+O(\epsilon^3).
\end{equation}
We have 
\begin{align}
    \theta_k=\partial_{\alpha}^k(I-\mathcal{H})\tilde{\theta}+[\partial_{\alpha}^k, \mathcal{H}]\tilde{\theta}=2\partial_{\alpha}^k\tilde{\theta}+[\partial_{\alpha}^k, \mathcal{H}]\tilde{\theta}.
\end{align}
So we have 
\begin{equation}\label{almostone}
    \norm{D_t\theta_k-2\partial_{\alpha}^kD_t\tilde{\theta}}_{L^2}\leq \norm{D_t[\partial_{\alpha}^k, \mathcal{H}]\tilde{\theta}}_{L^2}+2\norm{[D_t,\partial_{\alpha}^k]\tilde{\theta}}_{L^2}\leq C\epsilon^2.
\end{equation}
Similarly, we have 
\begin{equation}\label{almosttwo}
    \norm{D_t\sigma_k-2\partial_{\alpha}^kD_t\tilde{\sigma}}_{L^2}\leq C\epsilon^2.
\end{equation}
Therefore,
\begin{equation}
    \Big|\sum_{k=0}^s \int (|D_t\theta_k|^2+|D_t\sigma_k|^2)d\alpha-4(\norm{\partial_{\alpha}^kD_t\tilde{\theta}}_{L^2}^2+\norm{\partial_{\alpha}^k D_t\tilde{\sigma}}_{L^2}^2)\Big|\leq C\epsilon^3.
\end{equation}
Decompose $\tilde{\theta}$ as 
\begin{equation}
    \tilde{\theta}=\frac{1}{2}(I+\mathbb{H})\tilde{\theta}+\frac{1}{2}(I-\mathbb{H})\tilde{\theta}
\end{equation}
Note that since $\tilde{\theta}=(I-\mathcal{H})(\zeta-\bar{\zeta})$, it's easy to obtain
\begin{equation}
    \norm{|D|^{1/2}\frac{1}{2}(I+\mathbb{H})\tilde{\theta}}_{H^{s}}\leq C\epsilon^2.
\end{equation}
Then we have 
\begin{equation}
    \Big|\norm{|D|^{1/2}\tilde{\theta}}_{H^s}^2-\norm{|D|^{1/2}\frac{1}{2}(I-\mathbb{H})\tilde{\theta}}_{H^s}^2\Big|\leq C\epsilon^3.
\end{equation}
Note that 
\begin{equation}
    \norm{|D|^{1/2}\frac{1}{2}(I-\mathbb{H})\tilde{\theta}}_{H^s}^2=i\sum_{k=0}^s \int \partial_{\alpha}^k\frac{1}{2}(I-\mathbb{H})\tilde{\theta}\overline{\partial_{\alpha}^{k+1}\frac{1}{2}(I-\mathbb{H})\tilde{\theta}}d\alpha
\end{equation}
Use the fact that 
\begin{equation}
    (I-\mathbb{H})\tilde{\theta}=2\tilde{\theta}+(\mathcal{H}-\mathbb{H})\tilde{\theta},
\end{equation}
and use
\begin{equation}
    \norm{\partial_{\alpha}^k |D|^{1/2}(\mathcal{H}-\mathbb{H})\tilde{\theta}}_{L^2}\leq C\epsilon^2,
\end{equation}
we obtain
\begin{equation}\label{potential1}
   \Big| \int i\theta_k \overline{\partial_{\alpha}\theta_k}d\alpha-4\|\partial_{\alpha}^k|D|^{1/2}\tilde{\theta}\|_{L^2}^2 \Big|\leq C\epsilon^3.
\end{equation}
Similarly,
\begin{equation}\label{potential2}
   \Big| \int i\sigma_k \overline{\partial_{\alpha}\sigma_k}d\alpha-4\norm{\partial_{\alpha}^k|D|^{1/2}\tilde{\sigma}}_{L^2}^2 \Big|\leq C\epsilon^3.
\end{equation}
By (\ref{almostone}), (\ref{almosttwo}), (\ref{potential1}), and (\ref{potential2}),  we obtain
\begin{equation}
    \Big|\mathcal{E}_s-4\sum_{k=0}^s \Big\{\norm{\partial_{\alpha}^kD_t\tilde{\theta}}_{L^2}^2+\norm{\partial_{\alpha}^k D_t\tilde{\sigma}}_{L^2}^2 +\norm{\partial_{\alpha}^k|D|^{1/2}\tilde{\theta}}_{L^2}^2+\norm{\partial_{\alpha}^k|D|^{1/2}\tilde{\sigma}}_{L^2}^2\Big\}\Big|\leq C\epsilon^3.
\end{equation}
\end{proof}
\begin{corollary}
Assume the assumptions of Theorem \ref{longtime}, then 
\begin{equation}
    \mathcal{E}_s(0)\leq 17\epsilon^2.
\end{equation}
\end{corollary}

\begin{proposition}\label{closedenergyboot}
Assume the assumptions of Theorem \ref{longtime}, there exists $\delta>0$ such that 
\begin{equation}
\|\zeta_{\alpha}-1\|_{H^s}\leq 5\epsilon,\quad \quad  \|\mathfrak{F}\|_{H^{s+1/2}}\leq 5\epsilon,\quad \quad \|D_t\mathfrak{F}\|_{H^s}\leq 5\epsilon\quad \quad t\in [0,\delta \epsilon^{-2}]
\end{equation}
Indeed, we can choose $\delta$ to be an absolute constant.
\end{proposition}
\begin{proof}
Let $\delta>0$ to be determined. Let 
\begin{equation}
\mathcal{T}:=\Big\{T\in [0,\delta \epsilon^{-2}]:\quad \|\zeta_{\alpha}-1\|_{H^s}\leq 5\epsilon,~~\|\mathfrak{F}\|_{H^{s+1/2}}\leq 5\epsilon, ~~\|D_t\mathfrak{F}\|{H^s}\leq 5\epsilon,\quad \forall ~t\in [0,T]  \Big\}
\end{equation}
At $t=0$, we have 
\begin{equation}
\|\mathfrak{F}\|_{H^{s+1/2}}+\|D_t\mathfrak{F}\|{H^s}\leq \frac{3}{2}\epsilon.
\end{equation}
To obtain estimate of $\|\zeta_{\alpha}-1\|_{H^s}$, use $D_t^2\zeta-iA\zeta_{\alpha}=-i$, we have 
\begin{equation}
    \zeta_{\alpha}-1=\frac{D_t^2\zeta-i(A-1)}{iA}.
\end{equation}
We have $D_t^2\zeta=D_t\bar{\mathfrak{F}}+D_t\bar{q}$, and
$$D_tq=\sum_{j=1}^2 \frac{\lambda_j i}{2\pi}\frac{D_t\zeta-\dot{z}_j}{(\zeta(\alpha,t)-z_j(t))^2}.$$
We have 
\begin{equation}
    \|D_tq\|_{H^s}\leq C\epsilon^2+K_s^{-1}\epsilon.
\end{equation}
Use (\ref{estimateAminusOne}), we obtain
\begin{equation}
    \|\zeta_{\alpha}(\cdot, 0)-1\|_{H^s}\leq \|D_t\mathfrak{F}(\cdot,0)\|_{H^s}+C\epsilon^2+K_s^{-1}\epsilon\leq 2\epsilon.
\end{equation}
Therefore,  $0\in \mathcal{T}$, so $\mathcal{T}\neq \emptyset$. Since $\|\zeta_{\alpha}-1\|_{H^s}, \|\mathfrak{F}\|_{H^{s+1/2}}, \|D_t\mathfrak{F}\|_{H^s}$ are continuous in $t$,  we have $\mathcal{T}$ is closed. To prove $\mathcal{T}=[0, \delta\epsilon^{-2}]$, it suffices to prove that if $T_0<\delta\epsilon^{-2}$, then there exists $c>0$ such that $[0, T_0+c)\subset \mathcal{T}$.

\noindent  Assume $T_0\in \mathcal{T}$ and assume $T_0<\delta \epsilon^{-2}$. By Proposition \ref{energy_estimate}, we have for any $t_0\leq T_0$,
\begin{align*}
\mathcal{E}_s(t_0)=&\mathcal{E}_s(0)+\int_0^{t_0} \frac{d}{dt}\mathcal{E}_s(t)dt\\
\leq &17\epsilon^2+\int_0^{t_0} (C\epsilon^4+K_s^{-1}\epsilon^3 d_I(t)^{-3/2}+K_s^{-1} \epsilon^2\frac{|\lambda|}{x(0)}d_I(t)^{-5/2})dt\\
\leq &17\epsilon^2+C\epsilon^4 T_0+K_s^{-1}\epsilon^3\int_0^{t_0} ((1+\frac{|\lambda|}{20\pi x(0)} t)^{-1})^{3/2}dt\\
&+K_s^{-1}\epsilon^2\frac{|\lambda|}{x(0)}\int_0^{t_0}((1+\frac{|\lambda|}{20\pi x(0)} t)^{-1})^{3/2}dt\\
\leq & 17\epsilon^2+C\epsilon^4T_0+K_s^{-1}\epsilon^3\frac{x(0)}{|\lambda|}+K_s^{-1}\epsilon^2.
\end{align*}
Since $\frac{|\lambda|}{x(0)}\geq M\epsilon$, we have 
$$K_s^{-1}\epsilon^3\frac{x(0)}{|\lambda|}\leq K_s^{-1}\epsilon^3(M)^{-1}\epsilon^{-1}=K_s^{-1}M^{-1}\epsilon^2\leq \frac{1}{2}\epsilon^2. $$

Since $T_0\leq \delta\epsilon^{-2}$, if we choose $\delta\leq C^{-1}$, then 
$$C\epsilon^4T_0\leq \epsilon^2.$$
Therefore we have $$\sup_{t\in [0,T_0]}\mathcal{E}_s(t)\leq 19\epsilon^2.$$  By lemma \ref{connect}, we obtain 
\begin{equation}
    4\sum_{k=0}^s \Big\{\|\partial_{\alpha}^kD_t\tilde{\theta}\|^2+\|\partial_{\alpha}^k D_t\tilde{\sigma}\|^2 +\|\partial_{\alpha}^k|D|^{1/2}\tilde{\theta}\|_{L^2}^2+\|\partial_{\alpha}^k|D|^{1/2}\tilde{\sigma}\|_{L^2}^2\Big\}\leq \mathcal{E}_s+C\epsilon^3\leq 20\epsilon^2.
\end{equation}
So we have 
\begin{equation}\label{selfcontrolbootstrap}
    \|D_t\tilde{\theta}\|_{H^{s+1/2}}+\|D_t\tilde{\sigma}\|_{H^s}+\| |D|^{1/2}\tilde{\theta}\|_{H^s}\leq 5\epsilon,\quad \quad
\end{equation}
By lemma \ref{comparetransform}, we obtain
\begin{equation}
    \|\mathfrak{F}\|_{H^{s+1/2}}\leq K_s^{-1}\epsilon+\frac{1}{2}\|D_t\tilde{\theta}\|_{H^{s+1/2}}\leq 3\epsilon.
\end{equation}
\begin{equation}
\|D_t\mathfrak{F}\|_{H^s}\leq \frac{1}{4}\|D_t\tilde{\sigma}\|_{H^s}+K_s^{-1}\epsilon\leq 2\epsilon.
\end{equation}
Since $\bar{\zeta}-\alpha$ is holomorphic, we have 
\begin{equation}
   \tilde{\theta}= (I-\mathcal{H})(\bar{\zeta}-\zeta)=(I-\mathcal{H})(\zeta-\alpha)=2(\zeta-\alpha)-(\mathcal{H}+\bar{\mathcal{H}})(\zeta-\alpha).
\end{equation}
It's easy to obtain
\begin{equation}
    \norm{|D|^{1/2}(\Big(\tilde{\theta}-2(\zeta-\alpha)\Big)}_{H^s}=\| |D|^{1/2} (\mathcal{H}+\bar{\mathcal{H}})(\zeta-\alpha)\|_{H^s}\leq C\epsilon^2.
\end{equation}
By (\ref{selfcontrolbootstrap}), we obtain
\begin{equation}
    2\||D|^{1/2}(\zeta-\alpha)\|_{H^s}\leq \||D|^{1/2}\tilde{\theta}\|_{H^s}+C\epsilon^2\leq 6\epsilon.
\end{equation}
So we have 
\begin{equation}
    \norm{|D|^{1/2}(\zeta-\alpha)}_{H^s}\leq 3\epsilon.
\end{equation}
To obtain control of $\|\zeta_{\alpha}-1\|_{H^s}$, again we use
\begin{equation}
    \zeta_{\alpha}-1=\frac{D_t^2\zeta-i(A-1)}{iA}.
\end{equation}
It's easy to obtain
\begin{equation}
    \|\zeta_{\alpha}-1\|_{H^s}\leq \|D_t\mathfrak{F}\|_{H^s}+C\epsilon^2+K_s^{-1}\epsilon\leq 3\epsilon.
\end{equation}
By continuity, we can choose $c>0$ sufficiently small such that 
\begin{equation}
    \|\zeta_{\alpha}-1\|_{H^s}\leq 5\epsilon,~~\|\mathfrak{F}\|_{H^{s+1/2}}\leq 5\epsilon, ~~\|D_t\mathfrak{F}\|{H^s}\leq 5\epsilon,\quad \forall ~t\in [0,T_0+c)
\end{equation}
So we must have $\mathcal{T}=[0,\delta\epsilon^{-2}]$, for some absolute constant $\delta>0$.
\end{proof}

\subsection{Change of variables back to lagrangian coordinates} Next, we need to change of variables back to system (\ref{vortex_boundary}). So we need to control $\kappa$ on time interval $[0, \delta\epsilon^{-2}]$. We have 
\begin{equation}
\kappa_t=b\circ\kappa
\end{equation}
So we have 
\begin{equation}
\partial_t\kappa_{\alpha}=b_{\alpha}\circ\kappa \kappa_{\alpha}.
\end{equation}
Recall that 
\begin{equation}
(I-\mathcal{H})b= -[D_t\zeta, \mathcal{H}]\frac{\bar{\zeta}_{\alpha}-1}{\zeta_{\alpha}}-\frac{i}{\pi} \sum_{j=1}^2 \frac{\lambda_j}{\zeta(\alpha,t)-z_j(t)}.
\end{equation}
So we have 
\begin{equation}
\begin{split}
(I-\mathcal{H})b_{\alpha}=[\zeta_{\alpha}, \mathcal{H}]b-\partial_{\alpha}[D_t\zeta, \mathcal{H}]\frac{\bar{\zeta}_{\alpha}-1}{\zeta_{\alpha}}+\frac{i}{\pi}\sum_{j=1}^2 \frac{\lambda_j \zeta_{\alpha}}{(\zeta(\alpha,t)-z_j(t))^2}.
\end{split}
\end{equation}
Clearly, 
\begin{equation}
\norm{[\zeta_{\alpha}, \mathcal{H}]b-\partial_{\alpha}[D_t\zeta, \mathcal{H}]\frac{\bar{\zeta}_{\alpha}-1}{\zeta_{\alpha}}}_{H^1}\leq C\epsilon^2,
\end{equation}
and
\begin{equation}
\norm{\frac{i}{\pi}\sum_{j=1}^2 \frac{\lambda_j \zeta_{\alpha}}{(\zeta(\alpha,t)-z_j(t))^2}}_{H^1}\leq K_s^{-1}d_I(t)^{-5/2}\epsilon,
\end{equation}
for some absolute constant $C>0$.  By lemma \ref{realinverse}, we have 
\begin{equation}
\norm{b_{\alpha}}_{H^1}\leq C\epsilon^2+K_s^{-1}d_I(t)^{-5/2}\epsilon.
\end{equation}
By Sobolev embedding, we have 
\begin{equation}
\norm{b_{\alpha}\circ\kappa}_{\infty}=\norm{b_{\alpha}}_{\infty}\leq \norm{b_{\alpha}}_{H^1}\leq C\epsilon^2+K_s^{-1}d_I(t)^{-5/2}\epsilon.
\end{equation}
It's easy  to obtain that
\begin{equation}
\norm{\kappa_{\alpha}(\cdot, 0)-1}_{\infty}\leq C\epsilon.
\end{equation}
So we obtain
\begin{align}
\kappa_{\alpha}(\alpha,t)-\kappa_{\alpha}(\alpha,0)=&\int_0^t \kappa_{\alpha \tau}(\alpha,\tau)d\tau\\
=& \int_0^t b_{\alpha}\circ\kappa(\alpha,\tau) \kappa_{\alpha}(\alpha,\tau)d\tau.
\end{align}
Let $\delta_1>0$ be a constant to be determined. 
\begin{equation}
\mathcal{T}_1:=\Big\{ T\in [0, \delta_1\epsilon^{-2}]:  \sup_{t\in [0,T]}\norm{\kappa_{\alpha}(\cdot,t)-\kappa_{\alpha}(\cdot, 0)}_{\infty}\leq \frac{1}{10}\Big\}
\end{equation}
In particular, if $t\in \mathcal{T}_1$, then for $\epsilon$ sufficiently small, we have $\frac{4}{5}\leq \kappa_{\alpha}\leq \frac{6}{5}$. Also, $\mathcal{T}_1$  is closed. For $T\in \mathcal{T}_1$, we have for any $t\in [0,T]$, 
\begin{align}
\Big| \kappa_{\alpha}(\alpha,t)-\kappa_{\alpha}(\alpha,0)\Big|\leq &\int_0^t \Big(C\epsilon^2+K_s^{-1}d_I(\tau)^{-5/2}\epsilon\Big) d\tau\\
\leq & \int_0^t (C\epsilon^2+K_s^{-1}(1+\frac{|\lambda|}{20\pi x(0)})^{-5/2}\epsilon) d\tau\\
\leq & C\epsilon^2 t+K_s^{-1}\frac{20\pi x(0)}{|\lambda|}\frac{2}{3}\epsilon\\
\leq & C\epsilon^2 t+\frac{1}{15K_s}.
\end{align}
Here we've used the assumption $\frac{|\lambda|}{x(0)}\geq 200\pi \epsilon$.  Choose $\delta_1=\frac{1}{30C}$. Then we have 
\begin{equation}
\sup_{t\in [0,T]} \norm{\kappa_{\alpha}(\cdot,t)-\kappa_{\alpha}(\cdot,0)}_{\infty}\leq \frac{1}{20}.
\end{equation}
Therefore, $\mathcal{T}_1$ is open in $[0,\delta_1\epsilon^{-2}]$,  we must have $\mathcal{T}_1=[0,\delta_1\epsilon^{-2}]$.

\vspace*{1ex}

\noindent Let $\delta_0:=\min\{\delta, \delta_1\}$. Since $\kappa_{\alpha}\geq \frac{3}{5}$ on $[0, \delta_0\epsilon^{-2}]$,  we can change of variables back to lagrangian coordinates and conclude the proof of Theorem \ref{longtime}.

\section*{Acknowledgement}
The author would like to thank his Ph.D advisor, Prof. Sijue Wu, for introducing him this topic, for many helpful discussions and  invaluable comments on this problem, and for carefully reading the drafts of this paper.   The author also would like to thank Prof. Robert Krasny, Prof. Charles Doering, Prof. Stefan Llewellyn Smith, and Prof. Christopher Curtis for providing references on water waves with point vortices. This work is partially supported by NSF grant DMS-1361791.

\bibliography{qingtangbib}{}
\bibliographystyle{plain}

\end{document}